\documentclass[11pt,letterpaper]{amsart}
\usepackage{graphicx} % Required for inserting images
\input{style.sty}
\title{The equivalence between two real Seiberg-Witten Floer homologies}
\author{Yonghan Xiao}
\address{Department of Mathematics, School of Mathematical Sciences\\Peking University}
\email{judy\_xyh0530@stu.pku.edu.cn}
\usepackage{fullpage, graphicx, enumitem}
\usepackage{url}
\usepackage{verbatim}
\usepackage{bbm}
\usepackage{amssymb,amsmath}
\usepackage{amsfonts,savesym,graphicx,bm,amsthm}
\usepackage{color}
\usepackage[mathscr]{eucal}
\usepackage[all]{xy}

\date{\today}

\begin{document}

\maketitle
\begin{abstract} 
We show that for a real rational homology sphere $(Y,\iota)$ equipped with a real $\mathrm{spin^c}$ structure $\s$, the real monopole Floer homology defined by Li and the real Seiberg-Witten Floer homology defined by Konno, Miyazawa and Taniguchi are isomorphic. As corollaries, we identify some Fr\o yshov-type invariants and prove two Smith-type inequalities.
    
\end{abstract}
\tableofcontents

\section{Introduction}
Let $Y$ be a three-manifold with a $\mathrm{spin^c}$ structure $\underline{\s}$. There are two different ways of defining Seiberg-Witten Floer homology for $(Y,\underline{\s})$. The first one is the monopole Floer homology $\mathit{HM}^\circ$ constructed by Kronheimer and Mrowka in \cite{Kronheimer_Mrowka_2007}. They used formal gradient on the configuration space to perturb the Seiberg-Witten map in order to achieve transversality and then took half-infinite dimensional Morse homology of the perturbed vector field. This approach works generally for all three-manifolds. When specialized to rational homology spheres, Manolescu used finite-dimensional approximation to produce the Seiberg-Witten Floer stable homotopy type $\mathit{SWF}(Y,\underline{\s})$, which is an $S^1$-equivariant suspension spectrum. Taking its Borel homology leads to another version of Seiberg-Witten Floer homology. In \cite{Lidman2016TheEO}, the two notions were identified in the sense that there is an isomorphism \[\widecheck{\mathit{HM}}_{*}(Y,\underline{\s})\cong \widetilde{H}^{S^1}_{*}(\mathit{SWF}(Y,\underline{\s});\Z),\] as absolutely graded $\widetilde{H}^*_{S^1}(S^0;\Z)\cong\Z[U]$-modules and its two counterparts.

Now we add more symmetry to this picture. Let $Y$ be a three-manifold with a real structure $\iota$. Here, a real structure on $Y$ is an orientation preserving involution on $Y$ with a \emph{non-empty} one-dimensional fixed set. A compatible real $\mathrm{spin^c}$ structure $\s$ is a $\mathrm{spin^c}$ structure $\underline{\s}$ together with a lift $\tau$ of $\iota$ to the spinor bundle. Following Kronheimer and Mrowka's approach, Li introduced real monopole Floer homology $\mathit{HMR}^{\circ}$ in \cite{li2022monopolefloerhomologyreal}. On the other hand, when $Y$ is a rational homology sphere, Konno, Miyazawa and Taniguchi defined the real Seiberg-Witten spectrum $\mathit{SWF}_{\Z_2}(Y,\iota,\s)$ in \cite{Konno2024}, following Manolescu's idea. They conjectured that these two real theories also coincide on rational homology spheres. In this paper, we prove this conjecture and also obtain 
\cite[Conjecture 1.4]{li2024realmonopolesspectralsequence} as a corollary.

\begin{thm}(\cite[Conjecture 1.27]{Konno2024}) \label{thm:main theorem}
Let $Y$ be a rational homology three sphere with a real structure $\iota$ and a compatible real $\mathrm{spin^c}$ structure $\s$. Then we have an isomorphism of relatively graded $\widetilde{H}_{\Z_2}^*(S^0;\F)\cong \F[v]$-modules \[\widecheck{\mathit{HMR}}_{*}(Y,\iota,\s)\cong \widetilde{H}^{\Z_2}_{*}(\mathit{SWF}_{\Z_2}(Y,\iota,\s);\F).\] Here, $\widecheck{\mathit{HMR}}_{*}(Y,\iota,\s)$ is the ``to'' version of real monopole Floer homology defined in \cite{li2022monopolefloerhomologyreal}  and $\mathit{SWF}_{\Z_2}(Y,\iota,\s)$ is the real Seiberg-Witten Floer homotopy type defined in \cite{Konno2024}. The isomorphism respects absolute grading when a well-defined absolute grading exists on $\widecheck{\mathit{HMR}}_{*}(Y,\iota,\s)$ (See Subsection \ref{subsub:Gradings}).
\end{thm}

Similarly, we have isomorphisms:\[\widehat{\mathit{HMR}}_{*}(Y,\iota,\s)\cong c\widetilde{H}_{*}(\mathit{SWF}_{\Z_2}(Y,\iota,\s);\F),\] \[\overline{\mathit{HMR}}_{*}(Y,\iota,\s)\cong t\widetilde{H}^{\Z_2}_{*}(\mathit{SWF}_{\Z_2}(Y,\iota,\s);\F).\]
where $c\widetilde{H}_{*}(\mathit{SWF}_{\Z_2}(Y,\iota,\s);\F)$ and $t\widetilde{H}^{\Z_2}_{*}(\mathit{SWF}_{\Z_2}(Y,\iota,\s);\F)$ are the coBorel and Tate homology of $\mathit{SWF}_{\Z_2}(Y,\iota,\s)$, respectively.
The main idea of the proof is to restrict the constructions in \cite{Lidman2016TheEO} to the invariant part, but we have to be careful when doing this. A key difference is that now the constant gauge group is a discrete group $\Z_2$, so the interpolation argument no longer works, and we have no tangent to the orbit after moving into the global Coulomb slice. For the former, we need some algebraic topology argument when we identify the grading and module structure (see Subsection \ref{sub: Gradings and U-action}). For the latter, we need to modify the definition of extended Hessians and some other notions on the path spaces. Also, we choose to use a real cylinder function to perturb Seiberg-Witten map equivariantly, so that we can use their result on `very compactness' without proof. But this means we may not achieve the necessary transversality in the non-invariant part of configuration space, so we cannot simply restrict all their constructions to the invariant part. Instead, we must adapt the construction of functions for the quasi-gradient and the self-diffeomorphism to the real case. (For the reason of our choice, see Subsection \ref{sub:Morse quasi-gradient flow and Morse Smale condition}.)

As noted above, the main difference between two approaches is that one introduces perturbation to the Seiberg-Witten equation, while the other one restricts to the global Coulomb slice (See Subsection \ref{subsub: Configuration space and Coulomb slices} for its definition) and takes a finite dimensional cut-off. We relate them by first reformulating $\widecheck{\mathit{HMR}}$ in the global Coulomb slice and generalizing the definition of $\mathit{SWF}_{\Z_2}$ to perturbed Seiberg-Witten map. Then, we construct an intermediate chain complex depending on the parameter $\lambda$ of cut-off, which is the Morse complex of an equivariant Morse quasi-gradient on a ball in the cut-off global Coulomb slice, and show that it is isomorphic to both $\widecheck{\mathit{HMR}}_*$ and $\widetilde{H}^{\Z_2}_*(\mathit{SWF}_{\Z_2})$ in a grading range determined by $\lambda$. Finally, we show that the grading range tends to $\infty$ as $\lambda$ does. Along this way, we also identify the module structure and grading on various homology groups.

\begin{remark}\label{rmk:generalize to nonhomology sphere}
It was pointed out by Manolescu that the argument in this paper works not only for real rational homology spheres, but also for any real three-manifold $(Y,\iota)$ satisfying that $H^1(Y;\Z)^{-\iota^*}=0$. Miyazawa showed in \cite{miyazawa2023gaugetheoreticinvariantembedded} that $\mathit{SWF}_{\Z_2}(Y,\iota,\s)$ is indeed well-defined for such a real manifold equipped with a real $\mathrm{spin^c}$ structure. This idea was also used in \cite{miyazawa2025satelliteformularealseibergwitten}. We will show how to fit our proof into this case after we finish the proof of our main theorem in Subsection \ref{sub:Main thm}.  
\end{remark}

As a corollary, we will prove the following proposition.
\begin{prop}\label{prop:Froyshov invariant identification}
    Let $Y$ be the double branched cover of $S^3$ over some link $K$ with $\mathrm{det}(K)\ne 0$ equipped with the canonical real structure $\iota$ and a compatible real $\mathrm{spin^c}$ structure. Then $d(Y,\iota,\s)=-h_R(K,\s)$, in which $h_R$ and $d$ are the Fr\o yshov-type invariants defined in \cite{Konno2024} and \cite{li2022triangle}, respectively.
\end{prop}

Based on the argument in \cite{LM2016coveringspaces}, we have two interesting Smith-type inequalities and some observations on L-spaces.

\begin{thm}\label{thm:Smith inequality}
Let $Y$ be a rational homology sphere with a real structure $\iota$ and a compatible real $\mathrm{spin^c}$ structure $\s$. Then we have the following inequality: \[\mathrm{dim} \widetilde{\mathit{HMR}}(Y,\iota,\s)\le \mathrm{dim} \widetilde{HM}(Y,\underline{\s}).\] 
Here, $\widetilde{\mathit{HM}}$ is the ``tilde version'' of monopole Floer homology introduced by Bloom in \cite{BLOOM20113216} and $\widetilde{\mathit{HMR}}$ is its real counterpart defined by Li in \cite{li2024realmonopolesspectralsequence}, both are considered over $\F$. In particular, when $Y$ is an L-space, $(Y,\iota)$ is a real L-space for any real structure $\iota$ on Y. 
\end{thm}

\begin{thm}\label{thm:inequality in reduced version}
    Under the same assumption as Theorem \ref{thm:Smith inequality}, we also have \[\mathrm{dim} \mathit{HMR}_{\mathrm{red}}(Y,\iota,\s)\le 2 \mathrm{dim} \mathit{HM}_{\mathrm{red}}(Y,\underline{\s}).\]
    Here, $\mathit{HMR}_{\mathrm{red}}(Y,\iota,\s)$ and $\mathit{HM}_{\mathrm{red}}(Y,\underline{\s})$ are reduced versions of the real and the usual monopole Floer homology model on $\mathit{HF}_{\mathrm{red}}(Y,\underline{\s})$. We will define them in Subsection \ref{sub: Smith type inequalities}.
\end{thm}

The idea of these two theorems originates from \cite{LM2016coveringspaces}. However, what they compared were the Floer homologies of a manifold and its covering space, while we are comparing the usual Floer homology and its real counterpart associated to the same manifold. Our proof is more involved since we need to deal with the difference between $S^1$ and $\Z_2$ equivariant theories.

\begin{remark}
    The implication for $L$-space has its real Heegaard Floer analogue conjectured in \cite{guth2025realheegaardfloerhomology} and proved in \cite{hendricks2025noterealheegaardfloer}. Hendricks's approach is roughly a localization spectral sequence from $\widehat{\mathit{HF}}$ to $\widehat{\mathit{HFR}}$, which is quite different from our approach, so it would be interesting to compare these. See Subsection \ref{sub: Smith type inequalities} for details. We also make the following conjecture, which is a real analogue of the main result from \cite{HF=HM1} and its sequel papers.
\end{remark}
\begin{conjecture}
    For a real rational homology three-sphere, or more generally any real three-manifold, we have the following isomorphisms \[\widetilde{\mathit{HMR}}(Y,\iota,\s)\cong \widehat{\mathit{HFR}}(Y,\iota,\s),\text{ } \widecheck{\mathit{HMR}}(Y,\iota,\s)\cong \mathit{HFR}^+(Y,\iota,\s),\]
    \[\widehat{\mathit{HMR}}(Y,\iota,\s)\cong \mathit{HFR}^-(Y,\iota,\s),\text{ } \overline{\mathit{HMR}}(Y,\iota,\s)\cong \mathit{HFR}^\infty(Y,\iota,\s).\]
    as graded $\F_2$-vector spaces or modules over a suitable ring.
\end{conjecture}
\subsection{Convention}
Throughout this paper, we use $\F=\F_2$ as the coefficient for (co)homology theories unless otherwise stated. To make things clear, we use $\Z_2$ for the group and $\F$ for the field, though they are essentially the same. For a real three-manifold, we setup the involution to be orientation-preserving with non-empty one-dimensional fixed set. For a real $\mathrm{spin^c}$ structure $\s$, we denote its underlying $\mathrm{spin^c}$ structure by $\underline{\s}$. We also abuse $\underline{\s}$ for an arbitrary $\mathrm{spin^c}$ structure which does not necessarily support a real structure. For various notations from gauge theory, we adopt notations from \cite{Lidman2016TheEO} as much as possible and add superscripts to distinguish the real analogue with its original construction. This is different from the convention in \cite{li2022monopolefloerhomologyreal}, which used the simplest notations for real configuration spaces and added underlines to the original spaces.

\subsection{Organization}
In Section \ref{sec:review}, we review some important notions that we shall use regarding the equivariant Morse quasi-gradient and the definitions of real Seiberg-Witten Floer homotopy type and real monopole Floer homology. Next, in Section \ref{sec:Real Monopole Floer homology in global Coulomb slice}, we reconstruct real monopole Floer homology in the global Coulomb slice and show that it is equivalent to the original version. In Section \ref{sec:Relating finite and infinite dimensional Morse homologies}, we first add a perturbation to the definition of homotopy type and construct the intermediate chain complex as we mentioned in the outline of the proof strategy. Then, we show convergence of the approximate stationary points and trajectories using the inverse function theorem. Finally, after some preparation in Subsection \ref{sub:Stabilit} and \ref{sub:U-action}, we prove our main theorem in Subsection \ref{sub:Main thm}. In the last two subsections, we apply our main theorem to prove the proposition on Fr\o yshov-type invariants and those Smith-type inequalities.
\begin{ack}
The author expresses gratitude to Jiakai Li, Masaki Taniguchi, Boyu Zhang for their helpful discussion. She thanks Jianfeng Lin and Ciprian Manolescu for their comments on an early draft and for their suggestions on generalizations and applications. She is also grateful for correction and advice from anonymous referees.
    
\end{ack}
\section{Review}\label{sec:review}
\subsection{Morse homology and Morse quasi-gradient flow}\label{sub:Morse homology and Morse quasi gradient flow}

In this subsection, we give some preliminary definitions and propositions about equivariant Morse homology and Morse quasi-gradient flow. For basic Morse homology, one can refer to \cite[Section 2.1-2.5]{Lidman2016TheEO} for a brief review. 

We consider equivariant Morse homology of a smooth manifold $X$ (without boundary) with a $\Z_2$-action. In the finite-dimensional setting, we will use $Q$ to denote the fixed point set. We can blow up $X$ to \[X^{\sigma}=(X-Q)\cup (N^1(Q)\times[0,\epsilon)),\] where the gluing is formed by identifying $N(Q)-Q$ with $N^1(Q)\times[0,\epsilon)$, for $N(Q)$ the normal bundle of $Q$ in $X$ and $N^1(Q)$ the unit normal bundle. Now, $X^{\sigma}$ has a free $\Z_2$-action, so that we can take the quotient $X^{\sigma}/\Z_2$ which is a smooth manifold with boundary $N^1(Q)/\Z_2$.

Similar to the $S^1$ case stated in \cite[Section 2.6]{Lidman2016TheEO}, a $\Z_2$-equivariant vector field $\widetilde{v}$ on $X$ induces smooth vector fields $v$ on $(X-Q)/\Z_2$ and $v^{\sigma}$ on $X^\sigma/\Z_2$, that are tangent to the boundary.

A point on the boundary can be written as $(q,[\phi])$ for $q\in Q$ and $\phi\in N^1_q(Q)$. The tangent space of $X^\sigma/\Z_2$ at this point decomposes as $T_q Q\oplus \left \langle\phi \right \rangle^{\perp}\oplus\R$, where $\left \langle\phi \right \rangle^{\perp}$ is the orthogonal complement of $\phi$ in $N_qQ$ and  $\R$ is the normal direction to the boundary.

The covariant derivative $(\nabla\tilde{v})_{q}$ is $\Z_2$-equivariant, so it takes the normal part $N_q Q$ to itself. Let $L_q$ be $(\nabla\tilde{v})|_{N_q Q}$. Using the decomposition of tangent space at $(q,[\phi])$ above, \[v^{\sigma}(q,[\phi])=(\widetilde{v}(q),\mathbb{L}_q\phi,0).\] When $\widetilde{v}(q)=0$, $\mathbb{L}_q\phi=L_q\phi-\mathrm{Re}\left \langle \phi,L_q\phi\right \rangle\phi$. Thus, stationary points of $v^\sigma$ on the boundary are those $(q,[\phi])$ for which $\widetilde{v}(q)=0$ and $\phi$ is an eigenvector of $L_q$.

\begin{definition}\label{defi:Z_2 equivariant Morse quasi-gradient}
    A smooth $\Z_2$-equivariant vector field $\widetilde{v}$ is a \emph{Morse equivariant quasi-gradient} if \begin{enumerate}
        \item All stationary points of $v$ on $(X-Q)/\Z_2$ are hyperbolic.
        \item All stationary points of $\widetilde{v}|_{Q}$ are hyperbolic.
        \item At each stationary point $q$ of $\widetilde{v}|_{Q}$, the operator $L_q:N_q\to N_q$ is self-adjoint with a simple spectrum away from zero. 
        \item There exists a smooth $\Z_2$-equivariant function $f:X\to \R$ such that $df(\widetilde{v})\ge 0$ for all $x\in X$ and equality holds iff $\tilde{v}=0$.
    \end{enumerate}
\end{definition}
Here, we remind readers that an operator is called \emph{hyperbolic} if its complexification has no purely imaginary eigenvalue and its spectrum is called \emph{simple} if each eigenvalue has multiplicity one.
\begin{lem}\label{lem:index of reducible in f.d. for Morse quasi gradient}
Conditions (1)-(3) in the above definition are equivalent to the requirement that all the stationary points of $v^\sigma$ on $X^\sigma/\Z_2$ are hyperbolic and the operator $L_q$ for any boundary stationary point $q$ is self-adjoint with a simple spectrum away from zero. Fix such a $q$ and label eigenvalues of $L_q$ by $\lambda_1(q)<\ldots<\lambda_{n}(q)$. Let $(q,[\phi_i(q)])$ be the corresponding stationary point of $v^\sigma$, then \[\mathrm{ind}(q,[\phi_i(q)])=\begin{cases}
    \mathrm{ind}_{Q}(q)+i-1 \text{  if } \lambda_i(q)>0\\
    \mathrm{ind}_{Q}(q)+i \text{  if } \lambda_i(q)<0\\
\end{cases}\]    
\end{lem}
Below are some properties of $\Z_2$-equivariant Morse quasi-gradient analogous to \cite[Lemma 2.6.4-2.6.6]{Lidman2016TheEO}.
\begin{lem} \label{lem:properties of equivariant Morse quasi gradient}
Let $\tilde{v}$ be a $\Z_2$-equivariant Morse quasi-gradient on $X$.
\begin{itemize}
    \item If $f$ is the function from (4), then any stationary point of $\tilde{v}$ is a critical point of $f$.
    \item If $\gamma:\R\to X$ is a flow line of $\tilde{v}$, then $\lim_{t\to \pm\infty}[\gamma(t)]$ exist in $X/\Z_2$ and both of them are projections of stationary points of $\tilde{v}$.
    \item If $\gamma:\R\to X^\sigma/\Z_2$ is a flow line of $v^\sigma$, then $\lim_{t\to \pm\infty}\gamma(t)$ exist in $X^\sigma/\Z_2$ and both of them are stationary points of $v^\sigma$.
\end{itemize}
    
\end{lem}

Recall that a Morse quasi-gradient vector field is called \emph{Morse-Smale} if for any pair of stationary points, their stable and unstable manifolds intersect transversely. We now introduce a $\Z_2$-equivariant version.
\begin{definition}
    A $\Z_2$-equivariant Morse quasi-gradient is called \emph{Morse-Smale} if the induced vector field $v^\sigma$ on $X^\sigma/\Z_2$ satisfies the Morse-Smale condition for boundary-unobstructed trajectories; and the Morse-Smale condition in $\partial(X^\sigma/\Z_2)$ for boundary-obstructed trajectories.
\end{definition}

This assumption tells us that $v^\sigma$ is a usual Morse-Smale quasi-gradient on $X^\sigma/\Z_2$, so we can associate a Morse complex \[(\widecheck{C}(X^\sigma/\Z_2),\widecheck{\partial})\]to it, as in \cite[Section 2.5]{Lidman2016TheEO}. 

$X^\sigma/\Z_2$ can be regarded as an approximation of the homotopy quotient $X//\Z_2=X\times _{\Z_2} E\Z_2$. More precisely, when $n$ is the connectivity of $(X,X-Q)$, we have \[H_j(X^\sigma/\Z_2)\cong H_j(X//\Z_2),\] for $j\le n-1$ since $X-Q$ is $\Z_2$-equivariantly homotopy equivalent to $X^\sigma$.

The description of the $U$-action in \cite[Section 2.7]{Lidman2016TheEO} for $S^1$-equivariant Morse homology works for $\Z_2$-equivariant version. The only changes are \begin{itemize}
    \item $\Z[U]$ with $\mathrm{deg}U=2$ $\mapsto$ $\F[v]$ with $\mathrm{deg}v=1$;
    \item an $S^1$-principal bundle leads to a complex line bundle $\mapsto$ a $\Z_2$-principal bundle leads to a real line bundle.
\end{itemize}
This improves the isomorphism of groups \[H_j(X^\sigma/\Z_2)\cong H_j(X//\Z_2), \text{ } j\le n-1\] to one of the $\F[v]$-modules.

Based on remarks at the beginning of \cite[Section 2.8]{Lidman2016TheEO}, we combine the construction above with the Conley index theory. \cite[Section 2.4]{Lidman2016TheEO} briefly reviewed the Conley index theory, and \cite{Pruszko1999} provided more details. Let $\tilde{v}$ be a smooth vector field on $X$ and $\scrI$ be a $\Z_2$-invariant, isolated invariant set. If $M\subset X$ is a closed $\Z_2$-invariant subset of $X$, then we denote by $M^\sigma$ the closure of $p^{-1}(M-Q)$ in $X^\sigma$ where $p$ is the blow-down map. We seek a chain complex $\widecheck{C}(X^\sigma/\Z_2)[\scrI]$ using only trajectories in $\scrI^\sigma/\Z_2$. To achieve this, we need a $\Z_2$-invariant neighborhood $A$ of $\scrI$ so that $\widetilde{v}|_{A}$ is a Morse-Smale equivariant quasi-gradient.

Assuming the existence of such an $A$, we can obtain the desired Morse chain complex $\widecheck{C}(X^\sigma/\Z_2)[\scrI]$ using trajectories in $A^\sigma/\Z_2$.  Since $\scrI/\Z_2$ is an isolated invariant set for $v^\sigma$ in $X^\sigma/\Z_2$, we can form its Conley index $I(\scrI^\sigma/\Z_2)$  and we have \[H_*(\widecheck{C}(X^\sigma/\Z_2)[\scrI])\cong \widetilde{H}_*(I(\scrI^\sigma/\Z_2)).\] 

On the other hand, $\scrI$ is a $\Z_2$-invariant, isolated invariant set, so we can take its equivariant Conley index $I_{\Z_2}(\scrI)$. Let $I_{\Z_2}(\scrI)^{\Z_2}$ be its fixed point set. Then we have the following approximation of homology:\[H_{\le n-1}(\widecheck{C}(X^\sigma/\Z_2)[\scrI])\cong \widetilde{H}_{\le n-1}^{\Z_2}(I^{\Z_2}(\scrI)),\] where $n$ is the connectivity of $(I^{\Z_2}(\scrI),(I^{\Z_2}(\scrI)-I^{\Z_2}(\scrI)^{\Z_2})\cup *)$, with $*$  denoting the base point of $I^{\Z_2}(\scrI)$. This is an isomorphism between $\F[v]$-modules when we define the $v$-actions properly. This will be important in our proof of Theorem \ref{thm:main theorem}.
 
\subsection{Real Seiberg-Witten homotopy type}\label{sub:Real Seiberg-Witten homotopy type}
\subsubsection{Configuration space and Coulomb slices}\label{subsub: Configuration space and Coulomb slices}
We will focus on a real $\mathrm{spin^c}$ three-manifold $(Y,\iota,g,\s,\mathbb{S})$, where $Y$ is a rational homology sphere, $\iota:Y\to Y$ is an involution with a codimension two fixed set, $g$ is an invariant metric on $Y$, $\s$ is a real $\mathrm{spin^c}$ structure and $\bS$ is the spinor bundle associated to $\s$, on which we have an involution $\tau$ satisfying $\rho(\iota_*\xi)\tau(\phi_y)=\tau(\rho(\xi)\phi_y)$ for any $y\in Y$, any vector field $\xi$ on $Y$ and any spinor $\phi\in \Gamma(\bS)$. Here, $\rho$ denotes the Clifford multiplication on the spinor bundle.

We will be working with the configuration space \[\cC(Y)=\Omega^1(Y;i\R)\oplus \Gamma(\bS).\] The gauge transformation group $\cG(Y)=C^\infty(Y,S^1)$ acts on $\cC(Y)$ by $u\cdot (a,\phi)=(a-u^{-1}du,u\cdot \phi)$. Since $Y$ is a rational homology sphere, any gauge transformation $u$ can be written as $u=e^{f}$ for some $f:Y\to i\R$. The normalized gauge group $\cG^\circ(Y)$ consists of those $u\in \cG$ such that $u=e^{f}$ with $\int_{Y} f=0$. 

The real structure on $Y$ and a compatible real $\mathrm{spin^c}$ structure $(\s,\tau)$ give rise to involutions $I:\cC(Y) \to \cC(Y)$ defined by $I(a,\phi)=(-\iota^*a,\tau(\phi))$ and $I:\cG(Y)\to \cG(Y)$ defined by $I(u)(y)=\overline{u(\iota(y))}$. Here, we identify the space of $\mathrm{spin^c}$ connection on $\bS$ with $\Omega^1(Y;i\R)$ by choosing an invariant connection as the base one. A $\mathrm{spin^c}$ connection $A_0$ is \emph{invariant} if for any section $\phi$ of $\bS$, $\nabla_{\tau^*A}(\phi)=\tau(\nabla_A(\tau\phi))$. We will be interested in the fixed part of these involutions, i.e., we will consider $\cG^I$ acting on $\cC^I$. 

As usual, we can consider Sobolev completion of the configuration space and gauge group. We will add a subscript $k$ when the space is completed with respect to the $L^2_k$ norm. 

The gauge group $\cG^I(Y)$ splits as $\cG^{I,\circ}\times \cG^{I,h}$ where $\cG^{I,\circ}=(\cG(Y)^\circ)^I$ is contractible and $\cG^{I,h}$ consists of harmonic maps $u:Y\to S^1$ fixed by $I$. 
Recall from \cite[section 5.7]{li2022monopolefloerhomologyreal} that we have an exact sequence \[0\to \Z_2\to\pi_0(\cG^{I,h})\to H^1(Y;\Z)^{-\iota^*}\to 0.\] Since $Y$ is a rational homology sphere, the fourth term in this sequence is zero, so $\pi_0(\cG^I) =\pi_0(\cG^{I,h})=\Z_2$. For any $p\in \mathrm{fix}(\iota)$, the evaluation map $ev_{p}:\cG^I(Y)\to \Z_2$ is continuous, so it is constant on each component. Since we have the constant map $1$ in the identity component $\cG^I_{+}$ and constant $-1$ in the other component $\cG^I_{-}$,  we can conclude that for any $u_{\pm}\in \cG^I_{\pm}$ and any $p\in \mathrm{fixed}(\iota)$, we have $u_{\pm}(p)=\pm 1$. 

As noted in the proof of \cite[Lemma 2.6]{Konno2024}, for each $u_{\pm}\in \cG_{\pm}$, there is a \emph{real function} $f:Y\to i\R$ satisfying $u=\pm e^{if}$. A function $f:Y\to i\R$ is called \emph{real} if $f(y)=-f(\iota(y))$.

We have a \emph{global Coulomb slice} for the action of $\cG$ on $\cC$ \[W=\mathrm{ker} d^*\oplus \Gamma(\bS)\subset \cC(Y),\] where $d^*$ acts on the imaginary-valued 1-forms. Given any $(a,\phi)\in \cC(Y)$, there is a unique element of $W$ that lies in the same orbit of the normalized gauge group action. We will call it the \emph{global Coulomb projection} of $(a,\phi)$. Explicitly, we have \[\Pi^{gC}(a,\phi)=(a-df,e^f \phi),\] where $f:Y\to i\R$ is such that $d^*(a-df)=0$ and $\int_Y f=0$, i.e., $f=Gd^*a$, where $G$ is the Green function for $\Delta=dd^*$. 

When $I(a,\phi)=(a,\phi)$, $f=Gd^*a$ is a real function on $Y$ since we are using an invariant metric $g$. Thus, $e^f$ lies in $\cG^{I,\circ}$, and we can conclude that each orbit of real configurations has a unique representative in $W^I$. That is, $W^I$ is a global Coulomb slice for the action of $\cG^{I,\circ}$ on $\cC^I(Y)$, and $\Pi^{gC}$ restricts to a map $\cC^I(Y)\to W^I$. As in \cite{Lidman2016TheEO}, we denote the $L^2$ orthogonal projection $a\mapsto a-dGd^*a$ from $\Omega^1(Y;i\R)$ to $\mathrm{ker}d^*$ by $\pi$.

$\Pi^{gC}$ has derivative \[(\Pi_{*}^{gC})_{(a,\phi)}(b,\psi)=(b-dGd^*b,e^{Gd^*a}(\psi+(Gd^*b)\phi)).\] 

When $(a,\phi)$ is already in $W$, this simplifies to \[(\Pi_{*}^{gC})_{(a,\phi)}(b,\psi)=(b-d\xi,\psi+\xi\phi)=(\pi(b),\psi+\xi\phi)\in T_{(a,\phi)}W,\] where $\xi=Gd^*b$.

We also consider an infinitesimal slice for the gauge group action. More precisely, we define the \emph{local Coulomb slice} $\calK_{(a,\phi)}$, at $(a,\phi)\in \cC(Y)$ to consist of tangent vectors $(b,\psi)$ satisfying \[-d^*b+i\mathrm{Re}\left \langle i\phi,\psi\right \rangle=0.\] Let $\cT$ denote the tangent space of $\cC(Y)$; away from the \emph{reducibles} (those $(a,\phi)$ with $\phi=0$),  we have a direct sum decomposition \[\cT_k=\cJ_k\oplus\calK_k,\] where $\cJ_k$ consists of $(b,\psi)$ tangent to the $\cG_{k+1}$ orbit. We also have a \emph{local Coulomb slice projection} $\Pilc_{(a,\phi)}\colon \cT_{(a,\phi)}\to \calK_{(a,\phi)}$ defined by \[\Pilc_{(a,\phi)}(b,\psi)=(b-d\zeta,\psi+\zeta\phi).\]  Here, when $\phi\ne 0$, $\zeta:Y\to i\R$ is the unique function satisfying $-d^*(b-d\zeta)+i \mathrm{Re}\left \langle i\phi,\psi+\zeta\phi\right \rangle =0$. 

There is another \emph{enlarged local Coulomb slice}, defined for the normalized gauge gauge group action characterized by $(b,\psi)\in \cT_{(a,\phi)}$ lies in $\calK^e_{(a,\phi)}$ if and only if $-d^*b+i\mathrm{Re}\left \langle i\phi,\psi\right \rangle$ is a constant function. Similarly, we have the \emph{enlarged local Coulomb slice projection} $\Pielc_{(a,\phi)}\colon\cT_{(a,\phi)}\to \calK^e_{(a,\phi)}$ defined by \[\Pielc_{(a,\phi)}(b,\psi)=(b-d\zeta,\psi+\zeta\phi),\] for which when $\phi\ne 0$, $\zeta:Y\to i\R$ is the unique function satisfying $-d^*(b-d\zeta)+i \mathrm{Re}\left \langle i\phi,\psi+\zeta\phi\right \rangle^\circ =0$. For $f:Y\to i\R$, $f^\circ(y)$ takes the value $f(y)-\mu_{Y}(f)$, in which $\mu_{Y}(f)$ denotes the average of $f$ over $Y$.

All definitions and formulas above can be restricted to $\cC^I$, which appears as a submanifold of $\cC$ and its tangent space. Since all the topological spaces in the discussion above are naturally closed under the $I$-action, we add a superscript $I$ to the denote their fixed point sets and call them \emph{real parts}. For example, we have used $\cC^I$ for the real configuration space and $W^I$ for the real global Coulomb slice. It is easy to see that the estimations in \cite[Section 3.2]{Lidman2016TheEO} still hold after restricting to the real subspaces. The only change is that after taking $I$-invariant part, the enlarged local Coulomb slice is no longer enlarged. Since on $W^I$ the remaining constant gauge group is $\Z_2$, the local slice for the normalized gauge group and the ordinary gauge group are the same and $\Pielc=\Pilc$ when restricting to $\cC^I$. Nevertheless, if we use $\cT^{gC,I}_{(a,\phi)}$ to denote the tangent space along $W^I$ at $(a,\phi)$, $(\Pigc_*)_{(a,\phi)}:\calK^{e,I}_{(a,\phi)}\to \cT^{gC,I}_{(a,\phi)}$  and $(\Pielc_*)_{(a,\phi)}:\cT^{gC,I}_{(a,\phi)}\to \calK^{e,I}_{(a,\phi)}$ still act as inverses to each other for any $(a,\phi)\in W^I$. The proof in \cite[Section~3.2]{Lidman2016TheEO} generalizes directly to our case. 

\subsubsection{The Seiberg-Witten equation}
For $(Y,\iota,g,\s,\mathbb{S})$ defined as above, we fix an invariant  $\mathrm{spin^c}$ connection $A_0$ and the induced identification between imaginary 1-forms and connections as above. We will use $D_a:\Gamma(\bS)\to \Gamma(\bS) $ to denote the Dirac operator corresponding to the connection $A_0+a$ and $D$ for the case $a=0$.

On $\cC(Y)$, we have the Chern-Simons-Dirac (CSD) functional, $\cL$ defined by \[\cL(a,\phi)=\frac{1}{2}(\int_Y \left \langle \phi,D_a\phi \right \rangle-\int_Ya \wedge da).\] Following the notation in \cite{Lidman2016TheEO}, we let $\cX$ denote the $L^2$ gradient of the CSD functional:\[\cX(a,\phi)=(*da+\chi(\phi,\phi),D_a\phi),\] where $\chi(\phi,\phi)=\rho^{-1}((\phi\phi^*)_0)$ is given by taking the preimage of the traceless part of $\phi\phi^*$. The critical points of $\cL$ are exactly solutions to the \emph{Seiberg-Witten equation} $\cX(a,\phi)=0$. Regarding $\cX$ as a map $\cC(Y)\to \cC(Y)$, it is equivariant with respect to the involution $I$, so it restricts to a map $\cC^I(Y)\to \cC^I(Y)$.

Following \cite{Lidman2016TheEO}, we introduce a new metric $\tg$ on $W$ defined by \[\left \langle (b,\psi),(b',\psi') \right \rangle_{\tg}=\mathrm{Re} \left \langle \Pi^{elC}_{(a,\phi)}(b,\psi), \Pi^{elC}_{(a,\phi)}(b',\psi')\right \rangle,\]  for $(b,\psi)$ and $(b',\psi')$ in $T_{(a,\phi)}W$. This metric restricts to a metric on $W^I$ and still has the nice property that the trajectories of the gradient flow of $\cL$ restricted to $W^I$ are precisely the Coulomb projections of the original gradient flow trajectories in $\cC^I(Y)$.

In the global slice $W$ with metric $\tg$, the (downward) gradient flow equation is given by \[\frac{d}{dt} \gamma(t)=-(\Pigc)_{\gamma(t)}\cX(\gamma(t)),\] where $\gamma(t)=(a(t),\phi(t))$. The right-hand side can be rewritten as $(l+c)(\gamma(t))$ where \[l(a,\phi)=(*da,D\phi)\]
\[c(a,\phi)=(\pi\circ\chi(\phi,\phi),\rho(a)\phi+\xi(\phi)\phi),\] with $\xi(\phi)$ characterized by $d\xi(\phi)=(1-\pi)\circ \chi(\phi,\phi)$ and $\int_Y\xi=0$. Taking completion, we have the maps \[\cX^{gC}=l+c:W_k\to W_{k-1},\] in which $l$ is a linear Fredholm operator (self-adjoint in $L^2$ metric, but not in $\tg$) and $c$ is a compact operator. The corresponding flow lines are called \emph{Seiberg-Witten trajectories}(in the global Coulomb slice). Such a flow line $\gamma:\R\to W$ is \emph{of finite type} if $\cL(\gamma(t))$ and $\left\Vert \phi(t) \right\Vert_{C^0}$ are bounded in $t$.  

\subsubsection{Finite dimensional approximation}\label{subsub:Finite dimensional approximation}
Let $W^\lambda$ denote the finite-dimensional subspace of $W$ generated by eigenvectors of $l$ with eigenvalues in the range $(-\lambda,\lambda)$ and $\Tilde{p}^\lambda$ denote the $L^2$ orthogonal projection onto $W^\lambda$. This can be modified to a smooth family in $\lambda$.
As in \cite[section 3.4]{Lidman2016TheEO}, we fix a sequence of positive real numbers \[\lambda_1^{\bullet}<\lambda_2^{\bullet}<\ldots\] such that $\lambda_i^{\bullet} \to \infty$ and none of them (and their negatives) are eigenvalues of $l$. We can take a smooth family of projections $p^\lambda:W\to W^\lambda$ so that $p^{\lambda_i^\bullet}$ are genuine $L^2$ projections.

On $W^{\lambda}$, we will consider flow equation \[\frac{d}{dt}\gamma(t)=-(l+p^{\lambda}c)(\gamma(t)).\]

Fix any natural number $k\ge 5$. There exists a constant $R>0$ such that all the Seiberg-Witten trajectories $\gamma:\R\to W$ of finite type are contained in $B(R)$, the ball of radius $R$ in $W_k$.

\begin{prop}(\cite[Proposition 3]{Manolescu_2003}, \cite[Proposition 3.4.2]{Lidman2016TheEO})\label{prop:flow lines in B(R)}
For any $\lambda$ sufficiently large compared to $R$, if $\gamma:\R\to W$ is a trajectory of $l+p^\lambda c$ and $\gamma(t)$ is in $\overline{B(2R)}$ for all $t$, then $\gamma(t)$ is contained in $B(R)$. 
\end{prop}

Moving to the real setup, everything holds without change when we restrict from $W$ to $W^I$.
\subsubsection{The Conley index and the real Seiberg-Witten Floer spectrum}\label{subsub:The Conley index and the real Seiberg-Witten Floer spectrum}
Using Proposition \ref{prop:flow lines in B(R)}, we know that, for $\lambda$ and $R$ as above, the union $S^\lambda$ of all trajectories of an appropriate cut-off $l+p^\lambda c$ inside $B(R)$ is an isolated invariant set. Then $S^\lambda\cap W^I$ is an isolated invariant set in $W^I$ for the same flow. Inside $W^I$, everything mentioned above is $\Z_2$-equivariant, so we can construct a $\Z_2$-equivariant Conley index $I^{\lambda}$.  We define the real Seiberg-Witten Floer homotopy type \[\mathit{SWF}_{\Z_2}(Y,\iota,\s)=\Sigma^{-\mathrm{dim}(V^{0}_{-\lambda})\R} \Sigma^{-(\mathrm{dim}(U^{0}_{-\lambda})+n^R(Y,\iota,\s,g) )\Tilde{\R}} I^{\lambda},\]
where $V^{\mu}_{\lambda}$ ($U^{\mu}_{\lambda}$) denotes the summand of $(W^I)^{\mu}_{\lambda}$ that is isomorphic to a direct sum of trivial $\Z_2$ representation $\R$ (sign representation $\Tilde{\R}$). $n^R(Y,\iota,\s,g)=\mathrm{ind}_{\C}(D^+_A)-1/8 (c_1^2(\s)-\sigma(X))$ is the correction term for the choice of metric, in which $X$ is a four-dimensional $\mathrm{spin^c}$ bound for the $\mathrm{spin^c}$ three-manifold $(Y,\underline{\s})$ equipped with a metric extending $g$. Here, we almost follows the notations from \cite{Konno2024}, but to distinguish the summand isomorphic to copies of $\tilde{\R}$ with the global Coulomb slice, we use $U$ in place of their $W$. 
%This $n^R$ is almost the same as the one appeared in \cite[Section 3.5]{Konno2024} since $\mathrm{ind}_{\C}(D^+_A)=\mathrm{ind}_{\widetilde{\R}}(D^+_A)$, but careful readers should note that the term $c_1^2(\s)$ was missing from their formula. 
The argument in \cite[Proposition 3.22]{Konno_Miyazawa_Taniguchi_2025} shows that this is indeed an invariant of $(Y,\iota,\s)$.

\begin{remark}\label{rmk: different definitions of correction terms}
We have an alternative definition for the correction term when the real $\mathrm{spin^c}$ three-manifold $(Y,\iota,\s)$ bounds a real $\mathrm{spin^c}$ four-manifold $(X,\iota,\s')$. In this case, we can define \[n^R(Y,\iota,\s,g)=\mathrm{ind}_{\R}^{\tau}(D^+_A)-1/8 (c_1^2(\s)-\sigma(X)),\] in which we equip $X$ with an equivariant metric extending the $g$ on $Y$.
       
The two definitions can be identified since the $\tau$-action on the spinor bundle is anti-complex-linear, so \[\mathrm{ind}_{\R}^{\tau}(D^+_A)=\mathrm{ind}_{\C}(D^+_A).\](cf.\cite[Section 4.3]{li2022monopolefloerhomologyreal}) The characterization using real $\mathrm{spin^c}$ bound will be useful when we identify the grading from $\mathit{SWF}_{\Z_2}$ with the one from $\mathit{HMR}$ in Subsection \ref{sub:Identification of grading}. 
%Thus, our definition coincides with the one given in \cite[Section 3.5]{Konno2024} and Proposition 3.7 there tells us that this is an invariant associated to the real $\mathrm{spin^c}$ three-manifold $(Y,\iota,\s)$. (For a detailed proof, see.)
\end{remark}

\subsection{Real monopole Floer homology}\label{sub:Real Monopole Floer homology}
Unlike real Seiberg-Witten homotopy type, we will be working on the invariant part of entire configuration space \[\cC^I(Y)= \Omega^1(Y;i\R)^{-\iota^*}\oplus\Gamma(\bS)^{\tau}.\]

\subsubsection{Seiberg-Witten equations on the blow-up}

Let $Z$ be the cylinder $I\times Y$ where $I$ is an interval and it might be equal to $\R$. As usual, the four-dimensional spinor bundle splits into $\bS^{\pm}$ according to the eigenvalue of $\rho(d\mathrm{vol}_Z)$, both of them can be identified with three-dimensional spinor bundle $\bS$ over $Y$. After choosing a base $\mathrm{spin^c}$ connection (we always choose a real flat connection as the base connection), we have an identification \[\cC(Z)=\{(a,\phi)|a\in \Omega^{1}(Z;i\R), \phi\in \Gamma(\bS)^{+}\}.\] This comes naturally with a gauge group action by $\cG(Z)=C^{\infty}(Z;S^1)$. An element in $\cC(Z)$ is \emph{in temporal gauge} if it is given by a path $\gamma(t)=(a(t),\phi(t))$ in $\cC(Y)$. Any configuration in $\cC(Z)$ is gauge-equivalent to one in the temporal gauge. 

The real structure and real $\mathrm{spin^c}$ structure on $Y$ give rise to involutions on $Z$ and its spinor bundles, so it makes sense to talk about the $I$ action on $\cC(Z)$ and $\cG(Z)$. Taking the invariant part, we obtain the real configuration space $\cC^I(Z)$ on which $\cG^I(Z)$ acts.

The four-dimensional Seiberg-Witten equations on $Z$ is \[\cF(a,\phi)=(d^+a-\rho^{-1}((\phi\phi^*)_0),D_a^+\phi)=0,\] where $\cF$ is a map from $\cC(Z)$ to $\Omega^+_2(Z;i\R)\oplus\Gamma(\bS^-)$. Let $\cV$ be the trivial bundle over $\cC(Z)$ with fiber $\Gamma(Z;i\bigwedge_+^2T^*Z\oplus \bS^-)$, then $\cF$ can be regarded as a section of this bundle. Note that $\cF$ is equivariant with respect to the involution $I$, so it restricts to a map between the invariant parts of the domain and codomain.

To deal with reducible solutions, we blow up the configuration space to \[ C^\sigma(Y)=\{(a,s,\phi)|s\ge 0,\left\Vert\phi\right\Vert_{L^2}=1\} \subset \Omega^1(Y;i\R)\times \R_{\ge 0}\times \Gamma(\bS).\] Similarly, we have the real blow-up configuration space \[C^{\sigma,I}(Y)=\{(a,s,\phi)|s\ge 0,\left\Vert\phi\right\Vert_{L^2}=1\} \subset \Omega^1(Y;i\R)^{-\iota^*}\times \R_{\ge 0}\times \Gamma(\bS)^{\tau}.\] The same construction gives rise to $C^{\sigma}(Z)$ and $C^{\sigma,I}(Z)$. Note that $C^{\sigma,I}$ embeds as a regular submanifold of $C^\sigma$, since it is the invariant part of the inherited $I$-action on the first and third factor. 

On the blow-up configuration spaces, we have \emph{blow-up Seiberg-Witten equations} \[\cF^\sigma(a,s,\phi)=(d^+a-s^2\rho^{-1}((\phi\phi^*)_0),D^+_a\phi).\] 
It is also useful to consider the $\tau$-model for blow-up: \[\cC^{\tau}(Z) \subset \Omega^1(Z;i\R)\times C^\infty(I)\times C^\infty(Z;\bS^+),\] the space of triples $(a,s,\phi)$ with $s(t)\ge 0$ and $\left\Vert \phi(t) \right\Vert_{L^2(Y)}=1$ for all $t$. Again, $\cC^\tau(Z)$ comes naturally with an involution extending the three-dimensional one, so we can talk about $\cC^{\tau,I}(Z)$. 

By rewriting the Seiberg-Witten equation in temporal gauge using the $\tau$-model, we get a vector field $\cX^{\sigma}$ on $C^\sigma(Y)$ extending $\cX$ on $\cC(Y)$. This vector field is equivariant with respect to the real structures, so we have a real analogue of \cite[Proposition 4.1.1]{Lidman2016TheEO}.
\begin{prop}\label{prop:zeros on the blow up}
    If $s>0$, then $(a,s,\phi)\in C^{\sigma,I}(Y)$ is a zero of $\cX^\sigma$ if and only if $(a,s\phi)\in C^{I}(Y) $ is a zero of $\cX$. If $s=0$, then $(a,s,\phi)\in C^{\sigma,I}(Y)$ is a zero of $\cX^\sigma$ if and only if $(a,0)\in C^{I}(Y) $ is a zero of $\cX$ and $\phi$ is an eigenvector of $D_a$.
\end{prop}
The gauge group action on the blow-ups is well-defined and free. As we did for configuration spaces before blowing up, we can consider the decomposition of the tangent bundle into its tangent and normal parts to the orbit, as usual (as well as their Sobolev completions). We omit details of this part; see \cite[Section 4.1]{Lidman2016TheEO} and \cite[Section 5]{li2022monopolefloerhomologyreal}.

For further references, we introduce notations \[\cB(Y)=\cC(Y)/\cG, \text{ } \cB^\sigma(Y)=\cC^\sigma(Y)/\cG, \text{ } \cB_k(Y)=\cC_k(Y)/\cG_{k+1}, \text{ }, \cB^{\sigma}_k(Y)=\cC^\sigma_k(Y)/\cG_{k+1} \text{ }\] and add superscript $I$ to their real analogues. Similar notations will also be used in dimension four.
\subsubsection{Perturbations}\label{subsub:Perturbations}

To achieve transversality on various moduli spaces, we need to perturb the Seiberg-Witten equations. This is done by adding $\frakq: \cC(Y)\to \cT_0$ to $\cX$. Such a $\frakq$ is \emph{the formal gradient} of $f$ if for any $\gamma\in C^{\infty}([0,1],\cC(Y))$, \[\int_{0}^1 \left \langle \frac{d}{dt}\gamma(t),\frakq(\gamma(t)) \right \rangle_{L^2}dt=f\circ\gamma(1)-f\circ\gamma(0).\]  We will consider \[\cX_{\frakq}=\cX+\frakq\text{ and } \cL_\frakq=\cL+f.\]

The perturbation can be decomposed into its form and spinor parts, denoted by $\frakq^0$ and $\frakq^1$, respectively. From this, we can also obtain perturbations on four-dimensional configuration spaces and blow-up configuration spaces. For a summary, one should look up \cite[Section 4.2]{Lidman2016TheEO}, and for details, one should refer to \cite{Kronheimer_Mrowka_2007}.

An analogue of Proposition \ref{prop:zeros on the blow up} holds for zeros of $\cX_\frakq^\sigma$ in $\cC^{\sigma}(Y)$ and $\cC^{\sigma,I}(Y)$. (cf. \cite[Proposition 4.2.1]{li2022monopolefloerhomologyreal})

In \cite[Section 10]{Kronheimer_Mrowka_2007}, they introduced a concept called \emph{k-tame} perturbation and constructed cylinder functions by pairing with certain 1-forms and sections of spinor bundles. The gradient of these functions generates \emph{large Banach spaces} of tame perturbations. In \cite[Section 6]{li2022monopolefloerhomologyreal}, Li constructed cylinder functions on the real configuration space similarly by pairing with invariant objects. He also obtained corresponding large Banach spaces. It is easy to observe that Li's perturbation can be regarded as the restriction of a cylinder function on $\cC$ to $\cC^I$. The cylinder function defined using pairing with invariant objects has the nice property that the perturbed function $\cL_{\frakq}$ is $I$-invariant and the vector field $\cX_\frakq$ and its various generalizations are $I$-equivariant when regarded as a section of an appropriate bundle. So although in \cite[Remark 6.3]{li2022monopolefloerhomologyreal}, Li remarked that we can define perturbations using pairing with more general forms or sections, we will restrict ourselves to this special class in order to get better equivariant properties.

%As remarked in  \cite[Remark 6.3]{li2022monopolefloerhomologyreal}, we actually have no need to pairing with invariant objects when defining perturbations on $\cC^{I}$ and its blow-ups. In other word, since cylinder functions constructed in \cite[Section 11]{Kronheimer_Mrowka_2007} are $\cG$ invariant on $\cC$, it restrict to $\cG^I$ invariant functions on $\cC^I$, so they can be used as perturbations for the real Monopole Floer homology. This observation will be useful for us to apply various results from \cite{Lidman2016TheEO} to the real setup.

In \cite[Section 4.3]{Lidman2016TheEO}, they further introduce a concept called \emph{very tame perturbation} based on a \emph{functionally bounded} property for non-linear operators. Such perturbations exist and form large Banach spaces is a fact that follows from the construction of cylinder functions, so this is also true for real configuration spaces.
 
Fix a tame perturbation $\frakq$ and a Sobolev number $k$. $\cX_{\frakq}^\sigma$ takes $\cC^{\sigma,I}_{k}(Y)$ to $\cT_{k-1}^{\sigma,I}$. The stationary points of $\cX_{\frakq}^\sigma$ are not isolated, because the gauge group action preserves stationary points. A stationary point $x$ of $\cX_{\frakq}^\sigma$ is called \emph{non-degenerate} if $\cX_{\frakq}^\sigma$ is transverse to $\cJ^{\sigma,I}$ at $x$. This condition can be reformulated in terms of the blow-down configuration for irreducibles and the spectrum of $D_a$ at reducibles, as in \cite[Proposition 7.3]{li2022monopolefloerhomologyreal}. \cite[Section 7.4]{li2022monopolefloerhomologyreal} tells us that we have a residual set of perturbations for which all the critical points are non-degenerate in a large Banach space of perturbations.

We will use $\frC$ to denote the set of critical points in $\cC^\sigma$ of $\cX^\sigma_\frakq$, which decomposed into $\frC^o$, $\frC^u$, $\frC^s$ according to whether they are irreducible, boundary unstable, or boundary stable. A reducible critical point is \emph{boundary-unstable (boundary-stable)} if it has negative (positive) spinorial energy $\Lambda_{\frakq}(a,s,\phi)=\mathrm{Re}\left \langle \phi,D_a\phi+\Tilde{\frakq}^1(a,s,\phi) \right \rangle_{L^2}$, where $\Tilde{\frakq}^1(a,s,\phi)=\int_{0}^1 \cD_{(a,st\phi)} \frakq^1(0,\phi)dt$. We further add a superscript $I$ when we talk about the invariant critical points lying in $\cC^{\sigma,I}(Y)$.

Now we specialize to $\R\times Y$, the infinite cylinder associated to $Y$. On the configuration space associated to $Z$, we can perform completion using local Sobolev norms $L_{k,loc}^2$ in place of the the usual $L_k^2$ norm. In particular, we have the locally integrable real configuration spaces $\cC^{\tau}_{k,loc}(Z)$ and $\cB^{\tau}_{k,loc}(Z)= \cC^{\tau}_{k,loc}(Z)/ \cG^{\tau}_{k,loc}(Z)$.

For regularity of trajectories, we introduce \[\cC^\tau_k(x,y)=\{\gamma\in \cC^{\tau}_{k,loc}(Z)| \gamma-\gamma_0\in L^2_k(Z;iT^*Z)\times L^2_k(\R;\R)\times L^2_k(Z;\bS^+)\}\] after fixing a smooth path $\gamma_0:\R\to \cC^\sigma(Y)$ with $\gamma(t)=x$ for $t\ll0$ and $\gamma(t)=y$ for $t\gg0$. This space has a gauge group action by maps $u:Z\to S^1$ with $1-u\in L^2_{k+1}$. The quotient will be denoted by $\cB^\tau_k([x],[y])$. We remove the condition $s(t)\ge 0$ to obtain the corresponding $\Tilde{\cC}^\tau_k(x,y)$ and $\Tilde{\cB}^\tau_k([x],[y])$. 

For $x,y\in \frC^I$, we can consider ${\cC^{\tau,I}_{k}(x,y)}$ consisting of paths in $\cC^I_{k,loc}(Y)$ from $x$ to $y$. Similarly, this space has an action by real maps $u:Z\to S^1$ with $1-u\in L^2_{k+1}$.  Then, we define the \emph{real moduli space of trajectories} from $[x]$ to $[y]$ to be \[ M([x],[y])=\{[\gamma]\in{\cB^{\tau,I}_{k,loc}(Z)} |\cF_\frakq
^\tau(\gamma)=0, \lim_{t\to-\infty} [\tau_t^*\gamma]=[\gamma_x], \lim_{t\to\infty} [\tau_t^*\gamma]=[\gamma_y]\},\] where $\tau_t$ is the translation $s\mapsto t+s$ on $\R$ and $\gamma_x$, $\gamma_y$ are constant trajectories at $x$, $y$. The moduli space is called \emph{boundary-obstructed} if $[x]$ is boundary-stable and $[y]$ is boundary-unstable. We will use $\Breve{M}$ to denote the quotient of $M$ by the usual $\R$-action. 

For $\gamma=(a,s,\phi)\in {\cC^{\tau,I}_k(x,y)}$, we consider the operator \[Q_{\gamma}=\cD^{\tau}_{\gamma}\cF^{\tau}_{\frakq}\oplus \bfd^{\tau,\dagger}_{\gamma}:\cT_{j,\gamma}^{\tau,I}\to \cV^{\tau,I}_{j-1,\gamma}(Z)\oplus L_{j-1}^2(Z;i\R)^{-\iota^{*}},\] where $\bfd^{\tau,\dagger}_{\gamma}(b,r,\psi)=-d^*b+is^2 \mathrm{Re} \left \langle i\phi,\psi \right \rangle +i\vert \phi \vert \mathrm{Re} \mu_{Y} \left \langle i\phi,\psi \right \rangle$ and $0<j\le k$. The moduli space $M([x],[y])$ is \emph{regular} if $Q_{\gamma}$ is surjective for all $\gamma$ when it is not boundary-obstructed. In the boundary-obstructed case, $M([x],[y])$ is \emph{regular} if for each $\gamma$ in it, $Q_{\gamma}$ has a 1-dimensional cokernel.

A perturbation $\frakq$ is called \emph{admissible} if it is tame and all critical points of $\cX_\frakq$ are non-degenerate and all moduli spaces of trajectories between critical points are regular. 

%For detail of these concepts, one can refer to \cite[Section 4.5]{Lidman2016TheEO} or relevant chapters in \cite{Kronheimer_Mrowka_2007}. For real analogues, one should refer to \cite[Section 6-7]{li2022monopolefloerhomologyreal}.

From \cite[Section 15]{Kronheimer_Mrowka_2007} and \cite[Section 6-7]{li2022monopolefloerhomologyreal}, we know that in any large Banach space of tame perturbations, we have an admissible one. Since we have a large Banach space of very tame perturbations, we can find a $\frakq$ that is very tame and admissible.

\subsubsection{Real monopole Floer homology}\label{subsub:Real Monopole Floer homology}
Let $C^{\theta}$ be the $\F$-vector space generated by $\frC^{\theta,I}$ for $\theta\in\{o,s,u\}$, the real monopole Floer chain complex is \[\widecheck{CMR}(Y,\s,\frakq)=C^o\oplus C^s.\] By counting zero-dimensional irreducible and reducible moduli spaces, we introduce $\partial^{\theta}_{\varpi}$ and $\bar{\partial}^{\theta}_{\varpi}$. The boundary operator $\check{\partial}$ on $\widecheck{CMR}$ is defined as a combination of this; the admissibility of $\frakq$ guarantees that $\check{\partial}^2=0$. We do not need the precise expression of $\check{\partial}$, so we refer readers to \cite[Section 11.1]{li2022monopolefloerhomologyreal} for an explicit formula. Thus, we can define $\widecheck{\mathit{HMR}}(Y,\iota,\s)$ to be the homology of this chain complex. In \cite{li2022monopolefloerhomologyreal}, Li showed that this is an invariant associated to the real $\mathrm{spin^c}$ three-manifold $(Y,\iota,\s)$.

Recall that we have assumed that $Y$ is a rational homology sphere, so $\cG^I$ has exactly two components, characterized by taking the fixed set to $\pm 1$ (recall the we assume $\mathrm{fix}(\iota)$ is non-empty), as we had analyzed in Subsection \ref{subsub: Configuration space and Coulomb slices}. So the $\cR_n$-module structure ($n= |\mathrm{fix}(\iota)|$) defined in \cite{li2022monopolefloerhomologyreal} reduces to a $\F[v]$-module structure. Take any $p\in \mathrm{fix}(\iota)$. The evaluation $ev_p:\cG^I\to \Z_2$ gives rise to a real line bundle $L\to\cB^{\sigma,I}(Y\times \R)$. The $v$-action is defined by counting zero-dimensional moduli spaces of the form \[M_z([x],[y])\cap V,\]  where $V$ is the zero set of a generic smooth section of $L$ and $M_z([x],[y])\subset M([x],[y])$ consists of paths in the relative homotopy class $z$. This will be identified with the $\F[v]$-action of $\widetilde{H}_{\Z_2}^*(S^0;\F)$ on $H^{\Z_2}_*(\mathit{SWF}_{\Z_2}(Y,\iota,\s);\F)$.

\subsubsection{Gradings}\label{subsub:Gradings}
In \cite[Section 8.3]{li2022monopolefloerhomologyreal}, Li defined a relative grading by \[\mathrm{gr}(x,y)=\mathrm{ind}(Q_\gamma),\] for $x$,$y$, critical points of $\cX^{\sigma}_{\frakq}$ in $\cC^I(Y)$. For gauge orbits $[x]$, $[y]$ and $z$, the relative homotopy class of projection of the path $\gamma$ in $\cB^{\tau,I}$, he defined \[\mathrm{gr}_z([x],[y])=\mathrm{ind}(Q_\gamma).\] This is additive before taking gauge orbits: \[\mathrm{gr}(x,y)+\mathrm{gr}(y,w)=\mathrm{gr}(x,w),\] for stationary points $x$, $y$, $w$ in $\cC^I(Y)$. 

Using \cite[Lemma~8.16]{li2022monopolefloerhomologyreal} (see also  \cite[Lemma~14.4.6]{Kronheimer_Mrowka_2007}), we know that when $H^1(Y;\Z)^{-\iota^*}=0$, the grading on $\cB^{\tau,I}$ does not depend on the choice of a relative homotopy class. So from now on, we will write $\mathrm{gr}([x],[y])$ for simplicity.

In \cite{li2022triangle}, he defined an absolute lift of this grading when $Y=\Sigma(K)$ is the double branched cover over $S^3$ branched over a link $K \subset S^3$ with $\mathrm{det}(K)\ne 0$. We generalize that definition as follows: For a real $\mathrm{spin^c}$ cobordism $W$ between two real $\mathrm{spin^c}$ rational homology spheres $Y_+$ and $Y_-$, we define \[\iota^R(W)=b^1_{-\iota^*}(W)-b^+_{-\iota^*}(W)-b^0_{-\iota^*}(W).\]
When a real $\mathrm{spin^c}$ rational homology sphere $(Y,\iota,\s)$ bounds a real $\mathrm{spin^c}$ four-manifold $(W,\iota,\s')$, we puncture $W$ so that it becomes a real $\mathrm{spin^c}$ cobordism from $S^3$ to $Y$. We define \[\mathrm{gr}^{\Q}([x])=-\mathrm{gr}_z([x_0],W,[x])+\frac{1}{8}(c_1(\s')^2-\sigma(W))+\iota^R(W),\]
where $[x_0]$ is the reducible critical point represented by an eigenvector of the lowest positive eigenvalue of a perturbed Dirac operator on $S^3$, equipped with the real structure as the branched double cover of the unknot. This is well-defined, since for a closed real $\mathrm{spin^c}$ four-manifold, we have the dimension formula \[d= \frac{1}{8}(c_1(\s')^2-\sigma(W))+(b^1_{-\iota^*}(W)-b^+_{-\iota^*}(W)-b^0_{-\iota^*}(W)).\](cf.\cite[Lemma 4.3]{li2022monopolefloerhomologyreal} and \cite[Section 28]{Kronheimer_Mrowka_2007})

Specializing to double branched covers, the discussion in \cite[Section 3]{li2022triangle} tells us that this coincides with the definition given there.

\section{Real monopole Floer homology in global Coulomb slice}\label{sec:Real Monopole Floer homology in global Coulomb slice}
In this section, we will recast the construction in subsection \ref{sub:Real Monopole Floer homology} in the global coulomb slice $W^I= (\mathrm{ker}d^*)^{-\iota^*}\oplus \Gamma(\bS)^\tau$ following the procedure in \cite[Chapter 5]{Lidman2016TheEO}.

\subsection{Construction in $W^I$}\label{sub:Construction in W^I}
On $W^I$, we use the Sobolev norm defined by $\nabla_{A_0}$ to perform completions, and we let $\cT_{j}^{gC,I}$ be the trivial bundle over $W_k^{I}$ with fiber $W_j^I$, $j\le k$. We can alternatively use the norm $\nabla_{A_0+a}$ for the fiber over $(a,\phi)\in W^I$, but these two norms are strongly equivalent, so we use the former one for simplicity.

Note that on $W^{I}$, the remaining constant gauge group is $\Z_2$, which is discrete. We do not have a tangent to orbit on $W^I$, so the bundle decomposition of the form $\cT=\cJ\oplus\calK$ becomes $\cT^{gC,I}_k=\calK^{gC,I}_k$. In \cite[Section 5.1]{Lidman2016TheEO}, they introduced \emph{anticircular global Coulomb slice} $\calK^{agC}_k$ as the orthogonal complement of the tangent to gauge orbit with respect to the new metric $\tg$. In the real case, $\calK^{agC,I}_k=\cT^{gC,I}_k=\calK^{gC,I}_k$ as vector bundles over $W_k$, but we keep the different notations to keep track of the metric. Moreover, $\Piagc=\Pigc_{*}$ now.

The metric $\tg$ can be extended to the trivial bundle $\cT_{j}^I$ over $W_k^I$ with fiber $\cC_j^I$. Explicit formulas are given in \cite[Section 5.1]{Lidman2016TheEO}; we restrict them to real subspaces.
%More pre have bundle decomposition \[\cT_j^I=\cJ^I\oplus\cT^{gC,I},\] where $\cJ^I$ is the tangent to the normalized gauge group action.

We have considered blow-up configuration space $\cC^{\sigma,I}(Y)$ in Subsection \ref{sub:Real Monopole Floer homology}. We now blow up $W^I$ using its remaining $\Z_2$-action. That is, \[(W^I)^\sigma=\{(a,s,\phi)| d^*a=0,s\ge 0,\left\Vert \phi \right\Vert_{L^2}=1\}\subset (\cC^I)^\sigma.\] Note that we still have $I$ acting on $\cC^{\sigma}$ by $(a,s,\phi)\mapsto (-\iota^*a,s,\tau(\phi))$, so we have $(\cC^I)^\sigma=(\cC^\sigma)^I$ and $(W^I)^\sigma=(W^\sigma)^I$. We will use notations like $W^{\sigma,I}$ for simplicity. 

On the blow-ups, we still have a global Coulomb projection: \[\Pigcs:\cC^{\sigma,I}\to W^{\sigma,I},\text{  } (a,s,\phi)\mapsto (a-df,s,e^f\phi),\] in which $f=Gd^*a$. The infinitesimal version is given by \[(\Pigcs_*)_x: \cT^I_x\to \cT^{gC,I}_{\Pi^{gC,\sigma}(x)},\text{  }(b,r,\psi)\mapsto (\pi(b),r,\psi+(Gd^*b)\phi)\]  at $x=(a,s,\phi)$.
We also have the local Coulomb slice projection \[\Pilcs_{(a,s,\phi)}:\cT^{\sigma,I}_{j,x}\to \calK^{\sigma,I}_{j,x},\text{ } (b,r,\psi)\mapsto (b-d\zeta,r,\psi+\zeta\phi),\] in which $\zeta$ is characterized by $-d^*(b-d\zeta)+is^2 \mathrm{Re} \left \langle i\phi, \psi+\zeta\phi \right \rangle=0$ and $\mathrm{Re} \left \langle i\phi, \psi+\zeta\phi \right \rangle_{L^2}=0$. 

The anticircular global Coulomb slice is again equal to the whole tangent space, thus the anticircular global Coulomb slice projection coincides with the global Coulomb slice projection. We will use the notation $agC$ to address that we are using the $\tg$ metric instead of the usual $L^2$ metric. Also, the enlarged local Coulomb projection is the same as local Coulomb slice projection on the blow-up. 

As \cite[Lemma 5.1.4-5.1.6]{Lidman2016TheEO}, we have the following:
\begin{lem}
Let $x=(a,s,\phi)\in W^{\sigma,I}$.
\begin{enumerate}
    \item If it is irreducible, the blow-down projection is an isomorphism from $\calK^{agC,\sigma,I}_{j,(a,s,\phi)}$ to $\calK^{agC,I}_{j,(a,s\phi)}$.
    \item The local Coulomb slice projection $\Pilcs_x$ induces a linear isomorphism between $\calK^{agC,\sigma,I}_{j,x}$ and $\calK^{\sigma,I}_{j,x}$. Its inverse is given by the anticircular global Coulomb projection.
    \item If it is reducible, then $\calK^{agC,\sigma,I}_{j,x}=\calK^{\sigma,I}_{j,x}$ and the global Coulomb projection and anticircular global Coulomb projection are both the identity maps.
\end{enumerate}
    
\end{lem}

We introduce a \emph{shear map} as in \cite[Section 5.1]{Lidman2016TheEO} $S_x:\cT^{\sigma,I} \to \cT^{\sigma,I}$ is given by \[\cJ^{\circ,\sigma,I}_{j,x}\oplus \calK^{e,\sigma,I}_{j,x} \to \cJ^{\circ,\sigma,I}_{j,x}\oplus \cT^{agC,\sigma,I}_{j,x}, \text{ } v\oplus w\mapsto v\oplus (\Pigcs_{*})_{x}(w),\] 
where \[\cJ^{\circ,\sigma,I}_{j,x}=\{(-d\xi,0,\xi\phi)| \int_Y \xi=0\} \subset \cJ^{\sigma,I}_{j,x}.\] 
This has an obvious inverse \[\cJ^{\circ,\sigma,I}_{j,x}\oplus \cT^{agC,\sigma,I}_{j,x} \to \cJ^{\circ,\sigma,I}_{j,x}\oplus \calK^{e,\sigma,I}_{j,x}, \text{ } v\oplus w\mapsto v\oplus (\Pielcs_{*})_{x}(w).\]

Let $Z=I\times Y $ be the cylinder, we have a four-dimensional configuration space $\cC(Z)$ consisting of pairs $(a,\phi)$ with $a\in \Omega^1(Z;i\R)$ and $\phi\in \Gamma(\bS^+)$. It can be rewritten as $(a(t)+\alpha(t)dt,\phi(t))$, with $t\in I$ for $a(t)\in \Omega^1(Y;i\R)$, $\alpha(t)\in C^\infty(Y;i\R)$ and $\phi(t)\in \Gamma(\bS)$. Taking the $I$-invariant part, we have the real configuration space $\cC^I(Z)$. 

To obtain a Coulomb slice model for it, we consider $W^I(Z)$ which is the subspace of $\cC^I(Z)$ consisting of configurations that are slicewise living in global Coulomb slice, i.e., $\alpha(t)=0$ and $a(t)\in \mathrm{ker}d^*$ for all $t$. $W^I(Z)$ has a slicewise constant gauge group action by $\cG^{gC,I}(Z)=C^\infty(I;\Z_2)$. However, $\Z_2$ is discrete, so $\cG^{gC,I}(Z)$ is just a copy of $\Z_2$, consisting of two constant gauge transformations. 

Recall that the $I$-action on $\cC(Z)$ is given by slicewise action of $-\iota^*$ and $\tau$, so if $(a+\alpha dt,\phi)\in \cC^I(Z)$, then $(a(t),\phi(t))\in \cC^I(Y)$ and $\alpha(t)$ is a real function on $Y$ for each $t$. In \cite[Section 5.2]{Lidman2016TheEO}, they introduced new concept called the \emph{pseudo-temporal gauge}, which requires $\alpha(t)$ to be slicewise constant. In the real case, this is actually equal to being in the \emph{temporal gauge}, since the only constant real function on $Y$ is zero. Using their notation, we have $\cC^I(Z)=W^I(Z)$.

On $W^I(Z)$, the Seiberg-Witten equations can be written as \[(\frac{d}{dt}+\cX^{gC})(a(t)+\phi(t))=0.\] This is already invariant under the action of $\cG^I(Z)$. The Seiberg-Witten map can also be regarded as a section \[\cF^{gC}: W^{I}(Z)\to \cV^{gC,I}(Z),  \] where $\cV^{gC,I}(Z)$ is the trivial bundle over $ W^I(Z)$ with fiber $W^I(Z)$.

For all these spaces, we can consider Sobolev completions of them and their tangent bundles. When $Z$ is compact, we can as well study their blow-ups; when $Z$ is infinite, we can also complete the spaces with local $L^2_k$ norms. We will follow the notation at the end of \cite[Section 5.2]{Lidman2016TheEO} for these spaces.

In Subsection \ref{sub:Real Monopole Floer homology}, we considered a perturbation of the Seiberg-Witten map by formal gradients. A similar perturbation can be done in the global Coulomb slice. If $\frakq$ is a perturbation on $\cC(Y)$, then \[\eta_{\frakq}(a,\phi)=(\Pigc_*)_{(a,\phi)}\frakq(a,\phi)\] is a valid perturbation on $W$. (So of course on $W^I$.) The formal gradient vector field changes to \[\cX_\frakq^{gC}=\cX^{gC}+\eta_{\frakq}.\] 

In \cite[Section 5.3]{Lidman2016TheEO}, they introduced a new concept \emph{controlled Coulomb perturbation} and proved that very tame perturbations project to controlled Coulomb perturbations by $\Pigc_*$. Thus, we always have controlled perturbations on $W^I(Z)$ from our discussion in Subsection \ref{sub:Real Monopole Floer homology}.

Now we fix a very tame perturbation $\frakq$ and the corresponding controlled Coulomb perturbation $\eta_\frakq$. $\cX^{gC}_\frakq$ can be written as $((\cX^{gC}_\frakq)^0,(\cX^{gC}_\frakq)^1)$ by splitting it into the form part and the spinor part. It will further induce a perturbed vector field on the blow up $W^{\sigma,I}$, given by \[ \cX^{gC,\sigma}_\frakq(a,s,\phi)=((\cX^{gC}_\frakq)^0(a,s\phi),\Lambda_{\frakq}(a,s,\phi)s,\widetilde{(\cX^{gC}_\frakq)}^1(a,s,\phi)-\Lambda_{\frakq}(a,s,\phi)\phi),\] for $\widetilde{(\cX^{gC}_\frakq)}^1(a,s,\phi)=\int_{0}^1 \cD_{(a,sr\phi)}(\cX^{gC}_\frakq)^1(0,\phi)dr$ and $\Lambda_{\frakq}(a,s,\phi)=\mathrm{Re}\left \langle \phi,\widetilde{(\cX^{gC}_\frakq)}^1(a,s,\phi)\right \rangle_{L^2}$.

In a concise way, these are just $\cX^{gC}_{\frakq}=\Pigc_*\circ \cX_\frakq$ and $\cX^{gC,\sigma}_{\frakq}=\Pigcs_*\circ \cX_\frakq^\sigma$. They are $I$-invariant smooth vector fields on $W$ and $W^\sigma$, respectively, thus they restrict to smooth vector fields on $W^I$ and $W^{\sigma,I}$ (without the need to perform a projection).

It is easy to see that any stationary point of $\cX_\frakq^\sigma$ in $\cC^{\sigma,I}$ can be moved into $W^{\sigma,I}$ by a normalized real gauge transformation, while any stationary point of $\cX^{gC,\sigma}_{\frakq}$ in $W^{\sigma,I}$ is a stationary point of $\cX_\frakq^\sigma$. Thus, the projection $\Pigcs$ induces a bijection \[\{\text{stationary points of }\cX_\frakq^\sigma \text{ in } \cC^{\sigma,I}\}/\cG^I\xrightarrow{\cong}\{\text{stationary points of }\cX_\frakq^{gC,\sigma} \text{ in } W^{\sigma,I}\}/\Z_2.\]

This bijection preserves the type of the critical point, i.e., irreducible, boundary-stable or boundary-unstable. The proof of \cite[Proposition 5.4.2]{Lidman2016TheEO} applies to show that \begin{prop}
    The trajectories of $\cX^{gC,\sigma}_{\frakq}$ in $W^{\sigma,I}$ are precisely the global Coulomb projections of the trajectories of $\cX^{\sigma}_{\frakq}$ in $\cC^{\sigma,I}(Y)$. The global Coulomb projection gives rise to a bijection \[\{\text{trajectories of }\cX_\frakq^\sigma \text{ in } \cC^I\}/\cG^I\xrightarrow{\cong}\{\text{trajectories of  }\cX_\frakq^{gC,\sigma} \text{ in } W^I\}/\Z_2.\]
\end{prop}

Again, the four-dimensional Seiberg-Witten map can be interpreted as a section \[\cF^{gC,\tau}_{\frakq}:\cC^{gC,\tau,I}(Z)\to \cV^{gC,\tau,I}(Z).\] In temporal gauge, it can be written as $\cF^{gC,\tau}_{\frakq}=\frac{d}{dt}+\cX_\frakq^{gC,\sigma}$.

\subsection{Hessians}

In \cite[Section 7.3]{li2022monopolefloerhomologyreal}, Li studied the derivative of the Seiberg-Witten map by introducing a notion of Hessian operator. We now further develop this using techniques from \cite[Section 5.5]{Lidman2016TheEO}. The original Hessian operator is defined by\[\mathrm{Hess}_{\frakq,x}=\Pilc_{x}\circ \cD_{x}\cX_{q}:\calK_{k,x}^I\to \calK_{k-1,x}^I,\] when $x$ is an irreducible configuration in $\cC^I(Y)$. Using the new metric $\tg$, we introduce a new Hessian on $W^I_k$ as \[\mathrm{Hess}_{\frakq,x}^{\tg}=\Piagc_{x}\circ \cD_{x}^{\tg}\cX^{gC}_{q}:\calK_{k,x}^{agC,I}\to \calK_{k-1,x}^{agC,I}.\] Here, $\cD^{\tg}$ denotes the connection on $\cT^{gC,I}$ induced by the $\tg$ metric on $W^I_k$. $\cD^{\tg}$ has a simple formula $\cD^{\tg}(X)=\Pigc_*\circ\cD(\Pielc(X))\circ\Pielc$. Using this, we have an alternative formula: 
$\mathrm{Hess}_{\frakq,x}^{\tg}=\Piagc_{x}\circ\cD_x\circ\Pielc_x$.
\cite[Proposition 7.8]{li2022monopolefloerhomologyreal} showed that $\mathrm{Hess}_{\frakq,x}$ is a Fredholm operator of index zero. A similar result holds for this new Hessian:
\begin{lem}\label{lem:extended hessian on blow-down is fredholm index 0}
For any $x=(a,\phi)\in W_k^I$ with $\phi\ne 0$, $\mathrm{Hess}_{\frakq,x}^{\tg}:\calK_{k,x}^{agC,I}\to \calK_{k-1,x}^{agC,I}$ is Fredholm of index zero. Thus, it is injective iff it is surjective.
\end{lem}

Recall that our real three-manifold $Y$ comes with an involution $\iota$. Let $E$ be a vector bundle over $Y$ that decomposes as a direct sum of a real and a complex vector bundle. Equip $E$ with a bundle map $\tau:E\to E$, which is a conjugate-linear involution covering $\iota$. Then $\tau$ acts on the sections of $E$ via $(\tau s)(y)=\tau(s(\iota(y)))$. The proof of the previous proposition relies on the following definition.

\begin{definition} (\cite[Definition 7.4]{li2022monopolefloerhomologyreal})
    An operator $L$ is called \emph{k-almost self-adjoint first order elliptic} ($k$-ASAFOE) if it is of the form $L=L_0+h$ where \begin{enumerate}
        \item $L_0$ is a first order, self-adjoint, elliptic differential operator (with smooth coefficients) acting on invariant sections of a vector bundle $E\to Y$.
        \item $h$ is an operator acting on $\tau$-invariant sections of $E$ as a map \[h:C^\infty(Y;E)^{\tau}\to L^2(Y;E)^{\tau}\] which extends to a bounded map on $L^2_j(Y;E)^{\tau}$ for all $j$ in the range $\vert j\vert\le k$.
    \end{enumerate}
\end{definition}

To fit the Hessian into this setup, we need to consider the extended Hessian \[\widehat{\mathrm{Hess}}_{\frakq,x}:\cT^{I}_{k,x}\oplus L^2_k(Y;i\R)^{-\iota^*}\to \cT^{I}_{k-1,x}\oplus L^2_{k-1}(Y;i\R)^{-\iota^*},\] given by \[\widehat{\mathrm{Hess}}_{\frakq,x}=\begin{bmatrix}
\cD_x\cX_{\frakq} & \bfd_{x}\\
\bfd^{*}_x & 0\\
\end{bmatrix},
\] where $\bfd_x$ encodes the infinitesimal gauge group action and $\bfd^{*}_x$ is its adjoint in $L^2$ metric.

Moving to the blow-up configuration space, we can form a similar  Hessian $\mathrm{Hess}^\sigma_{\frakq,x}=\Pilcs_x\circ\cD^{\sigma}_x\cX^{\sigma}_{\frakq}:\calK^{\sigma,I}_{k,x}\to \calK^{\sigma,I}_{k-1,x}$ and its extension by \[\widehat{\mathrm{Hess}}_{\frakq,x}=\begin{bmatrix}
\cD_x^\sigma\cX_{\frakq} & \bfd^\sigma_{x}\\
\bfd^{\sigma,\dagger}_x & 0\\
\end{bmatrix}.
\] In this formula \[\bfd^{\sigma}_x(\xi)=(-d\xi,0,\xi \phi),\]
\[\bfd^{\sigma,\dagger}_x(b,r,\psi)=-d^* b+is^2 \mathrm{Re}\left \langle i\phi,\psi \right \rangle+i\vert\phi\vert^2 \mathrm{Re} \mu_Y \left \langle  i\phi,\psi \right \rangle.\]

Note that $\bfd^{\sigma,\dagger}_x$ is not the adjoint operator of $\bfd^{\sigma}_x$, so the extended Hessians on the blow-up is no longer symmetric. We use the trick of forming a combination $\bspsi=\psi+r\phi$, so that we can think of $\widehat{\mathrm{Hess}}^\sigma_{\frakq,x}$ as acting on $L^2_j(Y;iT^*Y\oplus\bS\oplus i\R)^{-\iota^*\oplus\tau\oplus-\iota^*}$. Then \cite[Lemma 7.6]{li2022monopolefloerhomologyreal} tells us that it is Fredholm of index zero. As in \cite[Section 7.3]{li2022monopolefloerhomologyreal}, we know that when $x$ is a non-degenerate critical point, $\mathrm{Hess}^\sigma_{\frakq,x}$ is invertible with real spectrum.

Moving to the global Coulomb gauge, we form the $\tg$-Hessian in the blow-up \[\mathrm{Hess}_{\frakq,x}^{\tg,\sigma}=\Piagcs_{x}\circ \cD_{x}^{\tg,\sigma}\cX^{gC,\sigma}_{q}:\calK_{k,x}^{agC,\sigma,I}\to \calK_{k-1,x}^{agC,\sigma,I},\]  where \[\cD_{x}^{\tg,\sigma}X=\Pigcs_{*}\circ\cD^{\sigma}(\Pielcs(X))\circ\Pielcs.\]
Using the relationship between $\Piagcs$ and $\Pielcs$ on $W^I$, we have $\cX^{\sigma}_{\frakq}=\Pielcs\circ \cX^{gC,\sigma}_{\frakq}$. Thus, the $\tg$-Hessian on the blow-up can be rewritten as \[\mathrm{Hess}_{\frakq,x}^{\tg,\sigma}=\Piagcs_{x}\circ \cD_{x}^{\sigma}\cX^{\sigma}_{q}\circ \Pielcs_x.\] 
\begin{lem}(\cite[Lemma 5.5.7-5.5.8]{Lidman2016TheEO})\label{lem:Fredholm properties of g-extended hessian on blow-up}
 \begin{enumerate}
     \item For any $x\in W^{\sigma,I}$, the operator $\mathrm{Hess}^{\tg,\sigma}_{\frakq,x}$ is Fredholm of index zero.
     \item When $x$ is a non-degenerate stationary point of $\cX^{gC,\sigma}_{\frakq}$, the operator   $\mathrm{Hess}^{\tg,\sigma}_{\frakq,x}$ is invertible with real spectrum.
     \item The map $\mathrm{Hess}^{\tg,\sigma}_{\frakq}$ is continuous as a bundle map from $\calK^{agC,\sigma,I}_k$ to $\calK^{agC,\sigma,I}_{k-1}$.
 \end{enumerate}
     
 \end{lem}
 \begin{proof}
     This can be proved in exactly the same way as in \cite[Section 5.5]{Lidman2016TheEO}. We just need to replace the results from \cite{Kronheimer_Mrowka_2007} with the corresponding real version from \cite[Section 7]{li2022monopolefloerhomologyreal}.
 \end{proof}

Now we define two new types of extended Hessians on the blow-up global Coulomb slice. We first consider the \emph{split extended Hessian on the blow-up}
\[\widehat{\mathrm{Hess}}_{\frakq,x}^{sp,\sigma}: \cT^{\sigma,I}_{j,x}\oplus L^2_j(Y;i\R)^{-\iota^*}\to \cT^{\sigma,I}_{j-1,x}\oplus L^2_{j-1}(Y;i\R)^{-\iota^*},\] given by \[\widehat{\mathrm{Hess}}_{\frakq,x}^{sp,\sigma}=\begin{bmatrix}
\cD_x^\sigma\cX_{\frakq}^{\sigma} & \bfd_{x}^{\sigma}\\
\bfd^{sp,\sigma,\dagger}_x & 0\\
\end{bmatrix},
\]
in which $\bfd^{sp,\sigma,\dagger}$ is just $-d^*$ acting on the component $(-d\xi,0,\xi\phi)\in \cJ^{\circ,\sigma,I}_{j,x} $ when we decompose $\cT^{\sigma,I}_{j,x}$ into $\calK^{\sigma,I}_{j,x} \oplus \cJ^{\circ,\sigma,I}_{j,x}$. These formulas are far simpler than those in \cite[Section 5.5.3]{Lidman2016TheEO}, since our constant gauge group is $\Z_2$, so we have no tangent to the orbit after quotienting out the normalized gauge group. Disregarding this difference, we still have \begin{lem}
    $\widehat{\mathrm{Hess}}_{\frakq,x}^{sp,\sigma}$ is a $k$-ASAFOE operator and when $x$ is a non-degenerate stationary point of $\cX^{gC,\sigma}_{\frakq}$ in $W^{\sigma,I}$, $\widehat{\mathrm{Hess}}_{\frakq,x}^{sp,\sigma}$ is invertible with real spectrum.
\end{lem}

Next, we consider the \emph{$\tg$-extended Hessian} given by \[\widehat{\mathrm{Hess}}_{\frakq,x}^{\tg,\sigma}: \cT^{\sigma,I}_{j,x}\oplus L^2_j(Y;i\R)^{-\iota^*}\to \cT^{\sigma.I}_{j-1,x}\oplus L^2_{j-1}(Y;i\R)^{-\iota^*},\] given by \[\widehat{\mathrm{Hess}}_{\frakq,x}^{\tg,\sigma}=\begin{bmatrix}
S_x\circ\cD_x^\sigma\cX_{\frakq}^{\sigma}\circ S_x^{-1} & \bfd_{x}^{\sigma}\\
\bfd^{\sigma,\tilde{\dagger}}_x & 0\\
\end{bmatrix},
\] in which $\bfd^{\sigma,\tilde{\dagger}}_x$ is the same as  $\bfd^{sp,\sigma,\dagger}$, we use various notations to be consistent with definitions in \cite[Section 5.5]{Lidman2016TheEO}. This kind of extended Hessian is not as good as others, however, as shown in \cite[Lemma 5.5.11-12]{Lidman2016TheEO}, we have the following:
\begin{lem}\label{lem:operator cH and kASAFOE property of gextended Hessian on blow up}
    Let $x\in W^I_k$ and $1\le j\le k$. Consider the operator $\cH_{x}^\sigma: \cT^{\sigma,I}_{j,x}\oplus L^2_j(Y;i\R)^{-\iota^{*}}\to \cT^{\sigma,I}_{j-1,x}\oplus L^2_{j-1}(Y;i\R)^{-\iota^{*}}$ given by the block matrix \[\cH_{x}^\sigma=\begin{bmatrix}
        (\cD^{\sigma}_x\cX^{gC,\sigma}_{\frakq})\circ \Pigcs_* & \bfd_x^\sigma \\
        \bfd^{\sigma,\tilde{\dagger}}_x& 0\\
    \end{bmatrix}\] 
    \begin{enumerate}
        \item Using the isomorphism $\cT^{\sigma,I}_{j,x} \oplus L^2_j(Y;i\R)^{-\iota^{*}}\cong L^2_j(Y;iT^*Y\oplus\bS\oplus\R)^{-\iota^*\oplus\tau\oplus-\iota^*}$ and a similar one for $j-1$, the operator $\cH_{x}^\sigma$ is $(k-1)$-ASAFOE with linear part \[L_0=\begin{bmatrix}
            *d& 0& -d\\
            0 & D &0\\
            -d^*& 0 & 0\\
        \end{bmatrix}\]
        \item When $j=k$, the operator $\cH_{x}^\sigma$ differs from $L_0$ by a compact operator from $L_k^2$ to $L_{k-1}^2$.
        \item If $x$ is a stationary point, then $\cH_{x}^\sigma$ is $k$-ASAFOE and $\cH_{x}^\sigma =\widehat{\mathrm{Hess}}_{\frakq,x}^{\tg,\sigma}$.
    \end{enumerate}
 \end{lem}

 \begin{lem}
     If $x$ is a non-degenerate stationary point of $\cX^{gC,\sigma}$ in $W^{\sigma,I}_k$, then $\widehat{\mathrm{Hess}}_{\frakq,x}^{\tg,\sigma}$ is invertible with real spectrum.
 \end{lem}

Next, we use a sequence of interpolations to relate different notions of extended Hessians. As a first step, we consider the convex linear combination of $\widehat{\mathrm{Hess}}_{\frakq,x}^{\sigma}$ and $\widehat{\mathrm{Hess}}_{\frakq,x}^{sp,\sigma}$. 
\begin{lem} 
If $x$ is a non-degenerate stationary point of $\cX^{gC,\sigma}$ in $W^{\sigma,I}_k$, then for any $\rho\in [0,1]$, $\rho\widehat{\mathrm{Hess}}_{\frakq,x}^{\sigma}+(1-\rho) \widehat{\mathrm{Hess}}_{\frakq,x}^{sp,\sigma}$is invertible with real spectrum.
 \end{lem}

Then, we consider a family of metrics $g_\rho=(1-\rho)g_{L^2}+\rho \tg$. Each of these is a well-defined metric on $\cT^{I}_{j,x}$. Using them, we can form the family $g_{\rho}$-extended Hessian\[\widehat{\mathrm{Hess}}_{\frakq,x}^{\rho,\sigma}=\begin{bmatrix}
S_x^{\rho}\circ\cD_x^\sigma\cX_{\frakq}^{\sigma}\circ (S_x^\rho)^{-1} & \bfd_{x}^{\sigma}\\
\bfd^{\sigma,\tilde{\dagger}}_x & 0\\
\end{bmatrix},
\]  
in which $S^\rho$ is the shear map associated to the metric $g_\rho$.
\begin{lem}
If $x$ is a non-degenerate stationary point of $\cX^{gC,\sigma}$ in $W^{\sigma,I}_k$, then for any $\rho\in [0,1]$ the operator $\widehat{\mathrm{Hess}}_{\frakq,x}^{\rho,\sigma}$ is invertible with real spectrum.
\end{lem}

As usual, we want to use these various Hessian operators to characterize the non-degeneracy of stationary points. We have the following real analogue of \cite[Lemma 5.6.1]{Lidman2016TheEO}.
\begin{lem}
    Let $x\in W^I_k$ be a stationary point of $\cX^{gC}_\frakq$. The following are equivalent.\begin{enumerate}
        \item $x$ is non-degenerate;
        \item $\mathrm{Hess}_{\frakq,x}:\calK_{k,x}^I\to \calK_{k-1,x}^I$ is surjective;
        \item $\cX^{gC}_\frakq$ is transverse to zero at $x$;
        \item $\mathrm{Hess}_{\frakq,x}^{\tg}:\calK_{k,x}^{agC,I}\to \calK_{k-1,x}^{agC,I}$ is surjective.
    \end{enumerate}
\end{lem}
Combining this with Lemma \ref{lem:extended hessian on blow-down is fredholm index 0}, we have an even simpler characterization of non-degeneracy.
\begin{lem} 
Let $x\in W^{\sigma,I}_k$ be a stationary point of $\cX^{gC,\sigma}_\frakq$, then $x$ is non-degenerate $\Longleftrightarrow$ $\mathrm{Hess}_{\frakq,x}^{\tg,\sigma}$ is injective $\Longleftrightarrow$ $\mathrm{Hess}_{\frakq,x}^{\tg,\sigma}$ is surjective.
\end{lem}

Till now, we have only worked with $W^I$ and $W^{\sigma,I}$. Now we are ready to quotient out the remaining $\Z_2$-action. For $x\in W^{\sigma,I}_k$, we write $[x]$ for its equivalence class in $W^{\sigma,I}_k/\Z_2$. The $\Z_2$-action is discrete, so the tangent map of this quotient is an isomorphism. Note that $\cX^{gC,\sigma}_{\frakq}$ is $\Z_2$-equivariant, so there is an induced vector field  $\cX^{agC,\sigma}_{\frakq}$ on $W^{\sigma,I}_k/\Z_2$. 
\begin{lem}\label{lem:identification of stationary points on blowup}
By composing the global Coulomb projection with the quotient of $\Z_2$-action, we have the following one-to-one correspondences:
\[\{\text{stationary points of }\cX_\frakq^\sigma \text{ in } \cC^{\sigma,I}\}/\cG^I_{k+1}\xrightarrow{\cong}\{\text{stationary points of }\cX_\frakq^{agC,\sigma} \text{ in } W^{\sigma,I}_{k}/\Z_2\},\]
\[\{\text{trajectories of }\cX_\frakq^\sigma \text{ in } \cC^{\sigma,I}\}/\cG^I_{k+1}\xrightarrow{\cong}\{\text{trajectories of  }\cX_\frakq^{agC,\sigma} \text{ in } W^{\sigma,I}_{k}/\Z_2\}.\]
\end{lem}

Let $x\in W^{\sigma,I}_k$ be a stationary point of $\cX^{gC,\sigma}_\frakq$, then $[x]$ is a stationary point of $\cX^{agC,\sigma}_\frakq$ in $ W^{\sigma,I}_k/\Z_2$. Since we have $L^2_{k-1}\text{ completion of }T_{x}W^{\sigma,I}=\calK^{gC,\sigma,I}_{k-1}\cong\calK^{agC,\sigma,I}_{k-1}\cong L^2_{k-1}\text{ completion of }T_{[x]} W^{\sigma,I}/\Z_2$. We have \[\mathrm{Hess}_{\frakq,x}^{\tg,\sigma}=\cD^{\sigma}_{[x]}\cX^{agC,\sigma}_{\frakq}.\]

\begin{lem}\label{lem:identification of non-degeneracy}
Under the identification in Lemma \ref{lem:identification of stationary points on blowup}, the non-degeneracy of a stationary point $x$ of $\cX^{\sigma}_{\frakq}$ is equivalent to the injectivity (or bijectivity) of $\cD^{\sigma}_{[\Pigcs(x)]}\cX^{agC,\sigma}_{\frakq}$.
    
\end{lem}

\subsection{Path space and trajectories}
For $x,y\in W^{\sigma,I}$, and a path $\gamma_0$ in $W^{\sigma,I}$ from $x$ to $y$ in the sense that it agrees with $x$ ($y$) when $t\to -\infty$ ($t\to \infty$). We introduce the following four-dimensional configuration space \[\cC^{gC,\tau,I}_{k}(x,y)=\{\gamma\in \cC^{gC,\tau,I}_{k,loc}(Z)|\gamma-\gamma_0\in L^{2}_k(Z;iT^*Z)^{-\iota^*}\times L^2_k(\R,\R)\times L^{2}_k(Z;\bS)^{\tau}\}.\]
Any $\gamma \in \cC^{gC,\tau,I}_{k}(x,y)$ can be written as a path \[\gamma(t)=(a(t)+\alpha(t)dt,s(t),\phi(t)),\] in which $s(t)\ge 0$ and $\left\Vert \phi(t) \right\Vert_{L^2}=1$ for any $t$.

$\cC^{gC,\tau,I}_{k}(x,y)$ embeds in to a larger Hilbert manifold $\widetilde{\cC}^{gC,\tau,I}_{k}(x,y)$, which is defined by removing the condition $s(t)\ge 0$. Inside these, we consider $W^{\tau,I}_{k}(x,y)$ and $\widetilde{W}^{\tau,I}_{k}(x,y)$ consisting of those configurations with $\alpha=0$.

On these configuration spaces, we have no nontrivial gauge group action, since any real map $u:\R\to \Z_2$ satisfying $1-u\in L_{k+1}^2$ is the constant map $1$. So we actually have $\cC^{gC,\tau,I}_{k}(x,y)=\cB^{gC,\tau,I}_{k}(x,y)$. Nevertheless, $\cB^{gC,\tau,I}_{k}(x,y)$ depends only on the classes $[x],[y]\in W^{\sigma,I}/\Z_2$. It is easy to see that $\cB^{gC,\tau,I}_{k}(x,y)$ is Hausdorff.

In Subsection \ref{subsub:Perturbations}, we have considered a similar space $\cB^{\tau,I}_{k}([x],[y])$ and a corresponding ``tilde'' version. We are now aiming to relate $\widetilde{\cB}^{\tau,I}_{k}([x],[y])$ to $\widetilde{\cB}^{gC,\tau,I}_{k}([x],[y])$. We first consider the map $\Pi^{gC,\tau}: \widetilde{\cC}^{\tau,I}_{k}(x,y)\to \widetilde{\cC}^{gC,\tau,I}_{k}(x,y)$. This is defined by \[(a(t)+\alpha(t)dt,s(t),\phi(t))\mapsto \Pigcs(a(t),s(t),\phi(t)).\] The last term in \cite[Equation (120)]{Lidman2016TheEO} is missing, since the average of a real function on a real manifold equipped with an equivariant metric is always zero.

\begin{lem}(\cite[Lemma 5.7.2]{Lidman2016TheEO})
There is a well-defined, continuous map \[\Pi^{[gC],\tau}: \widetilde{\cB}^{\tau,I}_{k}([x],[y])\to \widetilde{\cB}^{gC,\tau,I}_{k}([x],[y]),\text{ }[\gamma]\mapsto [\Pi^{gC,\tau}(\gamma)].\] And this takes $\cB^{\tau,I}_{k}([x],[y])$ to $\cB^{gC,\tau,I}_{k}([x],[y])$
    
\end{lem}
\begin{proof}
    The argument for \cite[Lemma 5.7.2]{Lidman2016TheEO} works in the real case as well.
\end{proof}

Due to the existence of three-dimensional gauge transformations, this map is surjective but not injective.

Consider $[x],[y]\in W^{\sigma,I}/\Z_2$ being stationary points of $\cX^{agC,\sigma}_\frakq$. Define $M^{agC}([x],[y])$ to be the moduli space of trajectories of $\cX^{agC,\sigma}_\frakq$, living inside $\cB^{gC,\tau,I}_{k}([x],[y])$. Similarly, we can consider $M^{agC,red}([x],[y])$ when both of them are reducible.
\begin{prop}
    Every trajectory of $\cX^{agC,\sigma}_\frakq$ in $W^{\sigma,I}/\Z_2$ is actually in $\cB^{gC,\tau,I}_{k}([x],[y])$. Further, the map $\Pi^{[gC],\tau}$ leads to a homeomorphism between moduli spaces $M([x],[y])$ and $M^{agC}([x],[y])$.
\end{prop}
\begin{proof}
    Using Lemma \ref{lem:identification of stationary points on blowup} and \cite[Theorem 8.6]{li2022monopolefloerhomologyreal}, the proof of \cite[Proposition 5.7.3]{Lidman2016TheEO} is still valid in real case.
\end{proof}

%As observed above, $\cG^{gC,I}(Z)$ is trivial, so $\widetilde{\cB}^{gC,\tau,I}_{k}([x],[y])=\widetilde{\cC}^{gC,\tau,I}_{k}([x],[y])$ is naturally a Hilbert manifold. Analysis in \cite[Section 5.8]{Lidman2016TheEO} becomes trivial when restrict to the real part.

%Let $x,y$ be stationary points of $\cX^{agC,\sigma}_\frakq$ in $\cC^{\sigma,I}$. Recall that the moduli space of trajectories can be described as the zero set of \[\cF^{\tau}_{\frakq}:\cC^{\tau,I}_k(x,y)\to \cV^{\tau,I}(Z)\] the modulo gauge. $\cF^{\tau}_{\frakq}$ can be written explicitly as $\frac{d}{dt}+\cX^{\sigma}_{\frakq}$ in temporal gauge.
%$\cV^{\tau,I}(Z)$ is a bundle over $\widetilde{\cC}^{\tau,I}_k(x,y)$, to which we can extend the section $\cF^{\tau}_{\frakq}$. The study of $\cD^{\tau}_{\gamma}\cF^{\tau}_{\frakq}$ is needed to understand the local structure of moduli space and to define gradings. Li studied this in \cite[Section 8]{li2022monopolefloerhomologyreal}.
%Fix $\gamma\in \cC^{\tau,I}(x,y)$ in temporal gauge. The tangent space $\cT_{j,\gamma}^{\tau,I}$ consists of elements of the form $(V,\beta)$, where $V(t)=(b(t),r(t),\psi(t))$ is a path in $\cC^{\sigma,I}(Y)$ and $\beta(t)$ is a path in $L^2(Y;i\R)^{-\iota^*}$.

%\begin{prop}(\cite[Proposition 8.13]{li2022monopolefloerhomologyreal})
    
%\end{prop}
%\YX{Review for original version TBA}

In \cite[Section 5.9]{Lidman2016TheEO}, they reviewed how to analyze Fredholm properties of the operator $Q_{\gamma}$ before working in $W$. To be concise, we directly work in the global Coulomb slice. One should refer to the beginning part of \cite[Section 5.9]{Lidman2016TheEO} and \cite[Section 8]{li2022monopolefloerhomologyreal} for original constructions.

We have previously considered the moduli spaces $M^{agC}([x],[y])\subset \widetilde{\cB}^{gC,\tau,I}_{k}([x],[y])$. This can also be described as the zero set of the section \[\cF^{gC,\tau}_{\frakq}:\widetilde{\cC}^{gC,\tau,I}_{k}(x,y) \to \cV^{gC,\tau,I}_{k-1}(Z).\] In contrast to the usual setup, we now have no gauge transformation to mod out.

Fix some $\gamma\in \cC^{gC,\tau,I}(x,y)$ in temporal gauge. We differentiate $\cF^{gC,\tau}_{\frakq}$ using the $L^2$ metric to study the local structure of moduli spaces. Define $\cD^{\tau}_{\gamma}\cF^{gC,\tau}_{\frakq}:\cT^{gC,\tau,I}_{j,\gamma}\to \cV^{gC,\tau,I}_{j,\gamma}$ by \[\cD^{\tau}_{\gamma}\cF^{gC,\tau}_{\frakq}(V)=\frac{\cD^{\sigma}}{dt} V+(\cD^{\sigma}_{\gamma(t)} \cX_{\frakq}^{gC,\sigma})(V),\]
here $V(t)=(b(t),r(t),\psi(t))$ is a path in $\cC^{\sigma,I}(Y)$.

The formula here is different from those in \cite[Section 5.9]{Lidman2016TheEO}, since adding the real assumption makes a trajectory in the pseudo-temporal gauge the same as in temporal gauge. 

\begin{prop}\label{prop: Fredholm property of Q gc gamma}
    Let $x,y\in W^{\sigma,I}$ be non-degenerate critical points of $\cX^{\sigma}_{\frakq}$. Take any path $\gamma\in W^{\tau,I}_k(x,y)$. Then for $j\le k$, the operator \[\cD^{\tau}_{\gamma}\cF^{gC,\tau}_{\frakq}: \calK^{gC,\tau,I}_{j,\gamma}\to \cV^{gC,\tau,I}_{j-1,\gamma}\] is Fredholm. Furthermore, the Fredholm index is the same as that of  \[(\cD^{\tau}_{\gamma}\cF^{\tau}_{\frakq})|_{\calK^{\tau,I}_{j,\gamma}}:\calK^{\tau,I}_{j,\gamma}\to \cV^{\tau,I}_{j-1,\gamma} \]
\end{prop}
\begin{proof}
    We modify the proof of \cite[Proposition 5.9.2]{Lidman2016TheEO}. Let $Q_{\gamma}^{gC}$ be a shorthand for $\cD^{\tau}_{\gamma}\cF^{gC,\tau}_{\frakq}$. Since we have no tangent to the orbit on the real global Coulomb slice or the difference between pseudo-temporal and genuine temporal gauge, this can alternatively be seen as a map from $\cT^{gC,\tau,I}_{j,\gamma}$ to $\cT^{gC,\tau,I}_{j-1,\gamma}$. Although the picture is simpler than the usual case due to these two reasons, $\cT^{gC,\tau,I}_{j,\gamma}$ is still not a space of invariant sections for some vector bundle, since we have added the slicewise global Coulomb gauge requirement. 

    To remedy this, we extend the operator as follows. Define a linear operator \[R=\frac{d}{dt}+\begin{bmatrix}
        0& -d\\
        -d^*& 0\\
    \end{bmatrix}: (\mathrm{im} d_{(0)}\oplus \mathrm{im} d^*_{(1)})^{-\iota^*} \to (\mathrm{im} d_{(0)}\oplus \mathrm{im} d^*_{(1)})^{-\iota^*}.\] Here, we follow the convention in \cite[Proposition 5.9.2]{Lidman2016TheEO}, i.e., we use a subscript $(p)$, $p\in\{0,1\}$ to denote the the imaginary $p$-forms on which those operators acts slicewise on $\{t\}\times Y$. Since $Y$ is a real rational homology sphere, $(\mathrm{im} d^*_{(1)})^{-\iota^*}$ consists of real functions on $Z$ (real condition automatically implies that it integrates slicewise to 0), and $(\mathrm{im} d_{(0)})^{-\iota^*}=(\mathrm{ker} d_{(1)})^{-\iota^*}$.

    Decompose $\cT^{\tau,I}_{j,\gamma}$ into $\cT^{gC,\tau,I}_{j,\gamma}\oplus(\cJ^{\circ,\tau,\I}_{j,\gamma}\oplus (\mathrm{im} d^*_{(1)})^{-\iota^*})$, where $\cJ^{\circ,\tau,\I}_{j,\gamma}$ consists of time dependent elements in $\cJ^{\circ,\sigma,\I}_{j,\gamma(t)}=\{(-d\xi,0,\xi \phi(t))|\int_{Y}\xi=0\}\subset \cJ^{\sigma,I}_{j,\gamma(t)}$. We have a natural identification $\Psi:(\mathrm{im} d_{(0)})^{-\iota^*}\to \cJ^{\circ,\tau,\I}_{j,\gamma}$ given slicewise by \[\Psi(-d\xi,0,0)=(-d\xi,0,\xi\cdot \phi(t)).\] We conjugate $R$ by $\Psi$ to get \[\widehat{R}:\Psi\circ \frac{d}{dt}\circ\Psi^{-1}+\begin{bmatrix}
        0& \bfd^\sigma\\
        -d^*& 0\\
    \end{bmatrix}: \cJ^{\circ,\tau,\I}_{j,\gamma} \oplus (\mathrm{im} d^*_{(1)})^{-\iota^*} \to \cJ^{\circ,\tau,\I}_{j,\gamma} \oplus (\mathrm{im} d^*_{(1)})^{-\iota^*} .\] 

We extend $Q_{\gamma}^{gC}$ to $\widehat{Q}_{\gamma}^{gC}=\begin{bmatrix}
    Q_{\gamma}^{gC} &0\\
    0& \widehat{R}\\
\end{bmatrix}: \cT^{\tau,I}_{j,\gamma}\to \cT^{\tau,I}_{j-1,\gamma}$. 

$\cT^{\tau,I}_{j,\gamma}$ has an alternative decomposition $\cV^{\tau,I}_{j,\gamma}\oplus L^2(Z;i\R)^{-\iota^*}$, where $\cV^{\tau,I}$ is the space of real four-dimensional configurations with trivial $dt$ component and slicewise being in $\cT^{\sigma,I}$. With respect to this new decomposition, we rewrite $\widehat{Q}_{\gamma}^{gC}$ as \[\widehat{Q}_{\gamma}^{gC}=\frac{D^{\sigma}}{dt}+\begin{bmatrix}
    M &0\\
    0& 0\\
\end{bmatrix}+\begin{bmatrix}
    H &\bfd^{\sigma}_{\gamma(t)}\\
    \bfd^{\sigma,\tilde{\dagger}}_{\gamma(t)} & 0\\
\end{bmatrix}.\]  Here, $M$ and $H$ have explicit formulas: $M$ acts as zero on $\cT^{gC,\sigma, I}_{j,\gamma(t)}\subset \cT^{\sigma, I}_{j,\gamma(t)}$ and for the summand $\cJ^{\circ,\sigma, I}_{j,\gamma(t)}\subset \cT^{\sigma, I}_{j,\gamma(t)}$,
\[M(-d\xi,0,\xi\phi)=\Pi^{\perp}_{\phi}(0,0,\xi\frac{d\phi}{dt}).\] Using the decomposition $\cT^{\sigma, I}_{j,\gamma(t)}=\cT^{gC,\sigma, I}_{j,\gamma(t)}\oplus \cJ^{\circ,\sigma,I}_{j,\gamma(t)}$, \[H=\begin{bmatrix}
    \cD_{\gamma(t)}^{\sigma}\cX^{gC,\sigma}_{\frakq} & 0\\
    0 & 0\\
\end{bmatrix},\] so the third term is exactly $\cH_{\gamma(t)}^{\sigma}$ in Lemma \ref{lem:operator cH and kASAFOE property of gextended Hessian on blow up}. 

Now we observe that $\widehat{Q}_{\gamma}^{gC}$ can be written in the form \[\widehat{Q}_{\gamma}^{gC}=\frac{d}{dt}+L_0+\hat{h}_{t}^{gC},\]  which appears in the proof of \cite[Proposition 8.13] {li2022monopolefloerhomologyreal} and \cite[Proposition 14.4.3]{Kronheimer_Mrowka_2007}, see also \cite[Proposition 5.9.1]{Lidman2016TheEO}. The time dependence of $L_0+\hat{h}_{t}^{gC}$ makes it different from $\cH_{\gamma(t)}^{\sigma}$ in Lemma \ref{lem:operator cH and kASAFOE property of gextended Hessian on blow up}, but as $t$ goes to $\infty$, we have $L_0+\hat{h}_{\pm\infty}^{gC}=\widehat{\mathrm{Hess}}^{\tg,\sigma}_{\frakq,\pm}$. These limits are hyperbolic under our non-degenerate assumption on stationary points and the lemmas from previous subsection. With the help of \cite[Proposition 8.10]{li2022monopolefloerhomologyreal}, the interpolation argument in \cite[p. 93]{Lidman2016TheEO} still works for our real version of $\widehat{Q}_{\gamma}^{gC}$ operator, showing that it is Fredholm and has the same index as the old operator $Q_{\gamma}$ considered in Subsection \ref{subsub:Perturbations}.

As in \cite[Proposition 5.9.2]{Lidman2016TheEO}, we see that $\hat{R}$ is bijective, so it has no contribution to the index. Our proof is complete, since we have no $\bfd$ factors in the definition of $Q^{gC}_{\gamma}$.
\end{proof}

As a summary, we have \begin{lem}\label{lem:equivalence of fredholm and surjectivity}
    Under the same hypothesis as in Proposition \ref{prop: Fredholm property of Q gc gamma},
    \begin{enumerate}
        \item the operators \[\cD^{\tau}_{\gamma}\cF^{gC,\tau}_{\frakq},Q_{\gamma}^{gC},
        \widehat{Q}_{\gamma}^{gC}\] are Fredholm of the same index;
        \item one of the operators 
        \[\cD^{\tau}_{\gamma}\cF^{gC,\tau}_{\frakq},Q_{\gamma}^{gC},
        \widehat{Q}_{\gamma}^{gC}\] is surjective if and only if the other two are surjective.
    \end{enumerate}
\end{lem}

The argument in \cite[Proposition 5.10.1]{Lidman2016TheEO} works with all those objects replaced by their real counterparts. In conclusion, we have the following:
\begin{prop}\label{prop: identification of regularity}
    Let $x,y\in \cC^{\sigma,I}(Y)$ be two stationary points of $\cX^{\sigma}_{\frakq}$ and $\gamma\in \cC^{\tau,I}_{k}(x,y) $ be a path in temporal gauge. Let $x^\flat=\Pigcs(x)$ and $y^\flat=\Pigcs(y)$ and $\gamma^\flat=\Pigcs(\gamma)$. Then \begin{enumerate}
        \item The operators  $\cD^{\tau}_{\gamma^\flat}\cF^{gC,\tau}_{\frakq}: \calK^{gC,\tau,I}_{j,\gamma^\flat}\to \cV^{gC,\tau,I}_{j-1,\gamma^\flat}$ and $(\cD^{\tau}_{\gamma}\cF^{\tau}_{\frakq})|_{\calK^{\tau,I}_{j,\gamma}}:\calK^{\tau,I}_{j,\gamma}\to \cV^{\tau,I}_{j-1,\gamma}$ have the same Fredholm index.
        \item Suppose that $\gamma$ is a trajectory of $\cX_{\frakq}^{\sigma}$, such that $[\gamma^{\flat}]\in \cB^{gC,\tau,I}_{k}([x],[y])$ is a trajectory of $\cX_{\frakq}^{agC,\sigma}$. If $(\cD^{\tau}_{\gamma}\cF^{\tau}_{\frakq})|_{\calK^{\tau,I}_{j,\gamma}}$ is surjective, so is $\cD^{\tau}_{\gamma^\flat}\cF^{gC,\tau}_{\frakq}$.
    \end{enumerate}
\end{prop}

%This improves the homeomorphisms between moduli spaces mentioned above to diffeomorphisms. 
It is important to note that this argument works for reducible moduli spaces as well. 

\subsection{Gradings and $v$-action}\label{sub: Gradings and U-action}

Throughout this subsection, we assume that we have chosen an admissible perturbation $\frakq$, so that all the stationary points are non-degenerate and all the moduli spaces $M([x],[y])$ and $M^{red}([x],[y])$ are regular.

The relative grading can be identified as follows using the index equalities from the previous subsection. 

\begin{align*}
    \mathrm{gr}(x,y)&= \mathrm{ind}(\cD^{\tau}_{\gamma}\cF^{\tau}_{\frakq})|_{\calK^{\tau,I}_{j,\gamma}}=\mathrm{ind}Q_{\gamma}\\
   & = \mathrm{ind}\cD^{\tau}_{\gamma^{\flat}}\cF^{gC,\tau}_{\frakq}=\mathrm{ind}Q_{\gamma^\flat}^{gC}\\
\end{align*}

\begin{lem}(\cite[Lemma 5.11.1]{Lidman2016TheEO})
    Let $x=(a,s,\phi)\in W^{\sigma,I}_k$, then \[\mathrm{Re}\left \langle \widetilde{\cX^{gC}_\frakq}^1(a,s,\phi),\phi \right \rangle_{L^2}=\mathrm{Re}\left \langle \widetilde{\cX_\frakq}^1(a,s,\phi),\phi \right \rangle_{L^2}\]
\end{lem}

This lemma tells us that the spinorial energy $\Lambda_\frakq$ works in the global Coulomb slice as well.
As in \cite[Section 5.11]{Lidman2016TheEO}, we can analyze the eigenvalue at reducible stationary points directly and show that 

\begin{lem}
    Fix a reducible stationary point $(a,0)\in W^{I}$ of $\cX^{gC}_{\frakq}$. For each $N\in\N$, there exists $\omega_1,\omega_2>0$ such that among all stationary points of $\cX^{agC,\sigma}_{\frakq}$ that agree with $(a,0)$ after blowing down, those having grading in $[-N,N]$ are precisely the reducible stationary points with spinorial energy in $[-\omega_1,\omega_2]$.
\end{lem}

Next, we define a $v$-action on the homology group by cutting down moduli spaces of the form $M^{agC}([x],[y])$ or $M^{agC,red}([x],[y])$. 

Recall that we have considered a $v$-action on the original $\widecheck{\mathit{HMR}}$ group defined by intersecting the moduli spaces with a generic section of the line bundle on $\cB^{\sigma,I}([x],[y])$ associated to the gauge group evaluation $ev_p:\cG^I\to\Z_2$ for $p\in \mathrm{fix}(\iota)$.

On $W^{\sigma,I}_k/\Z_2$, we have a natural real line bundle $L^{agC}$ associated to the quotient map $W^{\sigma,I}_k\to W^{\sigma,I}_k/\Z_2$. By cutting the moduli spaces using a generic section of this line bundle, we obtain a $v$-action on the new $\widecheck{\mathit{HMR}}$, which we will define using stationary points and trajectories in the global Coulomb slice.

In the following, we will identify these two actions by identifying the corresponding cut-down moduli spaces based on the existing identification between stationary points and trajectories. To see this, we need a modification pf the $v$-action on the original $\widecheck{\mathit{HMR}}$. Let $p$ be the chosen point on the fixed set at which we perform the evaluation. The $v$-action can be alternatively defined by intersecting the moduli spaces with a transverse section of the line bundle $L^{\sigma}$ associated to the evaluation $\cG_{k+1}^I(N_p)\to \Z_2$, $u\mapsto u(p,0)$ for $N_{p}=[-1,1]\times Y \subset\R \times Y$. The unique continuation results identify this with the one defined in \cite[Section 13]{li2022monopolefloerhomologyreal}. 

Recall that for $u_{\pm}\in\cG_{\pm}^I(Y)$, we have $u_{\pm}=\pm e^{f}$ for some real function $f:Y \to i\R$. The same analysis holds for $Z=\R\times Y$, since they share the same homotopy type. In view of this, the second step of modification in \cite[Section 5.12]{Lidman2016TheEO} is not needed for us, since the average value of a real function using an equivariant metric is always zero.

Next, we move from the $\sigma$-model to the $\tau$-model as they did. There is a restriction map $r$ from an open subset $\cU\subset\cB^{\sigma,I}_{k}(N_p)$ to $\cB^{\tau,I}_{k}(N_{p})$. Again by the unique continuation, the moduli spaces $M([x],[y])$ and $M^{red}([x],[y]),$ restricted to $N_p$, yield configurations in $\cU$, which map homoemorphically onto their image under $r$. The evaluation map above also yields a line bundle $L^{\tau}$ over $\cB^{\tau,I}_{k}(N_{p})$. Its pullback by $r^*$ coincides with $L^{\sigma}|_{\cU}$. Thus, the $v$-action can also be defined by cutting the moduli space by a generic section of $L^{\tau}$.

After these modifications, we are ready to relate this to the line bundle $L^{agC}$ over $W^{\sigma,I}_{k}$. Recall that we have the global Coulomb projection $\Pi^{[gC],\tau}$ for configurations on $\R\times Y$. The same formula applies to $ [-1,1]\times Y$ as well. We use the same notation for the map $\Pi^{[gC],\tau}:\cB^{\tau,I}_k(N_p)\to \cB^{gC,\tau,I}_k(N_p)$.

We have a natural ``evaluation map'' \[ \cB^{gC,\tau,I}_k(N_p) \to W^{\sigma,I}_{k-1/2},\text{ } [\gamma]\mapsto [\gamma(0)].\] 

Pick any generic section $\zeta^{agC}$ of $L^{agC}$ such that it is transverse to the zero section as well as all the moduli spaces $M^{agC}([x],[y])$ and $M^{agC,red}([x],[y])$. Let $\mathscr{Z}^{agC}$ be the zero set of $\zeta^{agC}$. The intersections \[M^{agC}([x],[y])\cap \scrZ^{agC} \text{ and } M^{agC,red}([x],[y])\cap \scrZ^{agC}\] will contribute to the $v$-action. (Here, we implicitly identify the moduli spaces in $L^2_k$ with their image in $L^2_{k-1/2}$.) They could also be identified with their image in $\cB^{gC,\tau,I}$ under restriction. 

We pull back $L^{agC}$ on $W^{\sigma,I}/\Z_2$ to $L^{agC,\tau}$ over $\cB^{gC,\tau,I}_{k}(N_p)$. Then $\zeta^{agC}$ pulls back to a section of $L^{agC,\tau}$, and the intersection of the zero set with the moduli spaces can be identified via these pull-backs. A key observation is that $L^{\tau}$ is exactly the pull-back of $L^{agC,\tau}$ via $\Pi^{[gC],\tau}$. Pull back the section again to obtain $\zeta^{\tau}$, which can be used to define the $v$-action on the original $\widecheck{\mathit{HMR}}$. We have the following commutative diagram with injective horizontal maps:
	\[\xymatrix{M([x],[y]) \ar[d]^{\Pi^{[gC],\tau}} \ar[r] & \cB^{\tau,I}_{k}(N_p) \ar[d]^{\Pi^{[gC],\tau}}\\
		M^{agC}([x],[y]) \ar[r] & \cB^{gC,\tau,I}_{k}(N_p)\\} \]

Recall that the left vertical map is a homeomorphism, so we have an identification between cut-down moduli spaces \[M([x],[y])\cap \scrZ^{\tau} \cong M^{agC}([x],[y])\cap \scrZ^{agC}\] and a similar reducible version. This concludes the identification of $v$-actions.
 
\subsection{New definition of $\mathit{HMR}$}\label{sub:New definition of HMR}

The preparations in the previous subsections allow us to rephrase $\widecheck{\mathit{HMR}}$ in terms of configurations in the global Coulomb slice. This is done as follows.

Fix an admissible perturbation $\frakq$. We take boundary-stable and irreducible stationary points of $\cX^{agC,\sigma}_{\frakq}$ in $W^{\sigma,I}/\Z_2$ as generators of $\widecheck{CMR}$. These are in one-to-one correspondence with $\frC^{o,I}\cup\frC^{s,I}$. Using Lemma \ref{lem:identification of non-degeneracy}, we know that all these stationary points are non-degenerate.

We use moduli spaces $M^{agC}([x],[y])$ and their reducible counterparts to define the differential on $\widecheck{CMR}$. They are regular due to Proposition \ref{prop: identification of regularity} and the admissible assumption on $\frakq$. We also know that they are homeomorphic to the original ones. This identifies the differentials. 

Finally, in Subsection \ref{sub: Gradings and U-action}, we have exhibited equalities for relative gradings as well as homeomorphisms between the cut-down moduli spaces. These identify the gradings and module structures.

Combining all these, we get a reformulation of $\widecheck{\mathit{HMR}}$ in the global Coulomb slice. 

\section{Relating finite and infinite dimensional Morse homologies}\label{sec:Relating finite and infinite dimensional Morse homologies}

\subsection{Finite dimensional approximations}
The two main distinctions between real monopole Floer homology and real Seiberg-Witten Floer homotopy type are \begin{enumerate}
    \item One works on the whole configuration space, the other works in the global Coulomb slice;
    \item One uses perturbation, the other does not but cut off the configuration space using eigenspaces. 
\end{enumerate}

In the previous section, we reconstructed real monopole Floer homology in the global Coulomb slice. Now, we want a version of real Seiberg-Witten Floer homotopy type with perturbation.

In \cite[Section 6.1]{Lidman2016TheEO}, they introduced a new property \emph{very compact} (\cite[Definition 6.1.1]{Lidman2016TheEO}) on self-maps of the global Coulomb slice $W$ to characterize perturbations that are allowed for finite-dimensional approximation. The same definition works with $W$ replaced by $W^{I}$. We choose not to review this definition, since any $I$-equivariant very compact map  $W\to W$  gives rise to a very compact map $W^{I}\to W^{I}$ and we will only need maps of this form. We restate some propositions in the real setup for future references.

\begin{prop}(\cite[Proposition 6.1.2]{Lidman2016TheEO})\label{prop:properties of very compact maps}
Let $\eta:W^{I}\to W^{I}$ be a very compact map. Fix $k\ge 5$. Suppose that there exists a closed, bounded subset $N$ of $W_k^I$ with an open subset $U\subset N$ satisfying that all finite trajectories contained in $N$ are actually contained in $W^I\cap U$. Then \begin{enumerate}
    \item For $\lambda\gg 0$, trajectories of $l+p^{\lambda}\eta$ contained in $N$ must be contained in $U$;
    \item We can define the Floer spectrum $\Sigma^{-\mathrm{dim}(V^{0}_{-\lambda})\R} \Sigma^{-\mathrm{dim}(U^{0}_{-\lambda} )\Tilde{\R}} I^{\lambda}$ as in Subsection \ref{subsub:The Conley index and the real Seiberg-Witten Floer spectrum}. And this is independent of $\lambda$ up to stable equivalence.
    \item Further, if $\eta$ is $\Z_2$-equivariant, then the Floer spectrum can be defined equivariantly and is well-defined up to $\Z_2$-equivariant stable equivalence.
\end{enumerate}
    
\end{prop}
\begin{lem}(\cite[Lemma 6.1.4]{Lidman2016TheEO})\label{lem:Compactness of trajectories with varying lambda}
Let $I\subset \R$ be a closed interval (we allow it to be $\R$.) Under the hypothesis of Proposition \ref{prop:properties of very compact maps}, suppose that we have a sequence of numbers $\lambda_n\to \infty$ and a sequence of trajectories $\gamma_n$ of $l+p^{\lambda_n}\eta$ that lie entirely in $N$. Then there exists a subsequence of $\gamma_n$ for which the restriction to any compact subinterval $I'\subset I^{\circ}$ converges in $C^{\infty}$ topology of $W^I(I\times Y)$ to a trajectory of $l+\eta$.
    
\end{lem}
\begin{prop}(\cite[Proposition 6.1.5]{Lidman2016TheEO})
If $\frakq$ is a very tame perturbation, then $\Pigcs_* \frakq: W^{I}\to W^{I}$ is very compact.
    
\end{prop}
\begin{prop}(\cite[Proposition 6.1.6]{Lidman2016TheEO})\label{prop:identification of perturbed and unperturned homotopy type}
    Let $\frakq$ be a very tame perturbation. Then $\mathit{SWF}_{\Z_2,\frakq}(Y,\iota,\s)$ and $\mathit{SWF}_{\Z_2}(Y,\iota,\s)$ are $\Z_2$-equivariantly stably homotopy equivalent.
\end{prop}
\begin{proof}
    This is the same as the proof of metric independence in \cite[Section 3]{Konno_Miyazawa_Taniguchi_2025}. We can interpolate between $0$ and $\frakq$ in such a way that the while family all fits into Proposition \ref{prop:properties of very compact maps}. Then we use the fact that the Conley index is invariant under perturbation. 
\end{proof}
\subsection{Outline of the proof}\label{sub:Outline of the proof}

We will work with $\mathit{SWF}_{\Z_2,\frakq}$ that is closer to the real monopole Floer homology due to the existence of perturbation. Now, we still need to relate $l+c_{\frakq}$ to $l+p^{\lambda}c_{\frakq}$ on $W^{\lambda,I}$. We shall consider a vector field on $W_k$ defined by taking finite-dimensional approximation of the non-linear part of $\cX_{\frakq}$: \[\cX^{gC}_{\frakq^{\lambda}}=l+p^{\lambda}c_{\frakq}=l+c+\eta_{\frakq}^{\lambda},\]
where $\eta_{\frakq}^{\lambda}=p^{\lambda}c_{\frakq}-c$. It was shown that $\eta_{\frakq}^{\lambda}$ is a very compact map $W\to W$. Since it is equivariant for $I$ (recall that we use $\frakq$ defined by pairing with invariant forms and spinors), it restricts to a very compact map $W^{I}\to W^{I}$. $\cX^{gC}_{\frakq^{\lambda}}$ induces $\cX^{gC,\sigma}_{\frakq^{\lambda}}$ on the blow-up $W^{\sigma,I}$ and can be further pushed forward to $W^{\sigma,I}/\Z_2$.

With the notations ready, we now outline a proof for Theorem \ref{thm:main theorem}, which is a straightforward modification of \cite[Section 6.2]{Lidman2016TheEO}.

Let $\frC_{[-N,N]}$ be the set of stationary points of $\xagcsq$ in $W^{\sigma,I}/\Z_2$ with grading in the range $[-N,N]$. Their $\Z_2$-orbits form the set of stationary points of $\xgcsq$ in $W^{\sigma,I}$. Let $\frD_{[-N,N]}$ be the union of their orbits. By the compactness result in \cite{li2022monopolefloerhomologyreal}, we can choose $N$ large enough so that the projection of $\frD_{[-N,N]}$ to $W^I$ contains all the stationary points of $\cX^{gC}_{\frakq}$. This implies that $\frC_{[-N,N]}$ contains all the irreducibles. We may further assume that all boundary stable reducibles have gradings bigger than $-N$ by choosing $N$ large enough. Under these assumptions, all generators of $\widecheck{CMR}_{\le N}(Y,\iota,\s,\frakq)$ lie in $\frC_{[-N,N]}$. Note that its homology agrees with $\widecheck{\mathit{HMR}}$ at degrees less than $N$.

Let \[\cN=\{x\in W^{\sigma,I}_k|d_{L^2_k}(x,\frD_{[-N,N]})\le 2\delta\},\] where $d_{L^2_k}$ denotes the $L^2_k$ distance and $\delta>0$ is sufficiently small so that the only stationary points in $\cN$ are those in $\frD_{[-N,N]}$. Similarly, let \[\cU=\{x\in W^{\sigma,I}_k|d_{L^2_k}(x,\frD_{[-N,N]})< \delta\}\subset \cN.\] Then $\cN/\Z_2$ and $\cU/\Z_2$ are closed and open neighborhoods of $\frC_{[-N,N]}$ in $W^{\sigma,I}/\Z_2$.

Using this, we construct a chain complex $\widecheck{CMR}^{\lambda}$ determined by $\xagcsql$. It has stationary points of $\xagcsql$ in $\cN/\Z_2$ as generators and its differential is defined analogously to the original real monopole Floer homology by counting invariant trajectories. This chain complex actually comes from the Morse-Smale quasi-gradient flow of $\xagcsql$ for $\lambda$ sufficiently large. We will see an explicit identification between $\widecheck{CMR}^{\lambda}$ and $\widecheck{CMR}_{\le N}(Y,\iota,\s,\frakq)$ with the help of the inverse function theorem. Using the description of  $\widecheck{CMR}_{\le N}(Y,\iota,\s,\frakq)$ in the global Coulomb slice, we shall also be able to identify gradings and $v$-actions.

On the other hand, $\widecheck{CMR}^{\lambda}$ can also be identified with a truncation of a Morse complex for $B(2R)\cap W^{\lambda,I}$. Then as in Subsection \ref{sub:Morse homology and Morse quasi gradient flow}, the homology of this Morse complex is isomorphic to $\widetilde{H}^{\Z_2}_{\le M}(\mathit{SWF}_{\Z_2,\frakq}(Y,\iota,\s))$ for some $M$. By taking a sufficiently large $\lambda$, we can assume $M>N$ and obtain \[\widecheck{\mathit{HMR}}_{\le N-1}(Y,\iota,\s)\cong \widetilde{H}^{\Z_2}_{\le N-1}(\mathit{SWF}_{\Z_2,\frakq}(Y,\iota,\s)). \] Finally, we show in Subsection \ref{sub:Main thm} that $M$ and $N$ tend to $\infty$ when $\lambda$ does. This concludes the proof.

We have remarked in Subsection \ref{sub:Construction in W^I} that all real configuration spaces or slices can be embedded into the original ones, as well as their blow-ups. In view of this, the results in \cite[Section 6.3]{Lidman2016TheEO} work for us without change. So we won't repeat the results from that section and we will quote them when needed. 
 
\subsection{Convergence of stationary points}\label{sub:Convergence of stationary points}
Starting from now, we shall fix the following:
\begin{itemize}
    \item a very tame, admissible perturbation $\frakq$;
    \item a Sobolev index $k\ge 5$;
    \item a radius $R>0$, so that all real stationary points and finite type trajectories live in $B(2R)\subset W^{I}_k$;
    \item a number $N>0$ and closed/ open subset $\cN$/ $\cU$ as in Subsection \ref{sub:Outline of the proof}. We assume additionally that projection of $\cN$ in $W_{k}^{I}$ is contained in $B(2R)$, and $N$ is large enough to contain each reducible critical point $(a,0,\phi)$ of $\xagcsq$, where $\phi$ is the unit eigenvector of $D_a$ with smallest positive eigenvalue.
    \item a strict bound $\omega$ on the absolute value of spinorial energy of points in $\cN$.
    
\end{itemize}
\subsubsection{Convergence and Stability}

By applying Lemma \ref{lem:Compactness of trajectories with varying lambda} to the constant trajectories, we have the following:
\begin{lem}\label{lem:compactness of stationary point on blow-down}
    Suppose $x_n$ is a sequence of stationary points of $\xgcqln$ in $B(2R)\subset W^{I}_k$, where $\lambda_n\to \infty$. Then, there is a subsequence that converges to $x\in W^{I}_k$, a stationary point of $\xgcq$. If all $x_n$ are reducible, then $x$ is reducible.
\end{lem}

Now we want to reproduce this result on the blow-up configuration space, which requires properties of the controlled Coulomb perturbations.
\begin{lem} (\cite[Lemma 7.1.2]{Lidman2016TheEO})
    Fix $\epsilon>0$. There is some $b\gg 0$ such that for all $\lambda>b$, we have: If $x\in \cN\subset W^{\sigma,I}$ is a zero of $\xgcsql$, then there exists $x'\in \cN$ such that $\xgcsq(x')=0$ and $x,x'$ have $L^2_k$ distance at most $\epsilon$ in $L^{2}_k(Y;iT^{*}Y)^{-\iota^*}\oplus \R\oplus L^2_k(Y;\bS)^{\tau}$.
\end{lem}
\begin{proof}
    This can be proved (by contradiction) in exactly the same way as \cite[Lemma 7.1.2]{Lidman2016TheEO} by a combined use of Lemma \ref{lem:compactness of stationary point on blow-down} and energy control from \cite[Section 6.3]{Lidman2016TheEO}.
\end{proof}

Using this, we have the following:\begin{cor}
    For $\lambda\gg 0$, if $x$ is a zero of $\xgcsql$ in $\cN$, then $x$ is actually in $\cU$ and $\vert \Lambda_{\frakq^\lambda} (x)\vert\le \omega$. Here, $\Lambda_{\frakq^\lambda}$ is defined using the same formula as $\Lambda_{\frakq}$, see \cite[Equation 166]{Lidman2016TheEO}.
\end{cor}
\begin{cor}
    For $\lambda\gg 0$, all the stationary points of $\xgcsql$ in $\cN$ live inside the finite-dimensional blow-up $(W^{\lambda,I})^{\sigma}$. 
\end{cor}

As in \cite[Section 7.2]{Lidman2016TheEO}, we list the eigenvalues of $l$ (with multiplicities) by $\lambda_n$ according to their absolute values and pick a homeomorphism $f: (0,\infty]\to [0,1)$ with the following properties:
\begin{enumerate}
    \item $f|_{(0,\infty)}$ is a strictly decreasing and takes value in $(0,1)$.
    \item $\lim_{n\to \infty}\vert \lambda_{n}\vert^2 f(\vert \lambda_{n+1}\vert) =\infty$.
\end{enumerate}
This allows us to analyze the change of $\lambda$ from a different point of view.

\begin{lem} (\cite[Lemma 7.2.1]{Lidman2016TheEO})
The map \[h:W_{k}^I \times (-1,1)\to W^I_{k-1},\text{ } h(x,r)=x-p^{f^{-1}(\vert r\vert)}(x)\]   is continuously differentiable with $\cD h_{(x,0)}(0,1)=0$ for all $x$. 
\end{lem}
Let $[x_0]\in W^{\sigma,I}/\Z_2$ be a non-degenerate irreducible stationary point of $\xagcsq$. The non-degeneracy assumption implies that \[\cD_{[x_0]}(\xagcsq): \calK^{agC,\sigma,I}_{k,[x_0]}\to \calK^{agC,\sigma,I}_{k-1,[x_0]}\] is an invertible linear operator.
\begin{prop} (\cite[Proposition 7.2.2]{Lidman2016TheEO})
    Let $[x]\in W^{\sigma,I}/\Z_2$ be a non-degenerate irreducible stationary point of $\xagcsq$. Then for any small enough neighborhood $U_{[x]}$ of $[x]$ in $W^{\sigma,I}/\Z_2$ and large enough $\lambda$, there is a unique $[x_{\lambda}]\in U_{[x]}$ such that $\xagcsql([x_{\lambda}])=0$.
\end{prop}
\begin{proof}
    As for \cite[Proposition 7.2.2]{Lidman2016TheEO}, the homeomorphism $f$ allows us to construct a vector field on $W^{\sigma,I}/\Z_2\times (-1,1)$. Then we double this Banach manifold and apply the inverse function theorem (with the help of non-degeneracy assumption) to prove the desired result.
\end{proof}

Now we introduce the notation $\frC_{\cN}$ as a short hand for $\frC_{[-N,N]}$, and $\frC_{\cN}^{\lambda}$ for the set of stationary points of $\xagcsql$ in $\cN/\Z_2$. The propositions above tell us:
\begin{cor}\label{cor:1 to 1 correspondence between stationary points of cutoff}
    For $\lambda\gg 0$, there is a one-to-one correspondence \[\Xi_{\lambda}: \frC_{\cN}^{\lambda}\to \frC_{\cN}.\] and this correspondence preserves the type of stationary points.
\end{cor}

From now on, we use $[x_{\infty}]$ for stationary points of $\xagcsq$
and $[x_{\lambda}]$ for its preimage $\Xi_{\lambda}^{-1}([x_{\infty}])$. Note that the implicit function theorem tells us that $[x_{\lambda}]$ is smooth in $\lambda$, provided it is sufficiently large.

Besides those stationary points in $\cN$, we have some extra control on reducible stationary points in $(B(2R)\cap W^{\lambda,I})^\sigma$ but not necessarily in $\cN$.
\begin{lem} (\cite[Lemma 7.2.4]{Lidman2016TheEO})
    Fix $\epsilon>0$. For $\lambda\gg 0$, we have a one-to-one correspondence in $B(2R)$ between reducible stationary points $x_{\infty}$ of $\xgcq$ and reducible stationary points $x_{\lambda}$ of $\xgcql$. Moreover, $x_{\lambda}$ is $\epsilon$-close to $x_{\infty}$ if $\lambda$ is large enough. 
\end{lem}
\subsubsection{Hyperbolicity}
Recall that an operator is called \emph{hyperbolic} if its complexification has a spectrum away from the imaginary axis. 

We have seen in Lemma \ref{lem:Fredholm properties of g-extended hessian on blow-up} that $\mathrm{Hess}^{\tg,\sigma}_{\frakq,x}$ is invertible with real spectrum when $x$ is a non-degenerate stationary point. Hence, it is by definition a hyperbolic operator. Now we introduce Hessians for the cut-off flow: \[\mathrm{Hess}^{\tg,\sigma}_{\frakq^{\lambda},x}=\Piagcs_x \circ\cD^{\tg,\sigma}_x \xgcsql :\calK^{agC,\sigma,I}_{k,x}\to \calK^{agC,\sigma,I}_{k-1,x}.\] When $x$ is a stationary point, the Hessian operator is independent of the connection chosen, so we can simplify the expression into \[\mathrm{Hess}^{\tg,\sigma}_{\frakq^{\lambda},x}=\Piagcs_x \circ\cD^{\sigma}_x \xgcsql .\] We shall say $x$ is a \emph{hyperbolic} stationary point when this operator is hyperbolic.

\begin{prop}(\cite[Proposition 7.3.1]{Lidman2016TheEO})\label{prop:byperbolicity in N}
    For $\lambda$ large enough, the stationary points of $\xgcsql$ inside $\cN$ are hyperbolic. Consequently, among stationary points of $\xagcsql|_{(W^{\lambda,I})^{\sigma}/\Z_2}$, those inside $\cN/\Z_2$ are all hyperbolic.
\end{prop}
\begin{proof}
    This can be proved in exactly the same way as \cite[Proposition 7.3.1]{Lidman2016TheEO}. One can observe that computations and analysis about Hessians in \cite[p122]{Lidman2016TheEO} hold without change for real operators, and that we can still use $l$ to define $L^{2}_k$ metrics, then argue by contradiction as they did.
\end{proof}

We again pay some attention to those reducible stationary points of $\xgcsql$ that lie in $B(2R)^{\sigma}$ but not in $\cN$. Such a reducible $(a,0,\phi)$
satisfies \[-*da=(p^{\lambda}c_{\frakq})^0(a,0)\text{ and } D_{\frakq^{\lambda},a}(\phi)=D\phi+(p^{\lambda})^1(\cD_{(a,0)}c_{\frakq}(0,\phi))=\kappa \phi\] for some $\kappa\in \R$. Though we have $(a,0)\in W^{\lambda,I}$ by assumption, $(a,0,\phi)$ is not necessarily in $(W^{\lambda,I})^{\sigma}$. 

Now we restrict the choice of $\lambda$ to the sequence $\lambda_1^{\bullet}<\lambda_2^{\bullet}<\ldots$ that we have fixed in Subsection \ref{subsub:Finite dimensional approximation}, so that $p^{\lambda}$ is a genuine $L^2$ orthogonal projection. Using such a $\lambda$, we can write \[D+(p^{\lambda})^1\cD_{(a,0)}c_{\frakq}(0,\cdot)=D+(p^{\lambda})^1\cD_{(a,0)}c_{\frakq}(0,\cdot)(p^{\lambda})^1\] in which the right-hand side is self-adjoint. Using this, we get a stronger version of Proposition \ref{prop:byperbolicity in N}.
\begin{prop} (\cite[Proposition 7.4.1]{Lidman2016TheEO})\label{prop:hyperbolicity in B(2R)}
    We can choose an admissible perturbation $\frakq$ so that for $\lambda\in\{\lambda_1^{\bullet}<\lambda_2^{\bullet}<\ldots\} $ sufficiently large, the restriction of $\xgcsql$ to $(B(2R)\cap W^{\lambda,I})^{\sigma}$ has only hyperbolic stationary points.
\end{prop}
\begin{proof}
    The strategy for \cite[Proposition 7.4.1]{Lidman2016TheEO} applies by noting that we have a real analogue \cite[lemma 7.13]{li2022monopolefloerhomologyreal} of \cite[Lemma 12.6.2]{Kronheimer_Mrowka_2007}.
\end{proof}

\subsection{Identification of gradings}\label{sub:Identification of grading}
We shall see in the next subsection that for $\lambda\gg 0$, $\xgcsql$ is a Morse-Smale equivariant quasi-gradient, so it can be used to define a  Morse chain complex as we have seen in Subsection \ref{sub:Morse homology and Morse quasi gradient flow}. In this part, we relate gradings of stationary points of $\xgcsq$ and $\xgcsql$. (We have a one-to-one correspondence between them in certain grading range provided that $\lambda$ is large enough.) A subtlety here is that we have two notions of grading for stationary points of $\xgcsql$, one from the infinite-dimensional manifold $W^{\sigma,I}/\Z_2$, the other from the finite-dimensional manifold $(B(2R)\cap W^{\lambda,I})^{\sigma}/\Z_2$. We shall identify them soon. 
\subsubsection{Relative grading}\label{subsub:Relative grading}
To define a relative grading between stationary points of $\xagcsql$ in $\cN/\Z_2$, we need to express the flow equation as a section and take its covariant derivative. Along trajectories, the choice of connection has no effect on the result, so we will use $\cD_{\gamma_{\lambda}}^{\tau}\cF^{gC,\tau}_{\frakq^{\lambda}}$.

\begin{lem}
    Fix $1\le j\le k$ and $\lambda\gg 0$. For any $[x_{\infty}]$ and $[y_{\infty}]\in\frC_{\cN}$ and each path $[\gamma_{\lambda}]\in \cB^{gC,\tau,I}_k([x_{\lambda}],[y_{\lambda}])$ with representative $\gamma_\lambda\in \cC^{gC,\tau,I}_k(x_{\lambda},y_{\lambda})$, the operator \[\cD_{\gamma_\lambda}^{\tau}\cF^{gC,\tau}_{\frakq^{\lambda}}:\calK^{gC,\tau,I}_{j,\gamma_{\lambda}}\to \calK^{gC,\tau,I}_{j-1,\gamma_{\lambda}}\] is Fredholm with index $\mathrm{gr}([x_{\infty}],[y_{\infty}])$.
\end{lem}
We remark that since we have no tangent to the orbit in the real case, we have no need to restrict the domain as in \cite[Lemma 9.1.1]{Lidman2016TheEO}.
\begin{proof}
We have proved Proposition \ref{prop: Fredholm property of Q gc gamma} in detail. The same argument applies here: we can introduce an extension $\widehat{Q}_{\gamma_{\lambda}}^{gC}$ of $\cD_{\gamma_\lambda}^{\tau}\cF^{gC,\tau}_{\frakq^{\lambda}}$ and show that it is Fredholm and that it shares the same index with $\cD_{\gamma_\lambda}^{\tau}\cF^{gC,\tau}_{\frakq^{\lambda}}$. So it remains to show that $\widehat{Q}_{\gamma_{\lambda}}^{gC}$ and $\widehat{Q}_{\gamma}^{gC}$ share the same index.

We can show that the indices of these operators are independent of the choice of $\gamma$ or $\gamma_{\lambda}$ by using a standard continuation argument. Then we shall use the interpolation argument as in the proof of \cite[Lemma 9.1.1]{Lidman2016TheEO} to prove the index identification using a specific choice of $\gamma$ and $\gamma_{\lambda}$.  That argument is quite long but standard, in particular, it applies to the real operators without change.
    
\end{proof}
We can now define a relative grading between stationary points of $\xagcsql$ in $\cN/\Z_2$ by \[\mathrm{gr}([x_{\lambda}],[y_{\lambda}])=\mathrm{ind}\cD_{\gamma_\lambda}^{\tau}\cF^{gC,\tau}_{\frakq^{\lambda}}.\]

\begin{cor}
    The correspondence $\Xi_{\lambda}: \frC_{\cN}^{\lambda}\to \frC_{\cN}$ from Corollary \ref{cor:1 to 1 correspondence between stationary points of cutoff} preserves relative grading.
\end{cor}

\begin{prop} (\cite[Proposition 9.1.4]{Lidman2016TheEO}) \label{prop:indentification between f.d. and inf.d gr}
   For stationary points $[x_{\lambda}]$ and $[y_{\lambda}]$ of $\xagcsql$ in $\cN/\Z_2$, $\mathrm{gr}([x_{\lambda}],[y_{\lambda}])$ is computed in infinite dimensional setup. It coincides with the difference of gradings of $[x_{\lambda}]$ and $[y_{\lambda}]$  considering as stationary points of $\xagcsql$ restricted to $(B(2R)\cap W^{\lambda,I})^{\sigma}/\Z_2$.
\end{prop}
\begin{proof}
    We follow the strategy of \cite[Proposition 9.1.4]{Lidman2016TheEO}. Since different gradings are defined using indices of different operators, it is necessary for us to relate those operators. Let $\gamma=(a(t),s(t),\phi(t)):\R \to (B(2R)\cap W^{\lambda,I})^{\sigma}$  be a path connecting stationary points $x_\lambda$, $y_\lambda$ in $(B(2R)\cap W^{\lambda,I})^{\sigma}$. We can associate to it an operator \[Q_{\gamma}^{gC}=\cD_{\gamma}^{\tau}\cF^{gC,\tau}_{\frakq^{\lambda}}: \cT^{gC,\tau,I}_{j,\gamma}(x_\lambda,y_\lambda) \to \cV^{gC,\tau,I}_{j,\gamma}(Z),\] which defines the relative grading for stationary points of $\xgcsql$ in the infinite-dimensional setting. This restricts to \[Q_{\gamma}^{gC,\lambda}:\cT^{gC,\tau,\lambda,I}_{j,\gamma}(x_\lambda,y_\lambda) \to \cV^{gC,\tau,\lambda,I}_{j,\gamma}(Z),\] for $\cT^{gC,\tau,\lambda,I}_{j,\gamma}(x_\lambda,y_\lambda)=\{(b,r,\psi)\in \cT^{gC,\tau,I}_{j,\gamma}(x_\lambda,y_\lambda)| (b(t),\psi(t))\in W^{\lambda,I} \text{ for all } t\}$ and $\cV^{gC,\tau,\lambda,I}_{j,\gamma}(Z)$ is defined similarly.

    We will show that $\mathrm{ind} Q_{\gamma}^{gC}= \mathrm{ind} Q_{\gamma}^{gC,\lambda}$ in the next lemma, so it suffices to see that $\mathrm{ind} Q_{\gamma}^{gC,\lambda}$ characterizes the relative grading in the finite-dimensional setup.

    This operator has both its domain and codomain consisting of paths $V(t)=(b(t),r(t),\psi(t)):\R\to T(W^{\lambda,I})^\sigma$. The observation about norm in the proof of \cite[Proposition 9.1.4]{Lidman2016TheEO} holds for real configurations: the $L^2_j$ norm of $V$ as a four-dimensional configuration is equivalent to its $L^2_j$ norm as a map from $\R$ to $W^{\lambda,I}$. We have already proved the hyperbolicity of stationary points for this vector field, so the relative index between them is well-defined. 

    We consider another operator\[\Piagcs\circ(\frac{D^{\sigma}}{dt}+\cD^\sigma \xgcsql):T_{j,\gamma}\cP(x_{\lambda},y_{\lambda})\to T_{j-1,\gamma}\cP(x_{\lambda},y_{\lambda})\] for $T_{j,\gamma}\cP(x_{\lambda},y_{\lambda})=L^2_{k}(\R,T(W^{\lambda,I})^\sigma)$. (Our expression is simpler than those in \cite[Proposition 9.1.4]{Lidman2016TheEO}, since in pseudo-temporal gauge is the same as in temporal gauge for real configurations.)  This can be used to define the Morse index in the finite-dimensional setup.

    Due to the vanishing of the tangent to the orbit on the global Coulomb slice and the $dt$ component in pseudo-temporal gauge, now \[\cT^{gC,\lambda,\tau,I}_{j,\gamma}=T_{j,\gamma}\cP(x_{\lambda},y_{\lambda})=\cV^{gC,\lambda,\tau,I}_{j,\gamma},\]so \[Q_{\gamma}^{gC,\lambda}=
    \frac{D^\sigma}{dt}+\cD^\sigma\xgcsql.\] The situation for us is far simpler than the one in \cite[Proposition 9.1.4]{Lidman2016TheEO} now. What we need to compare are just $\Piagcs\circ(\frac{D^{\sigma}}{dt}+\cD^\sigma \xgcsql)$
 and $\frac{D^{\sigma}}{dt}+\Piagcs\circ\cD^\sigma \xgcsql$. A straightforward linear interpolation finishes the proof. 
\end{proof}
\begin{lem}(\cite[Lemma 9.1.5]{Lidman2016TheEO})\label{lem:index of Q gamma gc}
     For $\lambda=\lambda_i^{\bullet}\gg 0$, the index of $Q^{gC}_{\gamma}$ is equal to that of $Q^{gC,\lambda}_{\gamma}$.
 \end{lem}
 \begin{proof}
     The proof for \cite[Lemma 9.1.5]{Lidman2016TheEO} works for us without any essential change. The main idea is that, by using a slicewise decomposition $W^I=W^{\lambda,I}\oplus (W^{\lambda,I})^\perp$, $Q^{gC,\lambda}_{\gamma}$ appears as a block of $Q^{gC}_{\gamma}$ when it is expressed as a lower-triangular matrix; then it can be shown by contradiction that the other diagonal block is invertible. This identifies the indices of the two operators in question.
 \end{proof}
 In the rest of this subsubsection, we will fix some $\lambda=\lambda_i^{\bullet}\gg 0$ and a reducible stationary point $(a,0)\in B(2R)$ of $\xgcql$. Consider a reducible critical point $[(a,0,\phi)]$ inside $(B(2R)\cap W^{\lambda,I})^{\sigma}/\Z_2$. Let $\kappa(\phi)$ be the corresponding eigenvalue of $D_{\frakq,a}$. We shall see in the next subsection that $\xagcsql$ is an equivariant Morse quasi-gradient. The relative grading can be computed as follows.
 \begin{lem}
     Let $[(a,0,\phi)]$ and $[(a,0,\phi')]$ be stationary points of $\xagcsql$ as above. Assume $\kappa(\phi)>\kappa(\phi')$. Then the relative grading between them computed from the finite-dimensional manifold $(B(2R)\cap W^{\lambda,I})^{\sigma}/\Z_2$ with vector field $\xagcsql$ is given by \[\mathrm{gr}([(a,0,\phi)],[(a,0,\phi')])=\begin{cases} i(\kappa(\phi),\kappa(\phi')) \text{, if the two eigenvalues have the same sign;}\\
     i(\kappa(\phi),\kappa(\phi'))-1 \text{, otherwise.}   
     \end{cases}\]
     Here $i(\kappa(\phi),\kappa(\phi'))$ denote the number of eigenvalues in between $\kappa(\phi)$ and $\kappa(\phi')$ plus one.
 \end{lem}
\begin{proof}
    This comes from classical Morse theory.
\end{proof}
\begin{lem}\label{lem:grading comparison in and outside N}
    Suppose that $[(a,0,\phi)]$ is a stationary point of $\xagcsql$ that lies in $(B(2R)\cap W^{\lambda,I})^{\sigma}/\Z_2$ but not in $\cN/\Z_2$.\begin{enumerate}
        \item If $\kappa(\phi)>0$, then for all stationary points of the form $[(a,0,\phi')]$ that are contained in $\cN/\Z_2$, we have $\mathrm{gr}([(a,0,\phi)],[(a,0,\phi')])\ge 1$.
        \item If $\kappa(\phi)<0$, then for all stationary points of the form $[(a,0,\phi')]$ that are contained in $\cN/\Z_2$, we have $\mathrm{gr}([(a,0,\phi)],[(a,0,\phi')])\le -1$.
    \end{enumerate}
\end{lem}
\begin{proof}
    We have done enough preparations, so the proof of \cite[Lemma 9.1.7]{Lidman2016TheEO} works for us.
\end{proof}
\begin{cor}\label{cor:identification of lowest positive reducible}
Let $x$ be a reducible stationary point of $\xagcsq$ that has the lowest positive eigenvalue among all reducible stationary points with some fixed connection component in $(B(2R)\cap W^{\lambda,I})^{\sigma}$. Then $[x_{\lambda}]$ is the reducible stationary point of $\xagcsql$ with the lowest positive eigenvalue among all reducible stationary points with a fixed connection component in $(B(2R)\cap W^{\lambda,I})^{\sigma}/\Z_2$.
\end{cor}
\subsubsection{Absolute grading}
In our main theorem, we claimed that the isomorphism between $\widetilde{H}^{\Z_2}_*(\mathit{SWF}_{\Z_2}(Y,\iota,\s))$ and $\widecheck{\mathit{HMR}}_*(Y,\iota,\s)$ respects the absolute grading when it is defined on both sides. Because our proving strategy uses $\widecheck{CMR}^{\lambda}$ as an intermediate step, we need to define an absolute grading for stationary points of $\xagcsql$ in $(B(2R)\cap W^{\lambda,I})^{\sigma}$. We will see that $\xagcsql$ is an equivariant quasi-gradient flow, so its Morse complex computes the reduced $\Z_2$-equivariant homology for $I^{\lambda}$ in a certain range. Recall that in Subsection \ref{sub:Real Seiberg-Witten homotopy type}, we defined $\mathit{SWF}_{\Z_2}(Y,\iota,\s)=\Sigma^{-\mathrm{dim}(V^{0}_{-\lambda})\R} \Sigma^{-(\mathrm{dim}(U^{0}_{-\lambda})+n^R(Y,\iota,\s,g) )\Tilde{\R}} I^{\lambda}$.  This motivates us to define \[\mathrm{gr}_{\lambda}^{\mathit{SWF}}([x_\lambda])=\mathrm{ind}([x_\lambda]\text{ in  } (W^{\lambda,I})^\sigma/\Z_2)-\mathrm{dim}(V^{0}_{-\lambda})-\mathrm{dim}(U^{0}_{-\lambda})-n^R(Y,\iota,\s,g).\] Equipped with this absolute grading, the Morse complex of $\xagcsql$ computes $\widetilde{H}^{\Z_2}_*(\mathit{SWF}_{\Z_2}(Y,\iota,\s))$ in a certain grading range (according to connectivity of $I^{\lambda}$.) Therefore, what we need to do is relate this with $\mathrm{gr}^{\Q}$ defined in Subsection \ref{subsub:Gradings}.
\begin{prop}
    For any $\lambda=\lambda_i^{\bullet}\gg 0$ and $[x]\in \frC$, we have \[\mathrm{gr}_{\lambda}^{\mathit{SWF}}([x_\lambda])=\mathrm{gr}^{\Q}([x]).\] 
\end{prop}
\begin{proof}
We follow the strategy of \cite[Proposition 9.2.1]{Lidman2016TheEO}.
Since we have already identified relative gradings in the previous subsection, we only need to show the absolute gradings of some specified generator are the same. We do this for $[x]=[(a,0,\phi)]\in \frC^{s,I}$, which corresponds to the lowest positive eigenvalue of $D_{\frakq,a}$. Pick a real $\mathrm{spin^c}$ bound $(W,\iota,\s)$ for $(Y,\iota,\s)$. (We abuse the same notations for three- and four-dimensional involutions and real $\mathrm{spin^c}$ structures.) Recall from Subsection \ref{subsub:Gradings} that only when such a $W$ exists, the $\mathrm{gr}^{\Q}$ on $\widecheck{\mathit{HMR}}$ is defined. In this case, \[\mathrm{gr}^{\Q}([x])=-\mathrm{gr}_z(W,[x])+\frac{1}{8}(c_1(\s)^2-\sigma(W))+\iota^R(W).\]
The formula here is a little different from that one in Subsection \ref{subsub:Gradings}, since we punctured $W$ to a cobordism from $S^3$ to $Y$ at that time, but now we regard $W$ as manifold with a single boundary component.
Recall that \[n^R(Y,\iota,\s,g)=\mathrm{ind}_{\C}(D^+_A)-1/8 (c_1^2(\s)-\sigma(X)),\]  and $\mathrm{ind}_{\C}(D^+_A)=\mathrm{ind}^{\tau}_{\R}(D^+_A)$, so what we want to show is \[\mathrm{gr}_z(W,[x])=\mathrm{ind}^{\tau}_{\R}(D^+_A)+b^1_{-\iota^*}(W)-b^+_{-\iota^*}(W)-b^0_{-\iota^*}(W)-\mathrm{ind}([x_\lambda]\text{ in  } (W^{\lambda,I})^\sigma/\Z_2)-\mathrm{dim}(V^{0}_{-\lambda})-\mathrm{dim}(U^{0}_{-\lambda}).\]

As in \cite[Lemma 28.3.2]{Kronheimer_Mrowka_2007}, we compute $\mathrm{gr}_z(W,[x])$ using a reducible configuration. Two parts contribute to this: the perturbed signature operator and the perturbed Dirac operator acting on $-\iota^*$-fixed and $\tau$-fixed parts of the original domains, respectively. The index of the former is $b^1_{-\iota^*}(W)-b^+_{-\iota^*}(W)-b^0_{-\iota^*}(W)$ if we ignore the perturbation. But now we need to deal with the difference, which is given by the index of the signature operator on $[0,1]\times Y$ with boundary data $(0,0)$ and $(\frakq,a)$. This can be computed from the spectral flow of the family \[\begin{bmatrix}
    0& -d^*\\
    -d& *d+2t\cD_{(ta,0)}\frakq^0\\
\end{bmatrix}:(\Omega^0(Y;i\R)\oplus\Omega^1(Y;i\R))^{-\iota^*}\to (\Omega^0(Y;i\R)\oplus\Omega^1(Y;i\R))^{-\iota^*},\text{ } t\in [0,1].\]

By taking a compact perturbation, we can reduce this to \[\begin{bmatrix}
    0& -d^*\\
    -d& *d+2t\cD_{(a,0)}\frakq^0\\
\end{bmatrix}, \text{ } t\in [0,1].\] 

Since $(a,0)$ is a stationary point in the blow-down, we have $\cD_{(a,0)}\frakq^0=\begin{bmatrix}
    0&0\\
    0& \cD_{(a,0)}\eta_{\frakq}^0
\end{bmatrix}$ when we decompose $-\iota^*$-invariant imaginary one forms into $(\mathrm{ker}d)^{-\iota^*}\oplus(\mathrm{ker}d^*)^{-\iota^*}$. Using this block form, only \[*d+2t\cD_{(a,0)}\eta_{\frakq}^0: (\mathrm{ker}d^*)^{-\iota^*}\to (\mathrm{ker}d^*)^{-\iota^*},\text{ } t\in [0,1]\]contributes to the spectral flow, we shall denote it by $\mathrm{SF}(\frakq)^0$.

The contribution from the perturbed Dirac operator is $\mathrm{ind}_{\R}^{\tau}(D_{\frakq,a}^+-\lambda_0)$ for $D_{\frakq,a}^+$, the APS operator with $D_{\frakq,a}$ on boundary and $\lambda_0$ the eigenvalue of $D_{\frakq,a}$ at $x$. Since $\lambda_0$ is the smallest positive eigenvalue, $D_{\frakq,a}^+$ shares the same domain with $D_A^+$ appear in the definition of $n^R$. Furthermore, $D_{\frakq,a}^+-\lambda_0$ and $D_{\frakq,a}^+$ differ by a constant, hence compact term, so they have the same real index. Note that \[\mathrm{ind}_{\R}^{\tau}(D_{\frakq,a}^+)-\mathrm{ind}_{\R}^{\tau}(D^+)=\mathrm{SF}(\frakq)^1,\] the spectral flow of the real part of perturbed Dirac operator on $Y$ moving from $(0,0)$ to $(\frakq,a)$. Since $(a,0)$ is a reducible stationary  point, we have $D^{gC}_{\frakq,a}=D_{\frakq,a}$, so this can be computed in Coulomb gauge.

In conclusion, we have \[\mathrm{gr}_z(W,[x])=\mathrm{ind}^{\tau}_{\R}(D^+_A)+b^1_{-\iota^*}(W)-b^+_{-\iota^*}(W)-b^0_{-\iota^*}(W)+\mathrm{SF}(\frakq)^0+\mathrm{SF}(\frakq)^1.\] 

It remains to show \[\mathrm{dim}(V^{0}_{-\lambda})+\mathrm{dim}(U^{0}_{-\lambda})-\mathrm{ind}([x_\lambda]\text{ in  } (W^{\lambda,I})^\sigma/\Z_2)=\mathrm{SF}(\frakq)^0+\mathrm{SF}(\frakq)^1.\]

We will analyze each term in this equation. $\mathrm{dim}(V^{0}_{-\lambda})+\mathrm{dim}(U^{0}_{-\lambda})$ is the number of eigenvalues of $l$ in $(-\lambda,0)$. The term $\mathrm{ind}([x_\lambda]\text{ in  } (W^{\lambda,I})^\sigma/\Z_2)$ is more complicated.  Let $x_{\lambda}=(a_{\lambda},0,\phi_{\lambda})$. By Corollary \ref{cor:identification of lowest positive reducible}, we know that $[x_{\lambda}]$ has the lowest positive eigenvalue among all stationary points of $\xagcsql$ in $(B(2R)\cap W^{\lambda,I})^\sigma$ with connection component $a_{\lambda}$. Two parts contribute to this index: the number of negative eigenvalues of the linearization of $l+p^{\lambda}c_{\frakq}$ restricted to connection part of $W^{\lambda,I}$, i.e., $*d+\cD_{(a_{\lambda,0})}(p^{\lambda}c_{\frakq})^0(\cdot,0)$ and the number of negative eigenvalues of the linearization of $l+p^{\lambda}c_{\frakq}$ restricted to spinorial summand of $W^{\lambda,I}$, i.e., $D+\cD_{(a_{\lambda,0})}(p^{\lambda}c_{\frakq})^1(0,\cdot)$. In summary, the left-hand side of the equation above is the spectral flow from $l$ to $l+p^\lambda A_{\lambda}$ for\[A_{\lambda}(b,\psi)=(\cD_{(a_{\lambda,0})}(p^{\lambda}c_{\frakq})^0(b,0),\cD_{(a_{\lambda,0})}(p^{\lambda}c_{\frakq})^1(0,\psi)),\] acting on $W^{\lambda,I}$. Recall that we require $\lambda=\lambda_i^{\bullet}\gg 0$, so this is the same as the spectral flow from $l$ to $l+p^\lambda A_{\lambda}p^\lambda$ considering as operators from $W^{I}_k$ to $W_{k-1}^I$. Now, we reduce the problem to showing that there is no spectral flow between $l+p^\lambda A_{\lambda}p^\lambda$ and $l+ A_{\infty}$ as operators from $W^{I}_k$ to $W_{k-1}^I$. The final step can be done exactly the same as in \cite[p.156]{Lidman2016TheEO}: First check that both operators are injective, then interpolate between them linearly and argue by contradiction that each operator in this family is injective, hence there is no spectral flow as we desired. 
\end{proof}
 With this identification in hand, Lemma \ref{lem:grading comparison in and outside N} can be packed into a more concise statement:\begin{prop}
    Any reducible stationary point of $\xagcsql$ in $(B(2R)\cap W^{\lambda,I})^\sigma/\Z_2$ that has grading in $[-N,N]$ is contained in $\cN/\Z_2$, provided that $\lambda=\lambda_{i}^\bullet\gg 0$.
 \end{prop}

\subsection{Morse quasi-gradient flow and Morse-Smale condition}\label{sub:Morse quasi-gradient flow and Morse Smale condition}
In \cite[Section 8]{Lidman2016TheEO}, they constructed an $S^1$-equivariant Morse function on $W^{\lambda}$ that has $\xgcql$ approximately as its gradient to show that $\xgcql$ is an $S^1$-equivariant Morse quasi-gradient. We now show that $\xgcql$ is a $\Z_2$-equivariant Morse quasi-gradient on $W^{\lambda,I}\cap B(2R)$ as defined in Definition \ref{defi:Z_2 equivariant Morse quasi-gradient}. As in \cite[Section 8]{Lidman2016TheEO}, we only consider the cut-off defined by $\lambda=\lambda_i^{\bullet}\gg0$ throughout this subsection.

Our aim is the following.
\begin{prop}\label{prop: xgcql is a equivariant Morse quasi gradient}
    We can choose an admissible perturbation $\frakq$ such that for all $\lambda=\lambda_i^{\bullet}\gg0$, the vector field $\xgcql$ is a $\Z_2$-equivariant Morse quasi-gradient on $W^{\lambda,I}\cap B(2R)$.
\end{prop}

Conditions (1)-(3) in Definition \ref{defi:Z_2 equivariant Morse quasi-gradient} follow from Proposition \ref{prop:hyperbolicity in B(2R)} and the fact that $\tg$ and the $L^2$ metric coincide at reducible stationary points. Thus, it is enough for us to find a $\Z_2$-equivariant function that fits into (4) of Definition \ref{defi:Z_2 equivariant Morse quasi-gradient}. We will adapt their method to show the following.

\begin{prop}\label{prop:existence of Flambdar}
    For each $\lambda\gg 0$, there exists a function $F^R_{\lambda}:W^{\lambda,I}\cap B(2R)\to \R$ such that \[\frac{1}{4}\left\Vert \xgcql \right\Vert_{\tg}^2\le dF^R_{\lambda}(\xgcql) \le 4\left\Vert \xgcql \right\Vert_{\tg}^2.\] In particular, $ dF^R_{\lambda}(\xgcql)\ge 0$ with equality holding only at stationary points of $\xgcql$.
\end{prop}
\begin{remark}\label{rmk: warning on construction of FlambdaR}
    (A warning on the construction of $F_{\lambda}^R$.) We have to be careful about how to adapt their proof to our case. If we use the pairing with invariant forms and sections to define the perturbations, then the function $F_\lambda$ from their construction is $I$-invariant, so the estimation carries over to the $I$-invariant subspace, but it is unclear to the author whether it can achieve admissible property for all (not just $I$-invariant) stationary points and trajectories, which is needed for $F_\lambda$ to be well-defined. On the other hand, if we use a more general perturbation, we won't face the admissibility issue, but now the functional under consideration is no longer $I$-invariant, so taking gradient in $W$ is not the same as in $W^I$, making it impossible to make use of the existing inequality to conclude the proof. Since we want to use the fact that $I$-equivariant very compact maps on $W$ restrict to very compact maps on $W^I$ (so we have no need to repeat some argument from \cite[Chapter 6]{Lidman2016TheEO}), we will keep the invariant assumption on perturbations and reconstruct $F^R_\lambda$ on $W^I$. Fortunately, their construction can be repeated on $W^I$ with minor changes.
\end{remark}

\begin{lem} (\cite[Lemma 8.1.1]{Lidman2016TheEO})
    Fix $\epsilon>0$. Then, for all $\lambda\gg 0$, we have \[\frac{1}{4}\left\Vert \xgcql \right\Vert_{\tg}^2\le d\cL_\frakq(\xgcql) \le 4\left\Vert \xgcql \right\Vert_{\tg}^2, \] at any point in $W^{\lambda,I}\cap B(2R)$ that is at $L_k^2$-distance at least $\epsilon$ from all stationary points of $\xgcql$ in $B(2R)^I$.
\end{lem}
\begin{proof}
    We can show by contradiction in exactly the same way as for \cite[Lemma 8.1.1]{Lidman2016TheEO}.
\end{proof}

We shall construct $F_\lambda^R$ in the same way as they do, i.e., consider the composition of $\cL_\frakq$ with a cut-off translation on $B(2R)\cap W^{\lambda,I}$ supported in the neighborhood of stationary points. To be concise, we won't repeat details since their  definitions and estimations directly apply to our case. Instead, we briefly review the construction and point out things that should be modified.

By the finiteness of stationary points, we can choose $\epsilon>0$ satisfying \cite[Assumption 8.2.1]{Lidman2016TheEO}. Note that now our orbits of stationary points are pairs of points when it is irreducible and are single points when it is reducible. In particular, all orbits are discrete point sets. When $\lambda$ is large enough, we can assume that for each $i$, the approximation $x_\lambda^i$ is within distance $\epsilon$ to $x^i_\infty$. Thus, the approximation for $(a,\phi)$ and $(a,-\phi)$ can be distinguished when they are a pair of irreducible stationary points sharing the same orbit. 

Assume that some $0<\epsilon\ll 1$ satisfying \cite[Assumption 8.2.1]{Lidman2016TheEO}(i.e., a constant that is far smaller than distance between orbits and norm of irreducible points) is fixed and $\lambda$ is chosen large enough in the sense above. We can take $H^j_\lambda$(an appropriate cut-off function that depends only on the distance function) defined in \cite[p130]{Lidman2016TheEO} and consider \[ T_{\lambda}(x)=x+\sum_{i=1}^m H^j_\lambda(x)(x^j_\infty-x_\lambda^j).\] Since we have discrete orbits, the functions $\omega^j_\lambda$ are redundant now. Here, we let $x^1_\infty,\ldots,x^m_\infty$ be all stationary points (not orbits!) of $\xgcsq$ in $W^I$. Then we define $F_\lambda^R=\cL_\frakq \circ T_{\lambda}$. This is obviously $\Z_2$-invariant since the definition of $H^j_\lambda$ involves only a fixed function $h$ and the distance function, and it supports in $2\epsilon$-neighborhoods of stationary points. Since $F^R_\lambda$ is actually a simplified version of $F_\lambda$ constructed in \cite[Section 8.2]{Lidman2016TheEO}, the estimations in \cite[Section 8.3-8.5]{Lidman2016TheEO} carry over to prove Proposition \ref{prop:existence of Flambdar} and thus to establish Proposition \ref{prop: xgcql is a equivariant Morse quasi gradient}. Some parts of proof can even be simplified since we do not need to take care of derivative of $\omega^j_\lambda$ terms. 

As a corollary, we have the following.
\begin{cor}
There is some $C_0>0$ such that for $\lambda\gg 0$, \[\frac{1}{4C_0}\left\Vert \xgcql \right\Vert_{L^2}^2\le dF^R_{\lambda}(\xgcql) \le 4C_0\left\Vert \xgcql \right\Vert_{L^2}^2.\]
\end{cor}
\begin{proof}
    This follows from the equivalence between $L^2$ and $\tg$ metric from \cite[Proposition 8.4.2]{Lidman2016TheEO}.
\end{proof}
\begin{cor}
    Let $\gamma:\R\to W^{\lambda,I}$ be a flow of $\xgcql$, for $\lambda\gg 0$. Then the limits $\lim_{t\to \pm \infty}[\gamma(t)]$ exist in $(W^{\lambda,I}\cap B(2R))/\Z_2$ and they are both projections of stationary points.
\end{cor}
\begin{cor}
    Let $[\gamma]:\R\to (W^{\lambda,I}\cap B(2R))^\sigma/\Z_2$ be a flow of $\xagcsql$, for $\lambda\gg 0$. Then the limits $\lim_{t\to \pm \infty}[\gamma(t)]$ exist in $(W^{\lambda,I}\cap B(2R))^\sigma/\Z_2$ and they are both stationary points of $\xagcsql$.
\end{cor}
For $I\subset \R$, an interval and $\gamma:\R\to W^{\lambda,I}\cap B(2R)$, a trajectory of $\xgcql$, we can define its energy \[\cE(\gamma)=\int_I \left\Vert \frac{d\gamma}{dt} \right\Vert^2_{L^2(Y)} dt=\int_I \left\Vert \xgcql(\gamma(t))\right\Vert^2_{L^2(Y)} dt. \]
\begin{cor}
There exists some $C_0>0$ such that for any $\lambda\gg 0$ and any closed interval $[t_1,t_2]\subset \R$, if $\gamma:[t_1,t_2]\to W^{\lambda,I}\cap B(2R)$ is a trajectory of $\xgcql$, then \[\frac{1}{4C_0} \cE(\gamma)\le F^R_\lambda(\gamma(t_2))- F^R_\lambda(\gamma(t_1)) \le  4C_0\cE(\gamma).\]
    
\end{cor}

Next, we turn to the Morse-Smale condition on $\xgcql$, which is also crucial since we want to define a Morse complex from it. In the case without a real structure, Lidman and Manolescu proved this in \cite[Chapter 10]{Lidman2016TheEO}. We will fit their argument into our setup.

Following their discussion, the Morse-Smale condition can be rephrased in terms of the surjectivity of the following operator: \[\Piagcs\circ(\frac{D^{\sigma}}{dt}+\cD^\sigma \xgcsql):T_{j,\gamma}\cP(x_{\lambda},y_{\lambda})\to T_{j-1,\gamma}\cP(x_{\lambda},y_{\lambda}),\] for $\gamma:\R\to (W^{\lambda,I}\cap B(2R))^\sigma\subset (W^{\lambda,I})^\sigma$ a trajectory between stationary points $x$, $y$ of $\xgcsql$. This has already appeared in the proof of Proposition \ref{prop:indentification between f.d. and inf.d gr}. Of course, $\gamma$ can also be regarded as a path in $W^{\sigma,I}$, along which we have the linearized operator \[\cD^{\tau}_{\gamma}\cF_{\frakq^\lambda}^{gC,\tau}:\cT^{gC,\tau,I}_{k,\gamma}(x,y)\to \cV^{gC,\tau,I}_{k-1,\gamma}(Z).\]

\begin{lem}\label{lem:surjectivity equivalence for Morse-Smale condition}
    The surjectivities of two operators associated to $\gamma$ are equivalent.
\end{lem}
\begin{proof}
    First, we note that Lemma \ref{lem:equivalence of fredholm and surjectivity} holds with $\frakq^{\lambda}$ in place of $\frakq$, so the surjectivity of $\cD^{\tau}_{\gamma}\cF_{\frakq^\lambda}^{gC,\tau}$ is equivalent to that of $Q_{\gamma}^{gC}$. We have remarked in the proof of Lemma \ref{lem:index of Q gamma gc} that $Q_{\gamma}^{gC,\lambda}$ appears as a block of $Q_{\gamma}^{gC}$ and their surjectivities are equivalent since the other diagonal block is invertible. Finally, we can relate $Q_{\gamma}^{gC,\lambda}$ to $\Piagcs\circ(\frac{D^{\sigma}}{dt}+\cD^\sigma \xgcsql)$ via an interpolation as in Proposition \ref{prop: Fredholm property of Q gc gamma} (see \cite[Proposition 9.1.4]{Lidman2016TheEO} for details) and show that their surjectivities are equivalent.
\end{proof}

Thus, it is enough for us to work with $\cD^{\tau}_{\gamma}\cF_{\frakq^\lambda}^{gC,\tau}$. We should also take extra care when $\gamma$ is boundary-obstructed: in that case, $\cD^{\tau}_{\gamma}\cF_{\frakq^\lambda}^{gC,\tau}$ can never be surjective, so instead, we ask for surjectivity of $(\cD^{\tau}_{\gamma}\cF_{\frakq^\lambda}^{gC,\tau})^\partial$, which acts along paths in $\partial (W^{\lambda,I})^\sigma$.

\begin{prop}\label{prop:Morse smale condition on xgcsql}
We can choose the admissible perturbation $\frakq$ such that for any $\lambda\in \{\lambda_{1}^\bullet<\lambda_{2}^\bullet<\ldots\}$ sufficiently large, we have the following. Given any trajectory $\gamma$ of $\xgcsql|_{(W^{\lambda,I}\cap B(2R))^\sigma}$ we have:\begin{enumerate}
    \item If $\gamma$ is boundary-unobstructed, then $\cD^{\tau}_{\gamma}\cF_{\frakq^\lambda}^{gC,\tau}$ is surjective;
    \item If $\gamma$ is boundary-obstructed, then $(\cD^{\tau}_{\gamma}\cF_{\frakq^\lambda}^{gC,\tau})^\partial$ is surjective;
\end{enumerate}
    
\end{prop}
\begin{proof}
    This is a real analogue of \cite[Proposition 10.0.2]{Lidman2016TheEO}. Their argument works well for us. Readers can also refer to \cite[Theorem 8.24]{li2022monopolefloerhomologyreal} and \cite[Proposition 15.1.3]{Kronheimer_Mrowka_2007} for detailed arguments.
\end{proof}

Now we put the results from Subsection \ref{sub:Convergence of stationary points}, \ref{sub:Identification of grading} and this one together. Our ultimate goal is to identify $\widecheck{\mathit{HMR}}_*(Y,\iota,\s)$ with $\widetilde{H}^{\Z_2}_*(\mathit{SWF}_{\Z_2}(Y,\iota,\s))$. We have outlined our strategy in Subsection \ref{sub:Outline of the proof}. More precisely, we will make use of an intermediate group-the Morse homology of $\xagcsql$ on $(B(2R)\cap W^{\lambda,I})^\sigma/\Z_2$ to show that, for each $N\gg0$, we have $\widecheck{\mathit{HMR}}_{\le N}(Y,\iota,\s) \cong \widetilde{H}^{\Z_2}_{\le N}(\mathit{SWF}_{\Z_2}(Y,\iota,\s))$. This intermediate group is well-defined by Proposition \ref{prop: xgcql is a equivariant Morse quasi gradient} and \ref{prop:Morse smale condition on xgcsql}. And we have seen in Subsection \ref{sub:Morse homology and Morse quasi gradient flow} that this is isomorphic to $\widetilde{H}^{\Z_2}_{\le N}(\mathit{SWF}_{\Z_2}(Y,\iota,\s))$ in an appropriate grading range. The remaining step is to identify this with $\widecheck{\mathit{HMR}}$ in the same grading range. 

We have a one-to-one correspondence between generating set of their respective chain complexes, provided $\lambda\in \{\lambda_{1}^\bullet<\lambda_{2}^\bullet<\ldots\}$ is sufficiently large. We will show in the next subsection that we have similar results for the moduli space of trajectories. This will identify the differentials, give rise to a chain map, and finish the proof. 

\subsection{Convergence of trajectories}\label{sub:Convergence of trajectories}
\subsubsection{Self-diffeomorphism of configuration space}

In \cite[Chapter 11]{Lidman2016TheEO}, they extended the one-to-one correspondence $\Xi_{\lambda}:\frC^{\lambda}_{\cN} \to \frC_{\cN}$ of stationary points from \cite[Corollary 7.2.3]{Lidman2016TheEO} to a self-diffeomorphism $\Xi_{\lambda}:W^{\sigma}_0\to W^{\sigma}_0$. The diffeomorphism is crucial in the identification of the moduli spaces of trajectories.

We want to use their proving strategy, so we need a real analogue of their result. As for the function $F_{\lambda}$, we cannot simply restrict their construction due to the reasons listed in Remark \ref{rmk: warning on construction of FlambdaR}. Nevertheless, we can closely follow their construction to get some $\Xi
_\lambda$ on the $I$-invariant configuration space. 

To be concise, we won't repeat their construction, but we will list the properties that shall be useful and point out the modifications that are needed.
\begin{lem}(\cite[Lemma 11.0.1]{Lidman2016TheEO})\label{lem: properties of Xi in dimension 3}
    For $\lambda\gg 0$, we have a $\Z_2$-equivariant diffeomorphism  $\Xi_{\lambda}:W^{\sigma,I}_0\to W^{\sigma,I}_0$ satisfying:\begin{enumerate}
        \item $\Xi_{\lambda}$ sends $x_{\lambda}$ to $x_{\infty}$ for each stationary point $x_{\infty}\in\cN$.
        \item $\Xi_{\lambda}$ restricts to a self-diffeomorphism of $W_{j}^{\sigma,I}$, for $0\le j\le k$.
        \item Let $\Xi_\infty$ be the identity map. Then, for $0\le j\le k$, $\Xi_{\lambda}:W_{j}^{\sigma,I}\to W_{j}^{\sigma,I}$ and all its derivatives are smooth in $\lambda$ at and near $\infty$.
        \item $\Xi_{\lambda}$ extends to the double $\widetilde{W}^{\sigma,I}$ with all the properties above preserved.
    \end{enumerate}
\end{lem}
\begin{prop}(\cite[Proposition 11.0.2]{Lidman2016TheEO})\label{prop: properties of Xi in dimension 4}
Let $\Xi_{\lambda}$ be as in Lemma \ref{lem: properties of Xi in dimension 3}. Fix $x_\infty$ and $y_\infty$, stationary points of $\xgcsq$ in $\cN$ and let $x_{\lambda}$, $y_\lambda$ be their approximations (stationary points of $\xgcsql$ with lowest distance to them, that exist and are unique when $\lambda$ is large enough). Then for $\lambda\gg 0$ and $1\le j\le k$, we have the following.
\begin{enumerate}
    \item For a compact interval $I\subset\R$, the map $\Xi_{\lambda}$ induces a $\Z_2$-equivariant diffeomorphism of $\widetilde{W}^{\tau,I}(I\times Y)$, which is smooth in $\lambda$ at and near $\infty$.
    \item $\Xi_{\lambda}$ induces diffeomorphisms from $\widetilde{W}^{\tau,I}_j(x_\lambda,y_\lambda)$ to $\widetilde{W}^{\tau,I}_j(x_\infty,y_\infty)$, which are smooth in $\lambda$ at and near $\infty$.
    \item $\Xi_{\lambda}$ induces diffeomorphisms from $\widetilde{\cB}^{gC,\tau,I}_j([x_\lambda],[y_\lambda])$ to $\widetilde{\cB}^{gC,\tau,I}_j([x_\infty],[y_\infty])$ that vary smoothly in $\lambda$ at and near $\infty$.
    \item The diffeomorphisms from $\widetilde{\cB}^{gC,\tau,I}_j([x_\lambda],[y_\lambda])$ to $\widetilde{\cB}^{gC,\tau,I}_j([x_\infty],[y_\infty])$ lift to bundle maps \[\xymatrix{\cV^{gC,\tau}_j \ar[d] \ar[r]^{(\Xi_{\lambda})_*} & \cV^{gC,\tau}_j \ar[d]\\
		\widetilde{\cB}^{gC,\tau,I}_j([x_\lambda],[y_\lambda]) \ar[r]^{\Xi_{\lambda}} & \widetilde{\cB}^{gC,\tau,I}_j([x_\infty],[y_\infty])\\}. \]
        When $x_\infty\ne y_\infty$, similar result holds for $\widetilde{\cB}^{gC,\tau,I}_j([x_\infty],[y_\infty])/\R$.
\end{enumerate}
    
\end{prop}
\begin{cor} (\cite[Corollary 11.0.3-11.0.4]{Lidman2016TheEO})
    Fix some $j$ with $1\le j\le k$. \begin{itemize}
        \item If a sequence $\gamma_n\in W^{\tau,I}_{j,loc}(I\times Y)$ converges to some $\gamma_\infty$, then $\Xi_{\lambda_n}(\gamma_n)\to \gamma_\infty$ for any sequence $\lambda_n\to \infty$. 
        \item Let $\gamma_0\in W^{\tau,I}_k(x_\infty, y_\infty)$. If a sequence $\gamma_n\in W^{\tau,I}_{j,loc}(x_{\lambda_n},y_{\lambda_n})$ with $\lambda_n\to \infty$ satisfies \[\left\Vert \Xi_{\lambda_n}^{-1}(\gamma_0)-\gamma_n \right\Vert_{L^2_j(\R\times Y)}\to 0,\] then \[\left\Vert \Xi_{\lambda_n}(\gamma_n)-\gamma_0 \right\Vert_{L^2_j(\R\times Y)}\to 0.\]
    \end{itemize}
    
\end{cor}

Now we list some key points for the construction and the proof.

\begin{enumerate}
    \item $W^{\sigma,I}_j$ is not an affine space, but we can embed it into a larger affine space $\widehat{W}_j^{\sigma,I}$ which is defined by releasing the restriction $s\ge0$ and $\left\Vert \phi \right\Vert_{L^2}=1$. Using this embedding, we can take derivatives, consider the difference between elements, and take $L_j^2$ norms on $W^{\sigma,I}_j$ in a natural way. The same remarks work for $W^{\sigma,I}_j(I\times Y)$.
    \item As in the proof of \cite[Lemma 11.1.1]{Lidman2016TheEO}, we work in charts near each stationary point of the form $x_\infty$. As we have seen in Subsection \ref{sub:Morse quasi-gradient flow and Morse Smale condition}, our orbits are discrete, so there is no tangent to the orbit. This makes the construction easier since we have no need to take $U_{x_\infty}$ at first and then let $S^1$ act on it. We just take one $x_\infty$ from each orbit, define maps supported in a neighborhood, then use $\Z_2$-equivariance to extend it to the other point in its orbit, and finally extend to the whole space by identity map. More precisely, for each $x_{\infty}=(a_\infty,s_\infty,\phi_\infty)$, we define a chart ``centered'' at it. Note that \[U_{x_\infty}=\{(a,s,\phi)\in W^{\sigma,I}_0|\left \langle \phi,\phi_{\infty} \right \rangle_{L^2}>0\}\]  is indeed a neighborhood of $x_{\infty}$ in $W^{\sigma,I}_0$, since $2>\left\Vert \phi-\phi_\infty \right\Vert_{L^2} =2-2\left \langle \phi,\phi_{\infty} \right \rangle_{L^2}$ implies $\left \langle \phi,\phi_{\infty} \right \rangle_{L^2}>0$. Let \[V_{x_\infty}=\{(a,s,\phi)\in (\mathrm{ker}d^*)_0^{-\iota^*}\oplus \R\oplus L^2(Y;\bS)^\tau| s\ge 0, \left \langle \phi,\phi_{\infty} \right \rangle_{L^2}=0\},\] as defined in the proof of \cite[Lemma 11.1.1]{Lidman2016TheEO}. The formulas \[G_{x_\infty}:U_{x_\infty}\to V_{x_\infty}, \text{ } (a,s,\phi)\mapsto (a,s,\frac{\phi}{\left \langle \phi,\phi_{\infty} \right \rangle_{L^2}}-\phi_\infty)\]
    and \[G_{x_\infty}^{-1}:V_{x_\infty}\to U_{x_\infty}, \text{ } (a,s,\phi)\mapsto (a,s,\frac{\phi+\phi_\infty}{\left\Vert \phi+\phi_\infty \right\Vert_{L^2}})\] give rise to a pair of diffeomorphisms inverse to each other. It is easy to see that $G_{x_{\infty}}(x_\infty)=(a,s,0)$ and the approximate stationary point $x_{\lambda}$ lies in $U_{x_\infty}$ whenever $\lambda$ is large enough. 

    We can choose $0<\delta\ll 1/2$ as in the last paragraph of \cite[p.166]{Lidman2016TheEO}, which satisfies the disjoint assumption for all stationary points, not just for the chosen representative of each orbit. The formula for $\Upsilon_{\lambda}$ (a cut-off translation $V_{x_{\infty}}\to V_{x_{\infty}}$)works for us without change. Then, $\Xi_{\lambda}=G_{x_\infty}^{-1}\circ \Upsilon_{\lambda}\circ G_{x_\infty}$ gives us the desired diffeomorphism in a neighborhood of $x_{\infty}$. We ``copy'' it to the corresponding neighborhood for the other point in its orbit and repeat this construction for each orbit of stationary points. This finishes the construction and it is obviously a $\Z_2$-equivariant diffeomorphism of $W^{\sigma,I}$.

    \item With the help of the explicit formula listed at the end of \cite[Section 11.1]{Lidman2016TheEO}, we can see that it maps $x_\lambda$ to $x_\infty$ and that it extends to a self-diffeomorphism of $\widetilde{W}^{\sigma,I}$. 

    \item The three-dimensional properties in Lemma \ref{lem: properties of Xi in dimension 3} mainly concern the smoothness of $\Xi_\lambda$ and its derivatives. It is obvious that $\Xi_\lambda$ extends to a neighborhood of $W^{\sigma,I}_0$ in $\widehat{W}^{\sigma,I}_0$, so we can take derivatives. The explicit formulas, together with the inverse function theorem used in the proof of convergence of stationary points tell us that $\Xi_\lambda$ has all smoothness we asked for.

    \item For the four-dimensional properties, the preliminary estimates in \cite[Section 11.2.2 \& 11.3.1]{Lidman2016TheEO} are quite formal, so they work in $W^I$ as well. We can introduce the four-dimensional $\widehat{W}^I$ using the same formula, so the smoothness in $\lambda$ can be defined in the same way as they described after \cite[Proposition 11.3.7]{Lidman2016TheEO}. Then, their proof works in the real case with only notation changes.

    \item For the extension to other path spaces, note that there is no difference between temporal gauge and pseudo-temporal gauge, and there is no non-trivial gauge group on infinite cylinder. So life is actually easier for us, and a simplified version of their argument works for us. The same remark also works for the lifting of bundle maps. 
 
 \end{enumerate}

Now, we can conclude the proof of Lemma \ref{lem: properties of Xi in dimension 3} and Proposition \ref{prop: properties of Xi in dimension 4}. This family of diffeomorphisms will be useful when we consider convergence of trajectories in the $L_k^2$ norm later.

\subsubsection{Convergence of trajectories downstairs} \label{subsub:Convergence of trajectories downstair}
\begin{prop}\label{prop:convergence of trajectories downstair}
    Let $I\subset \R$ be an interval and let $\gamma_n$ be a sequence of trajectories of $\xgcqln$ contained in $B(2R)$, with $\lambda_n\to \infty$. Then there is a subsequence of $\gamma_n$ for which the restriction to any compact subinterval $I'\subset I^\circ$ converges in $C^\infty$ topology of $W^I(I\times Y)$ to $\gamma$, a trajectory of $\xgcq$.
\end{prop}
\begin{proof}
    This follows from Proposition \ref{prop:properties of very compact maps} and Lemma \ref{lem:Compactness of trajectories with varying lambda}. 
\end{proof}

The definitions in \cite[Section 12.1]{Lidman2016TheEO} for \emph{stationary point class}, \emph{parametrized trajectory class}, \emph{unparametrized trajectory class}, etc. work for us without change, so we won't repeat them for simplicity. The compactness argument in \cite[Proposition 12.1.4]{Lidman2016TheEO} holds for us also. \begin{prop}\label{prop:convergence of trajectory class to broken ones downstair}
    Fix $[x]$, $[y]$ stationary point classes of $\xgcq$. Fix $\lambda_n\to \infty$ and a sequence of unparametrized trajectories $[\Breve{\gamma}_n]$ of $\xgcqln$ from $[x_{\lambda_n}]$ to $[y_{\lambda_n}]$ and such that the representatives $\gamma_n$ of $[\Breve{\gamma}_n]$ are contained in $W^{\lambda_n,I}\cap B(2R)$. Then there exists a subsequence of $[\Breve{\gamma}_n]$ that converges to a broken trajectory class $[\Breve{\gamma}_\infty]$ of $\xgcq$.
\end{prop}

\subsubsection{Convergence of parametrized trajectories in the blow-up}

Now we move to the blow-up configuration space. We first consider convergence of parametrized trajectories under some control of spinorial energy. One should note that any real trajectory in $W^I$ and its cut-off or blow-up is a trajectory that is considered in usual monopole Floer homology, so all those energy controls in \cite[Chapter 12]{Lidman2016TheEO} work for us. (And we may even get better bounds.) We will use $\gamma^\tau$ for a path upstairs and $\gamma$ for its blow-down image.

\begin{prop}(\cite[Proposition 12.2.1]{Lidman2016TheEO})\label{prop:upstair covergence with energy control}
    Fix $\omega>0$ and a compact interval $I=[t_1,t_2]\subset \R$. Consider a smaller interval $I_{\epsilon}=[t_1+\epsilon,t_2-\epsilon]\subset [t_1,t_2]$ for $\epsilon>0$. Suppose that $\gamma_n^\tau:I\to (W^{\lambda_n,I}\cap B(2R))^\sigma$ is a sequence of trajectories of $\xgcsqln$ with $\lambda_n\to \infty$. If there are bounds \[\Lambda_{\frakq^{\lambda_n}}(\gamma^{\tau}_n(t_1+\epsilon))\le \omega,\text{ } \Lambda_{\frakq^{\lambda_n}}(\gamma^{\tau}_n(t_2-\epsilon))\ge -\omega\]
    at ends of $I_{\epsilon}$, for all $n$. Then there exists a subsequence of $\gamma_n^\tau$ whose restriction to any compact subinterval $I'\subset I_{\epsilon}^\circ$ converges in $C^\infty$ topology of $W^{\tau,I}(I_{\epsilon}\times Y)$ to $\gamma^\tau$, a trajectory of $\xgcsq$. 
\end{prop}
\begin{proof}
    The key to proving \cite[Proposition 12.2.1]{Lidman2016TheEO} is the estimation in \cite[Lemma 12.2.4]{Lidman2016TheEO}. The argument is based on \cite[lemma 10.9.1 and Theorem 10.9.2]{Kronheimer_Mrowka_2007}. We have real analogues \cite[Lemma 6.9 and Theorem 6.10]{li2022monopolefloerhomologyreal} for it and Li observed that the original argument in \cite{Kronheimer_Mrowka_2007} works in the real case. Although on the real global Coulomb slice we can no longer multiply the configuration by $e^{i\theta}$ and take derivatives, we can first embed the sequence in $W^I$ into the larger space $W$ and use their result. This allows us to repeat their argument and conclude the proof. For later use, we restate \cite[Lemma 12.2.3 and 12.2.4]{Lidman2016TheEO} below.
\end{proof}
\begin{lem}( \cite[Lemma 12.2.3]{Lidman2016TheEO})\label{lem:unique continuation for traj in B(2R)}
Let $\gamma(t)=(a(t),\phi(t))$ be a trajectory of $\xgcql$ in $B(2R)$ for some $\lambda\in (0,\infty]$. If $\phi(t)=0$ for some $t$, then $\phi\equiv 0$.
\end{lem}
\begin{lem}(\cite[Lemma 12.2.4]{Lidman2016TheEO})\label{lem:energy control in B(2R)}
There is a constant $C>0$ such that for any $\lambda\gg 0$ and interval $[t_1,t_2]\subset\R$, trajectory $\gamma^\tau:[t_1,t_2]\to B(2R)^\sigma$ of $\xgcsql$ and $t\in [t_1,t_2]$, we have \[\frac{d}{dt}\Lambda_{\frakq^\lambda}(\gamma^\tau(t))\le C\cdot \left\Vert \xgcql(\gamma(t)) \right\Vert_{L^2_k(Y)}.\]
    
\end{lem}

Using Proposition \ref{prop:upstair covergence with energy control} and a diagonalization argument, we have the following.

\begin{cor}
   Fix $\omega>0$ and a closed interval $I\subset \R$. Suppose that $\gamma_n^\tau:I\to (W^{\lambda_n,I}\cap B(2R))^\sigma$ is a sequence of trajectories of $\xgcsqln$ with $\lambda_n\to \infty$. Further assume $\vert \Lambda_{\frakq^{\lambda_n}}(\gamma_n^\tau(t)) \vert\le \omega$. Then there exists a subsequence of $\gamma_n^\tau$ whose restriction to any compact subinterval $I'\subset I^\circ$ converges in $C^\infty$ topology of $W^{\tau,I}(I'\times Y)$ to $\gamma^\tau$, a trajectory of $\xgcsq$. 
\end{cor}
\subsubsection{Near constant approximations}\label{subsub:Near constant approximations}

\cite[Section 12.3]{Lidman2016TheEO} consists of technical results about moduli spaces of broken trajectories. Almost all of their arguments work for us, so we shall restate the results in our context and notation and make some remarks on the proof. 

\begin{lem}(\cite[Lemma 12.3.1]{Lidman2016TheEO})\label{lem:contol on L^2_1 norm by Flambda}
Fix any stationary point $x_{\infty}=(a_{\infty},\phi_\infty)$ of $\xgcq$ in $W^{I}$. It can be regarded as a constant trajectory of  $\xgcq$ on $W^{I}(I\times Y)$. Then there is a neighborhood $U$ of $x_{\infty}$ in $W^{I}(I\times Y)$ and a constant $C$ independent of $\lambda\gg 0$, such that if $\gamma\in U$ is a trajectory of $\xgcql$, we have \[\left\Vert \gamma-x_{\lambda} \right\Vert_{L^{2}_k(I\times Y)}\le C(F^R_{\lambda}(\gamma(t_1))-F^R_{\lambda}(\gamma(t_2))).\]
\end{lem}
\begin{proof}
    First, note that we don't need to apply a gauge transformation as in \cite[Lemma 12.3.1]{Lidman2016TheEO}. That is because if $(a,\phi)$ and $(b,\psi)$ are both $I$-invariant configurations, then it is automatically true that $(0,i\phi)$ and $(b,\psi)$ are orthogonal, since we are using an $I$-invariant metric and we can choose $U$ small enough so that the corresponding neighborhood of $(-1)\cdot x_{\infty}$ is disjoint from it. This observation also tells us that the proof of \cite[Lemma 12.3.1]{Lidman2016TheEO} works in an even simpler way for us.
    \end{proof}

By bootstrapping, we have 
\begin{lem}(\cite[Lemma 12.3.2]{Lidman2016TheEO})\label{lem:contol on L^2_k norm by Flambda}
Fix any stationary point $x_{\infty}=(a_{\infty},\phi_\infty)$ of $\xgcq$ in $W^{I}$. It can be regarded as a constant trajectory of  $\xgcq$ on $W^{I}(I\times Y)$. Then there is a neighborhood $U$ of $x_{\infty}$ in $W^{I}(I\times Y)$ and a constant $C$ independent of $\lambda\gg 0$, such that if $\gamma\in U$ is a trajectory of $\xgcql$, we have \[\left\Vert \gamma-x_{\lambda} \right\Vert_{L^{2}_{k+1}(Y\times I')}\le C(F^R_{\lambda}(\gamma(t_1))-F^R_{\lambda}(\gamma(t_2))),\] for any compact subinterval $I'\subset I^\circ$.
\end{lem}

Based on this, we have 
\begin{prop}(\cite[Proposition 12.3.3]{Lidman2016TheEO})\label{prop:contol on L2k norm by Flambda on blow up1}
Fix any stationary point of $\xgcsq$, $x_{\infty}\in B(2R)^{\sigma}\subset W^{\sigma,I}$  and let $x_{\lambda}$ be the nearby stationary point of $\xgcsql$. $I'$ is a compact subinterval of $I^\circ=(t_1,t_2)$. Then there is a neighborhood $U$ of $x_{\infty}$ in $W^{I}(I\times Y)$ and a constant $C$ independent of $\lambda\gg 0$, such that if $\gamma^\tau: I\to (B(2R)\cap W^{\lambda,I})^{\sigma}$ is a trajectory of $\xgcsql$ in $U$, then we have \begin{enumerate}
    \item if $x_{\infty}$ is irreducible, then \[\left\Vert \gamma^\tau-x_{\lambda} \right\Vert_{L^{2}_{k+1}(I'\times Y)}\le C(F^R_{\lambda}(\gamma(t_1))-F^R_{\lambda}(\gamma(t_2))).\]
    \item if $x_{\infty}$ is reducible, then \[\left\Vert \gamma^\tau-x_{\lambda} \right\Vert_{L^{2}_{k+1}(I'\times Y)}\le C((F^R_{\lambda}(\gamma(t_2))-F^R_{\lambda}(\gamma(t_1)))^{\frac{1}{2}}+\Lambda_{\frakq^\lambda}(\gamma(t_1))-\Lambda_{\frakq^\lambda}(\gamma(t_2))).\] 
\end{enumerate} 
    
\end{prop}

\begin{prop}(\cite[Proposition 12.3.4]{Lidman2016TheEO})\label{prop:contol on L2k norm by Flambda on blow up2}
Fix any stationary point of $\xgcsq$, $x_{\infty}\in B(2R)^{\sigma}\subset W^{\sigma,I}$  and let $x_{\lambda}$ be the corresponding stationary point of $\xgcsql$. Then there is a neighborhood $U$ in $W^{I}([-1,1]\times Y)$ of the blow-down image of $x_{\infty}$ and a constant $C$ independent of $\lambda\gg 0$ with the following properties.  If $\gamma^\tau: [-1,1] \to (B(2R)\cap W^{\lambda,I})^{\sigma}$ is a trajectory of $\xgcsql$ whose blow-down image lies in $U$ and $x_{\infty}$ is irreducible, then \[\frac{d}{dt}\Lambda_{\frakq^\lambda}(\gamma^\tau(t))|_{t=0}\le C(F^R_{\lambda}(\gamma(-1))-F^R_{\lambda}(\gamma(1)))^{\frac{1}{2}}.\]    
\end{prop}

Since $F^R_{\lambda}$ is $\Z_2$-equivariant, it induces a function $F^R_{\lambda}:(B(2R)\cap W^{\lambda,I})^\sigma/\Z_2\to \R$ (we abuse the same notation for it). Using the one-to-one correspondence between stationary points in Lemma \ref{lem:identification of stationary points on blowup}, we have the following.

\begin{prop}(\cite[Proposition 12.3.5]{Lidman2016TheEO})\label{prop:exponential decay of Flambda on blow up}
    Let $[x_{\infty}]$ be a stationary point of $\xagcsq$ in $W^{\sigma,I}/\Z_2$ with grading in $[-N,N]$. Then there is some $\delta>0$ such that for all $\lambda\gg 0$ and every trajectory $[\gamma]:[0,\infty)\to (B(2R)\cap W^{\lambda,I})^\sigma/\Z_2$ of $\xagcsql$ with $\lim_{t\to \infty}[\tau^{*}_t\gamma]=[x_{\lambda}]$ in $L_{k,loc}^2$, there is some $t_0$, so that for all $t\ge t_0$, \[F^R_{\lambda}([\gamma(t)])-F^R_{\lambda}([x_{\lambda}])\le C e^{-\delta t}\] where $C=F^R_{\lambda}([\gamma(t_0)])-F^R_{\lambda}([x_{\lambda}])$.
\end{prop}
\begin{prop}(\cite[Proposition 12.3.6]{Lidman2016TheEO})\label{prop:exponential decay of Flambda on blow down}
     Let $x_{\infty}$ be a stationary point of $\xgcq$ in $W^{I}$. Then there is a neighborhood $U$ of $[x_{\infty}]$ in $(B(2R)\cap W^{I})/\Z_2$ and a constant $\delta>0$ such that for all $\lambda\gg 0$ and every trajectory $[\gamma]:[t_1,t_2] \to W
     ^{\lambda,I}\cap U$ of $\xgcql$ in $L_{k,loc}^2$, we have inequalities \[-C_2 e^{\delta (t-t_2)}\le F^R_{\lambda}([\gamma(t)])-F^R_{\lambda}([x_{\lambda}])\le C_1 e^{-\delta (t-t_1)}\] where $C_1=\vert F^R_{\lambda}([\gamma(t_1)])-F^R_{\lambda}([x_{\lambda}])\vert$ and $C_2=\vert F^R_{\lambda}([\gamma(t_2)])-F^R_{\lambda}([x_{\lambda}])\vert$.
\end{prop}

Finally, for a trajectory $\gamma$ of $\xgcsql$, we introduce \[K_{\lambda}(\gamma)=\int_{\R}\vert \frac{d\Lambda_{\frakq^\lambda}(\gamma)}{dt}\vert dt, \text{ }K_{\lambda,+}(\gamma)=\int_{\R}( \frac{d\Lambda_{\frakq^\lambda}(\gamma)}{dt})^{+} dt.\] These two quantities may be infinite in prior, but when one is finite, so is the other. Furthermore, if $\gamma$ is a trajectory from $x_{\lambda}$ to $y_{\lambda}$, we have \[\Lambda_{\frakq^\lambda}(x_{\lambda})-\Lambda_{\frakq^\lambda}(y_{\lambda})=K_{\lambda}(\gamma)-2K_{\lambda,+}(\gamma).\] Using this new terminology, we have \begin{cor}
    Let $x_{\infty}$ be a stationary point of $\xgcq$ in $W^I$. Given any $\eta>0$, there is a neighborhood $U$ of $x_{\infty}$ as a constant trajectory in $W_{k}^I([-1,1]\times Y)$ with the following property for $\lambda\gg 0$. Let $J\subset \R$ be an interval and $J'=J+[-1,1]$. If we have a trajectory $\gamma^\tau:J'\to (W^{\lambda,I}\cap B(2R))^{\sigma}$ of $\xgcsql$ such that $\tau_t\gamma$ are contained in $U$ for all $t\in J$, then $K_{\lambda,+}^J(\gamma^\tau)\le \eta$.
\end{cor}

\subsubsection{Convergence of unparametrized trajectories on the blow-up}
Similar to those definitions from \cite[Section 12.1]{Lidman2016TheEO} that we have used in Subsection \ref{subsub:Convergence of trajectories downstair}, we can define stationary point class, (un)parametrized trajectory in $W^{\sigma,I}$ and its quotient. We want analogous convergence on the blow-up and its quotient. More precisely, we will show the following.
\begin{prop}(\cite[Proposition 12.4.1]{Lidman2016TheEO})\label{prop:unpara convergenc in local norm}
    Let $[x]$, $[y]$ be a pair of $I$-fixed stationary points of $\xagcsq$ in grading range $[-N,N]$. Fix $\lambda_n\to \infty$ and a sequence of unparametrized trajectories $[\Breve{\gamma}_n]$ of $\xagcsqln$ from $[x_{\lambda_n}]$ to $[y_{\lambda_n}]$. We require $[\Breve{\gamma}_n]$ to have representatives $\gamma_n$ contained in $(W^{\lambda,I}\cap B(2R))^\sigma$. Then there is a subsequence of $[\Breve{\gamma}_n]$ that converges to an unparametrized trajectory $[\Breve{\gamma}_\infty]$ of $\xagcsq$.
\end{prop}\begin{proof}
    The proof of \cite[Proposition 12.4.1]{Lidman2016TheEO} used various energy bounds from \cite[Lemma 12.4.2-12.4.4]{Lidman2016TheEO}. Although the definition of grading changes, so we cannot directly quoted their results, the argument in that section works for us.
\end{proof}
\begin{cor}(\cite[Corollary 12.4.5]{Lidman2016TheEO})
Fix $\epsilon>0$. For $\lambda\gg 0$, let $[\gamma_\lambda]$ be a $I$-invariant trajectory of $\xagcsql$ from $[x_{\lambda}]$ to $[y_{\lambda}]$ such that one of the following is true.\begin{itemize}
    \item $[\gamma_\lambda]$ is boundary-unobstructed and $\mathrm{gr}([x_{\lambda}],[y_{\lambda}])=1$; 
    \item $[\gamma_\lambda]$ is boundary-obstructed and $\mathrm{gr}([x_{\lambda}],[y_{\lambda}])=0$.
\end{itemize}
    Further, suppose that the gradings of $[x_{\lambda}]$, $[y_{\lambda}]$ lie in $[-N,N]$. Then $[\gamma_\lambda]$ is $\epsilon$-closed in $W_{k,loc}^{I}(\R\times Y)/\Z_2$ to $[\gamma]$, a trajectory of $\xagcsq$ with endpoints in the grading range $[-N,N]$.
\end{cor}
\begin{proof}
    We argue by contradiction, as they did for \cite[Corollary 12.4.5]{Lidman2016TheEO}. We just need to note that the properties of moduli spaces they quoted from \cite[Section 14 \& 16] {Kronheimer_Mrowka_2007} are still true for real moduli spaces.
\end{proof}
\subsubsection{Convergence in $L_k^2$}

With all these preparations in hand, we now improve the local norm convergence in Proposition \ref{prop:unpara convergenc in local norm} to convergence in $L^{2}_k$ norm.
\begin{prop}(\cite[Proposition 12.5.1]{Lidman2016TheEO})
    Let $[\gamma_n]:\R\to (W^{\lambda_n,I}\cap B(2R))^\sigma/\Z_2$ be a sequence of trajectories of $\xagcsql$ between stationary points $[x_{\lambda_n}]$, $[y_{\lambda_n}]$. We also suppose
    that unparametrized trajectories $[\Breve{\gamma}_n]$ converge to $[\Breve{\gamma}_\infty]$, an unbroken trajectory from $[x_{\infty}]$ to $[y_{\infty}]$. Then possibly after some reparametrization, $\Xi_{\lambda_n}(\gamma_n)$ converges to $[\gamma_\infty]$ in $\cB^{gC,\tau}_k([x_{\infty}],[y_{\infty}])$, where $[\gamma_{\infty}]$ is a representative of $[\Breve{\gamma}_\infty]$.
\end{prop}
\begin{proof}
 They proved \cite[Proposition 12.5.1]{Lidman2016TheEO} based on \cite[Theorem 13.3.5]{Kronheimer_Mrowka_2007}, and there is a real analogue \cite[Theorem 8.6]{li2022monopolefloerhomologyreal}. In the real case, the only gauge transformations available are constant $\pm 1$. Nevertheless, we note that in the proof, they showed a stronger result that $\left\Vert \Xi_{\lambda_n}(\gamma_n)-\gamma_\infty\right\Vert_{L^2_k(\R\times Y)}\to 0$, based on their estimation in \cite[Section 12.3]{Lidman2016TheEO}. We can do the same by using corresponding results from Subsection \ref{subsub:Near constant approximations} to conclude the proof.
    
\end{proof}
We can specialize to trajectories in small index moduli spaces and get the following refined result.
\begin{lem}(\cite[Lemma 12.5.2]{Lidman2016TheEO})
Fix some $\epsilon,N>0$. For $\lambda$ sufficiently large, we have the following. Let $[\gamma_{\lambda}]$ be a trajectory from $[x_\lambda]$ to $[y_{\lambda}]$ such that either \begin{itemize}
    \item $[\gamma_\lambda]$ is boundary-unobstructed and $\mathrm{gr}([x_{\lambda}],[y_{\lambda}])=1$; 
    \item $[\gamma_\lambda]$ is boundary-obstructed and $\mathrm{gr}([x_{\lambda}],[y_{\lambda}])=0$.
\end{itemize} 
Further suppose that the gradings of $[x_{\lambda}]$, $[y_{\lambda}]$ lie in $[-N,N]$. Then $[\gamma_\lambda]$ is $\epsilon$-closed in $L_{k}^2(\R\times Y)$ to $[\gamma]$, a trajectory of $\xagcsq$ with endpoints in the grading range $[-N,N]$.
\end{lem}

\section{The equivalence of homology theories}\label{sec:The equivalence of homology theories}
\subsection{Stability}\label{sub:Stabilit}
Consider a pair of $[x_\infty]$ and $[y_\infty]$ in $\frC_{\cN}$. Assume the relative grading between them is one and they are not boundary-obstructed. Fix any admissible perturbation, we have moduli space \[\Breve{M}^{agC}([x_\infty],[y_\infty])=M^{agC}([x_\infty],[y_\infty])/\R\] consisting of finitely many points. This can be alternatively viewed as the zero set of $\cF^{gC,\tau}_\frakq$ as a section from $\widetilde{\cB}^{gC,\tau,I}([x_\infty],[y_\infty])/\R$ to $\cV^{gC,\tau,I}(\R\times Y)$, restricted to the part with $s(t)\ge 0$.

When $\lambda$ is large enough, we have well-defined nearby stationary points \[[x_\lambda]=\Xi_{\lambda}^{-1}([x_\infty]), \text{ } [y_\lambda]=\Xi_{\lambda}^{-1}([y_\infty])\] and a similar moduli space between them in $(B(2R)\cap W^{\lambda,I})^\sigma$ \[\Breve{M}^{agC}([x_\lambda],[y_\lambda])=M^{agC}([x_\lambda],[y_\lambda])/\R.\] We know that this space is also zero-dimensional, since we have identified relative gradings in Subsection \ref{subsub:Relative grading}. We now want to show the following:\begin{prop}(\cite[Proposition 13.1.1]{Lidman2016TheEO})
    For $\lambda=\lambda_i^\bullet\gg 0$, we have a one-to-one correspondence between $\Breve{M}^{agC}([x_\infty],[y_\infty])$ and $\Breve{M}^{agC}([x_\lambda],[y_\lambda])$.
\end{prop}
\begin{proof}
    The proof of this follows from Proposition \ref{prop:unique approximation for traj} as \cite[Proposition 13.1.1]{Lidman2016TheEO} follows from \cite[Proposition 13.1.2]{Lidman2016TheEO} using an argument by contradiction.
\end{proof}
\begin{prop}(\cite[Proposition 13.1.2]{Lidman2016TheEO})\label{prop:unique approximation for traj}
    Consider $[x_\infty]$ and $[y_\infty]$ in $\frC_{\cN}$ with $\mathrm{gr}([x_\infty],[y_\infty])=1$ and not boundary-obstructed. Fix $[\Breve{\gamma}_\infty]\in \Breve{M}^{agC}([x_\lambda],[y_\lambda]) $ and a small neighborhood $U$ of it in $\cB^{gC,\tau,I}([x_\infty],[y_\infty])/\R$. Then, for $\lambda\gg 0$, there is a unique $[\gamma_\lambda]\in \Breve{M}^{agC}([x_\lambda],[y_\lambda])$ with $\Xi_{\lambda}(\gamma_\lambda)$ in $U$. 
\end{prop}
\begin{proof}
    See the proof of \cite[Proposition 13.1.2]{Lidman2016TheEO}; replacing objects there by their real analogue is enough.
\end{proof}

We have analogous results in the boundary-obstructed case when the relative grading between $x$ and $y$ is 0 instead of 1.

\subsection{$v$-action}\label{sub:U-action}
We have seen the identification \[M([x],[y])\cap \scrZ^{\tau} \cong M^{agC}([x],[y])\cap \scrZ^{agC},\] in Subsection \ref{sub: Gradings and U-action} and a similar result on reducible moduli spaces. We now wish to further identify this with cut-down moduli spaces consisting of approximate trajectories. The approximate trajectories live in $(W^{\lambda,I})^\sigma/\Z_2$, so we restrict the zero set of $\zeta^{agC}$ to such spaces.

We will use an admissible perturbation guaranteed in Subsection \ref{sub:Morse quasi-gradient flow and Morse Smale condition}, making the moduli spaces of flow lines of $\xgcsql$ in $(W^{\lambda,I})^\sigma/\Z_2$ regular for all $\lambda=\lambda_i^{\bullet}$ sufficiently large. Since we only need to take care of countably many moduli spaces, we can choose the section $\zeta^{agC}$ to be transverse to all of them. Now, we can consider cut-down moduli spaces \[M^{agC}([x_{\lambda}],[y_{\lambda}])\cap \scrZ^{agC} \text{ and } M^{agC,red}([x_{\lambda}],[y_{\lambda}])\cap \scrZ^{agC},\] for $[x_\infty]$, $[y_\infty]$ in $\frC_{\cN}$. These moduli spaces are sufficient for determining the $v$-action in the grading range $[-N,N]$.

\begin{prop}
    \begin{enumerate}
        \item Suppose that $[x_\infty],[y_\infty]\in \frC_{\cN}$ have $\mathrm{gr}([x_\infty],[y_\infty])=2$ and are not boundary-obstructed. For $\lambda=\lambda_i^{\bullet}\gg 0$, there is a one-to-one correspondence between $M^{agC}([x_{\lambda}],[y_{\lambda}])\cap \scrZ^{agC}$ and $M^{agC}([x_{\infty}],[y_{\infty}])\cap \scrZ^{agC}$.
        \item Suppose that$[x_\infty],[y_\infty]\in \frC_{\cN}$ are reducibles having $\mathrm{gr}([x_\infty],[y_\infty])=1$ and are boundary-obstructed. For $\lambda=\lambda_i^{\bullet}\gg 0$, there is a one-to-one correspondence between $M^{agC,red}([x_{\lambda}],[y_{\lambda}])\cap \scrZ^{agC}$ and $M^{agC,red}([x_{\infty}],[y_{\infty}])\cap \scrZ^{agC}$.
    \end{enumerate}
\end{prop}
\begin{proof}
    See the proof of \cite[Proposition 13.2.1]{Lidman2016TheEO}.
\end{proof}

\subsection{Main theorem}\label{sub:Main thm}
Now we are ready to prove our main theorem. We follow the discussion in \cite[Chapter 14]{Lidman2016TheEO}, but various changes are needed to fit it into the real definitions, so we provide details as complete as possible.
\begin{proof}(proof of Theorem \ref{thm:main theorem})
Recall that in \ref{prop:identification of perturbed and unperturned homotopy type}, we have seen that \[ \widetilde{H}^{\Z_2}_{*}(\mathit{SWF}_{\Z_2}(Y,\iota,\s))\cong \widetilde{H}^{\Z_2}_{*}(\mathit{SWF}_{\Z_2,\frakq}(Y,\iota,\s)),\] in which $\mathit{SWF}_{\Z_2,\frakq}(Y,\iota,\s)$ is the spectrum defined using $l+p^\lambda c_{\frakq}$ in place of $l+p^\lambda c$. Here, we need $\frakq$ to be a very tame, admissible perturbation satisfying the properties in Proposition \ref{prop:hyperbolicity in B(2R)}, \ref{prop: xgcql is a equivariant Morse quasi gradient} and \ref{prop:Morse smale condition on xgcsql}. Recall that \[\mathit{SWF}_{\Z_2,\frakq}(Y,\iota,\s)=\Sigma^{-\mathrm{dim}(V_{-\lambda}^0)\R} \Sigma^{-(\mathrm{dim}(U_{-\lambda}^0)+n^{R}(Y,\iota,\s,g))\widetilde{\R}} I^{\lambda}_\frakq,\] in which $I^{\lambda}_\frakq$ is the $\Z_2$-equivariant Conley index for the flow $l+p^\lambda c_{\frakq}$.

Proposition \ref{prop: xgcql is a equivariant Morse quasi gradient} and \ref{prop:Morse smale condition on xgcsql} tell us that for $\lambda=\lambda_i^\bullet\gg 0$, $\xgcsql$ is a Morse-Smale equivariant quasi gradient on $W^{\lambda,I}\cap B(2R)$. Thus, we can construct a Morse complex $(\widecheck{C}_{\lambda},\widecheck{\partial}_\lambda)$ for $\xagcsql$ on $(W^{\lambda,I}\cap B(2R))^{\sigma}$. Then, as discussed in Subsection \ref{sub:Morse homology and Morse quasi gradient flow}, we have an isomorphism of $\F[v]$-modules  \[\widetilde{H}^{\Z_2}_{i}(I^{\lambda}_\frakq)\cong H_i(\widecheck{C}_{\lambda},\widecheck{\partial}_\lambda),\]  for $0\le i\le n_{\lambda}-1$, where $n_\lambda$ is the connectivity of $(I^\lambda_\frakq,(I^\lambda_\frakq-(I^\lambda_\frakq)^{\Z_2})\cup *)$.

Taking the grading shift (suspension) into account, this becomes \[\widetilde{H}^{\Z_2}_{i}(\mathit{SWF}_{\Z_2,\frakq}(Y,\iota,\s))\cong H_{i+\mathrm{dim}V_{-\lambda}^0+\mathrm{dim}U_{-\lambda}^0+n^R(Y,\iota,\s,g)}(\widecheck{C}_{\lambda},\widecheck{\partial}_\lambda),\] 
 for $-\mathrm{dim}V_{-\lambda}^0-\mathrm{dim}U_{-\lambda}^0-n^R(Y,\iota,\s,g)\le i\le n_{\lambda}-1-\mathrm{dim}V_{-\lambda}^0-\mathrm{dim}U_{-\lambda}^0-n^R(Y,\iota,\s,g)$.

Let \[M_{\lambda}=\min\{\mathrm{dim}V_{-\lambda}^0+\mathrm{dim}U_{-\lambda}^0+n^R(Y,\iota,\s,g), n_{\lambda}-1-\mathrm{dim}V_{-\lambda}^0-\mathrm{dim}U_{-\lambda}^0-n^R(Y,\iota,\s,g)\}.\] The isomorphism above holds in grading range $[-M_\lambda,M_\lambda]$. To finish the proof as outlined in Subsection \ref{sub:Outline of the proof}, we need to show $M_\lambda\to \infty$ as $\lambda=\lambda_i^\bullet$ goes to $\infty$. 

It is clear that $\mathrm{dim}V_{-\lambda}^0+\mathrm{dim}U_{-\lambda}^0$ goes to $\infty$ as $\lambda$ does. We need to show that $n_{\lambda}-1-\mathrm{dim}V_{-\lambda}^0-\mathrm{dim}U_{-\lambda}^0$ also grows without bound, since $n^R$ does not change with $\lambda$. This is done by analyzing how $n_\lambda$ changes with $\lambda$, i.e. the growth of connectivity of the pair $(I^\lambda_\frakq,(I^\lambda_\frakq-(I^\lambda_\frakq)^{\Z_2})\cup *)$.

Fix a $\mu=\lambda_i^\bullet$ that is large enough for all nice properties in previous sections to hold. Then, for $\lambda=\lambda_j^\bullet>\mu$, we have \[I_\frakq^\lambda=I_\frakq^\mu\wedge I(l)^\lambda_\mu,\] where $I(l)^\lambda_\mu$ is the Conley index associated to the isolated invariant set $\{0\}$ in the linear flow induced by $l$ on the complement of $W^{\mu,I}$ in $W^{\lambda,I}$ (See \cite[Section 7]{Manolescu_2003} or \cite[Section 3]{Konno_Miyazawa_Taniguchi_2025}). We decompose this with respect to the sign of eigenvalues and type of eigenvectors (or say type of representations.) More precisely, define 
\[a^{\mu,\lambda}_+=\mathrm{dim}V_{\mu}^{\lambda},\text{ } b^{\mu,\lambda}_+=\mathrm{dim}U_{\mu}^{\lambda},\]
\[a^{\mu,\lambda}_-=\mathrm{dim}V_{-\lambda}^{-\mu},\text{ } b^{\mu,\lambda}_-=\mathrm{dim}U_{-\lambda}^{-\mu}.\]

Recall that we combine notations from \cite{Lidman2016TheEO} and \cite{Konno2024}, but replace $W$ in \cite{Konno2024} by $U$ for the space of representations $\widetilde{\R}$ to avoid confusion with the global Coulomb slice. Then we have \[I(l)^\lambda_\mu\simeq D(\R^{a_+^{\mu,\lambda}})_+\wedge D(\widetilde{\R}^{b_+^{\mu,\lambda}})_+\wedge (\R^{a_-^{\mu,\lambda}})^+\wedge (\widetilde{\R}^{b_-^{\mu,\lambda}})^+.\] Here, if $V$ is a vector space, $D(V)_+$ is the unit disk of $V$ union an extra base point and $V^+$ is the one point compactification of $V$. With all terminologies set up, we have \[(I_\frakq^\lambda,(I_\frakq^\lambda)^{\Z_2})\simeq (D(\R^{a_+^{\mu,\lambda}})_+\wedge D(\widetilde{\R}^{b_+^{\mu,\lambda}})_+\wedge (\R^{a_-^{\mu,\lambda}})^+\wedge (\widetilde{\R}^{b_-^{\mu,\lambda}})^+\wedge I^\mu_\frakq, D(\R^{a_+^{\mu,\lambda}})_+\wedge (\R^{a_-^{\mu,\lambda}})^+ \wedge (I_\frakq^\mu)^{\Z_2}).\] Note that $\wedge V^+$ has the same effect as $\Sigma^{V}$. 

Observe that for a $\Z_2$-space $(X,Y)$, changing it into $(D(\R)_+\wedge X,D(\R)_+\wedge Y)$, $(D(\widetilde{\R})_+\wedge X,Y)$, $(\Sigma^{\R }X,\Sigma^{\R} Y)$ and $(\Sigma^{\widetilde{\R}}X,Y)$ changes the connectivity of $(X,(X-Y)\cup *)$ by 0, 1, 1, 1, respectively. Therefore, \[n_{\lambda}=n_\mu+b_+^{\mu,\lambda}+a_-^{\mu,\lambda}+b_-^{\mu,\lambda}.\]

By definition, $\mathrm{dim}U_{-\lambda}^0+\mathrm{dim}V_{-\lambda}^0=\mathrm{dim}U_{-\mu}^0+\mathrm{dim}V_{-\mu}^0+a_-^{\mu,\lambda}+b_-^{\mu,\lambda}$, so \[n_{\lambda}-\mathrm{dim}U_{-\lambda}^0-\mathrm{dim}V_{-\lambda}^0=n_{\mu}-\mathrm{dim}U_{-\mu}^0-\mathrm{dim}V_{-\mu}^0+b_+^{\mu,\lambda}\to \infty, \text{ as } \lambda\to \infty.\] 

This conclude the estimation of $M_\lambda$.

As they remarked in \cite[p206-207]{Lidman2016TheEO}, the natural grading on $\widecheck{C}_{\lambda}$ is not the same as the gradings we have considered throughout this paper. To remedy this, we define $\widecheck{CMR}^{\lambda}$ by taking the Morse complex $(\widecheck{C}_{\lambda},\widecheck{\partial}_\lambda)$ and shifting the gradings down by $n^{R}(Y,\iota,\s,g)+\mathrm{dim}V_{-\lambda}^0+\mathrm{dim}U_{-\lambda}^0$, so that the stationary points have the grading $\mathrm{gr}^{\mathit{SWF}}$as  we considered in \ref{sub:Identification of grading}. Now we have \[\widetilde{H}^{\Z_2}_{i}(\mathit{SWF}_{\Z_2,\frakq}(Y,\iota,\s))\cong \widecheck{\mathit{HMR}}_i^\lambda(Y,\iota,\s,\frakq),\] for $i\in [-M_\lambda,M_\lambda]$. As a final step, we want to identify $ \widecheck{\mathit{HMR}}_i^\lambda(Y,\iota,\s,\frakq)$ with $ \widecheck{\mathit{HMR}}_i(Y,\iota,\s,\frakq)$. The chain complexes are the same in each fixed grading range $[-N,N]$ - as we have a grading-preserving bijection between stationary points from Subsection \ref{sub:Identification of grading} and a one-to-one correspondence between trajectories from Subsection \ref{sub:Stabilit}.  Then we have \[\widecheck{\mathit{HMR}}_i^\lambda(Y,\iota,\s,\frakq)\cong \widecheck{\mathit{HMR}}_i(Y,\iota,\s,\frakq),\] for $i\in [-N+1,N-1]$. Then Subsection \ref{sub:U-action} tells us that this isomorphism is actually one between $\F[v]$-modules.

Finally, for each $N$, we can take $\lambda=\lambda_i^\bullet$ large enough so that $M_\lambda>N$, as we shown above. Then, by combining all these discussions, we see that \[\widecheck{\mathit{HMR}}_{*}(Y,\iota,\s)\cong \widetilde{H}^{\Z_2}_{*}(\mathit{SWF}_{\Z_2}(Y,\iota,\s))\] in each grading. This concludes the proof. \end{proof}

Now we generalize the theorem as we promised in Remark \ref{rmk:generalize to nonhomology sphere}. We will adapt the proof above to show the following.
\begin{prop}\label{prop:generalize to non-homology sphere}
Let $(Y,\iota)$ be a real three-manifold and $\s$ be a compatible real $\mathrm{spin^c}$ structure. Suppose that $H^1(Y;\Z)^{-\iota^*}=0$. Then we have an isomorphism of relatively graded $\widetilde{H}_{\Z_2}^*(S^0;\F)\cong \F[v]$-modules \[\widecheck{\mathit{HMR}}_{*}(Y,\iota,\s)\cong \widetilde{H}^{\Z_2}_{*}(\mathit{SWF}_{\Z_2}(Y,\iota,\s);\F).\] We also have counterparts for the ``bar'', `` from'' and ``tilde'' version of $\mathit{HMR}$ as for real rational homology spheres.
\end{prop}

We first recall the reason why finite-dimensional approximation in \cite{Manolescu_2003} does not work for a general three manifold. When $b_1(Y)\ne 0$, $W$ is only a Hilbert bundle over the Picard torus. The key point here is that under the assumption of Proposition \ref{prop:generalize to non-homology sphere}, the fixed part $W^I$ of the global Coulomb slice is now a Hilbert space. Since the proof that occupy the previous sections of this paper mainly focuses on $W^I$ and does not rely heavily on the property of $W$, it almost works for this new setting. 

Note that any real $\mathrm{spin^c}$ structure $\s$ on such a real three-manifold must be torsion (in the sense that $c_1(\s)$ is torsion), since, by Hodge theory and Poincare duality, $H^1(Y;\Z)^{-\iota^*}=0$ implies $H^1(Y;\R)^{-\iota^*}=0$ and $H^2(Y;\R)^{-\iota^*}=0$. It was shown in \cite[Proposition 2.4]{miyazawa2023gaugetheoreticinvariantembedded} that for such a triple $(Y,\iota,\s)$, $\mathit{SWF}_{\Z_2}(Y,\iota,\s)$ is a well-defined.   
Real monopole Floer homology is originally defined for all closed real three-manifolds, so we have no need to pay extra attention to its well-definedness. When arguing for an alternative characterization of $v$-action, we used the fact that for a rational homology sphere, $\cG^I$ has only two components distinguished by whether it has constant $1$ or $-1$. This was guaranteed by the long exact sequence $0\to \Z_2\to\pi_0(\cG^{I,h})\to H^1(Y;\Z)^{-\iota^*}\to 0$ and the vanishing of its fourth term, which is still true now by assumption. There is no difficulty in defining $\mathit{HMR}^\circ$ in $W^I$ for these three-manifolds and in identifying them with the old definitions reviewed in Subsection \ref{sub:Real Monopole Floer homology} as we did in Section \ref{sec:Real Monopole Floer homology in global Coulomb slice}. 

Now $W$ is a Hilbert bundle instead of a Hilbert space. To make our lives easier, we introduce a subspace $W_0$. By the assumption that  $H^1(Y;\Z)^{-\iota^*}=0$, we know that there is still a unique invariant flat connection $A_0$. We let $A_0$ be the base connection, so that now the space of connections can be identified with the space of imaginary valued 1-forms. Pick a basis of harmonic 1-forms $\{a_1, \ldots,a_b\}$. By pairing with these, we can define a projection on $pr:W\to H^1(Y;\R)$ by taking the harmonic part of those 1-forms. We define $W_0$ to be $pr^{-1}(0)$. Unlike $W$, $W_0$ is no longer parametrized by the Picard torus, so it is a Hilbert space and the remaining ``gauge group'' action on it is $S^1$. Also, from $H^1(Y;\R)^{-\iota^*}=0$, we know that $W^I=W_0^I$, in particular $W^I\subset W_0$. Thus, when passing to the real part, we have no information loss.

During the proof of Proposition \ref{prop: Fredholm property of Q gc gamma}, we used the fact that $Y$ is a rational homology sphere to conclude $(\mathrm{im}d_{(0)})^{-\iota^*}=(\mathrm{ker}d_{(1)})^{-\iota^*}$. This is still true under the assumption $H^1(Y;\Z)^{-\iota^*}=0$.

To define perturbed real Seiberg-Witten homotopy type, we need the notion of a very compact map. Defining such a notion on $W$ requires extra care on the Hilbert bundle, but the characterization from \cite[Section 6]{Lidman2016TheEO} works on $W_0$. We can then use the fact that an equivariant very compact map on $W_0$ restricts to very compact map on $W^I$ to see that $\mathit{SWF}_{\Z_{2},\frakq}(Y,\iota,\s)$ is well-defined and is isomorphic to $\mathit{SWF}_{\Z_2}(Y,\iota,\s)$. In Section \ref{sec:Relating finite and infinite dimensional Morse homologies}, we claimed that we can just use estimates in \cite[Section 6.3]{Lidman2016TheEO}, since for rational homology spheres $W^I\subset W$, the analytic properties can be inherited by subspace. These estimations are needed in the identification of stationary points in Subsection \ref{sub:Convergence of stationary points}, but this strategy no longer works. Nevertheless, we note that the proof in \cite[Section 6.3]{Lidman2016TheEO} can be directly repeated on $W^I$ using results from \cite{Khandhawit_2018} as Miyazawa did in \cite{miyazawa2023gaugetheoreticinvariantembedded}.

The convergence of stationary points, identification of grading and the proof of cut-off flow is a Morse-Smale equivariant quasi-gradient works verbatim in this case. The construction of self-diffeomorphism $W^I\to W^I$ also makes sense. 

Now, to finish the proof, it remains to consider convergence of trajectories. Most of the arguments from Subsection \ref{sub:Convergence of trajectories} still work, but when proving Proposition \ref{prop:upstair covergence with energy control}, we embedded $W^I$ into $W$ and made use of \cite[Lemma 12.2.4]{Lidman2016TheEO}, whose proof relies on taking derivative of the $S^1$-gauge action. To remedy this, we observe that the argument from \cite[Section 3]{Khandhawit_2018} tells us that finite energy trajectories of $l+p^\lambda c_\frakq$ in $W_0$ share the same property as those in $W(\text{rational  homology  sphere})$, so we can use the estimates from \cite[Section 12]{Lidman2016TheEO} in $W_0$ using the $S^1$-action and then see that the energy control is still valid for the current $W^I$.

After these preparation, we can argue as in previous subsections to conclude the proof of Proposition \ref{prop:generalize to non-homology sphere}.

\subsection{Fr\o yshov-type invariants}
In \cite{Konno2024}, they considered a Fr\o yshov-type invariant and proved a Fr\o yshov-type inequality for real Seiberg-Witten homotopy type. In \cite{li2022triangle}, Li defined counterparts in real monopole Floer homology for the double branched covers of links with non-zero determinant and also proved a Fr\o yshov-type inequality. In this subsection, we review their definitions and prove Proposition \ref{prop:Froyshov invariant identification}. 

On the Floer homotopy type side, they introduced a more general notion called a space of type $(G,H)-\mathit{SWF}$. Here, we only care about the case $(G,H)=(\Z_2,\Z_2)$.

\begin{definition}(\cite[Definition 3.1]{Konno2024})
    Let $G$ be a group and $H$ be a subgroup of $G$. Let $\scrV$ be a countable direct sum of a fixed one-dimensional real representation of $G$. Let $X$ be a pointed $G$-$CW$ complex. We call $X$ \emph{a space of type $(G,H)-\mathit{SWF}$}, if \begin{itemize}
        \item $X^H$ is $G$-homotopy equivalent to $V^+$, where $V$ is a finite-dimensional subrepresentation of $\scrV$.
        \item $H$ acts freely on $X-X^H$.
    \end{itemize}
\end{definition}

For every space $X$ of type $(\Z_2,\Z_2)-\mathit{SWF}$, we can associate to it a numerical invariant as follows:
\[d(X)=\min \{r| \exists x, 0\ne x\in \widetilde{H}^{\Z_2}_{r}(X;\F) \cap (\bigcap_{l\ge 0}) \mathrm{im}(v^l)\}.\]
Here, $v$ comes from the action of $H_{\Z_2}^*(S^0;\F)=\F[v]$ on $\widetilde{H}^{\Z_2}_*(X;\F)$. We modify the original definition a little since we want to use homology instead of cohomology. Up to (de)suspension, $\mathit{SWF}_{\Z_2}(Y,\iota,\s)$ is a space of type $(\Z_2,\Z_2)-\mathit{SWF}$ using trivial representation, so we can associate to it an invariant $d(Y,\iota,\s)$ by taking $d$ of the Conley index and shifting the grading according to the suspensions in the formula. 

This invariant satisfies the following.
\begin{prop}(\cite[Theorem 3.23]{Konno2024})\label{prop:Froyshov inequality from SWFR}
    Let $(Y_0,\iota_0,\s_0)$, $(Y_1,\iota_1,\s_1)$ be $\mathrm{spin^c}$ real rational homology three spheres. Suppose that $(W,\iota,\s)$ is a smooth compact oriented real cobordism from $(Y_0,\iota_0,\s_0)$ to $(Y_1,\iota_1,\s_1)$ with $b_1(W)=0$. Assume further that $b^+(W)=b^+_{\iota^*}(W)$. Then we have \[d(Y_0,\iota_0,\s_0)+\frac{c_1(\s)^2-\sigma(W)}{8}\le d(Y_1,\iota_1,\s_1).\]
\end{prop}

On $\widecheck{\mathit{HMR}}$, Li specialized to $Y=\Sigma(K)$, the double branched cover of a link $K\subset S^3$ with $\mathrm{det}(K)\ne 0$. From now on, we assume $Y$ is of this form. There is a long exact sequence connecting different types of real monopole Floer homology: 
\[\ldots \xrightarrow{p_*} \overline{\mathit{HMR}_*}(Y,\iota,\s) \xrightarrow{i_*} \widecheck{\mathit{HMR}}_*(Y,\iota,\s) \xrightarrow{j_*} \widehat{\mathit{HMR}}_*(Y,\iota,\s) \xrightarrow{p_*} \overline{\mathit{HMR}}_{*-1}(Y,\iota,\s) \xrightarrow{i_*}\ldots\] 

\begin{definition} (\cite[Definition 3.2]{li2022triangle})
    Let $K$ be link as above and $\s$ be a real $\mathrm{spin^c}$ structure on its double branched cover. The \emph{real Fr\o yshov-invariant} $h_{R}(K,\s)$ is the number with the property that the element with the lowest absolute grading in $i_*(\overline{\mathit{HMR}_*}(Y,\iota,\s))\subset \widecheck{\mathit{HMR}}_*(Y,\iota,\s)$ has $\mathrm{gr}^\Q=-h_{R}(K,\s)$.
 \end{definition}

Li showed the following:
\begin{prop}(\cite[Proposition 3.3]{li2022triangle})
Let $K_\pm$ be two links in $S^3$ with nonzero determinant and $S:K_-\to K_+$ be a connected cobordism. Let $\s$ be a real $\mathrm{spin^c}$ structure on $\Sigma(S)$ that restricts to $\s_{\pm}$ on its two boundaries. Suppose that $\Sigma(S)$ is negative definite, i.e.,\[b^+(\Sigma(S))=b_1(S)-b_0(S)+\sigma(K_+)-\sigma(K_-)-\frac{S\cdot S}{2}=0.\] Then \[h_R(K_-,\s_-)\ge h_R(K_+,\s_+)+\frac{c^2_1(\s)-\sigma(\Sigma(S))}{8}.\]
\end{prop}

Now, we are ready to state and prove a refined version of Proposition \ref{prop:Froyshov invariant identification}.
\begin{prop}
    Let $Y$ be the double branched cover of some link $K\subset S^3$ with $\mathrm{det}(K)\ne 0$. equipped with the canonical real structure $\iota$ and a compatible real $\mathrm{spin^c}$ structure, then $d(Y,\iota,\s)=-h_R(K,\s)$. Moreover, the two Fr\o yshov-type inequalities coincide when $Y$ is specialized to this form and $W$ is taken to be a double branched cover over some knot cobordism. 
\end{prop}
\begin{proof}
    The first claim follows from the observation that \begin{itemize}
        \item both $d(Y,\iota,\s)$ and $-h_R(K,\s)$ can be characterized as the starting grading of the infinite $v$-tower in the corresponding homology theory;
        \item we have shown that for $Y$ as in the description, $\widecheck{\mathit{HMR}}_*(Y,\iota,\s)\cong\widetilde{H}^{\Z_2}_*(\mathit{SWF}_{\Z_2}(Y,\iota,\s))$ as absolutely graded $\F[v]$-modules.
    \end{itemize} 

    For the second claim, the computation at the beginning of \cite[Section 3]{li2022triangle} tells us that, for $W=\Sigma(S)$,  $b^1(W)=0$ and $b^+_{\iota^*}(W)=0$, so \[b^+(W)=0 \Longleftrightarrow b^+(W)=b^+_{\iota^*}(W).\] This identifies the assumption in two propositions, the expression of the two inequalities are obviously the same.
\end{proof}

\begin{remark}\begin{enumerate}
    \item On the existence of $\mathrm{gr}^\Q$ on $\widecheck{\mathit{HMR}}$: The grading formula in Subsection \ref{subsub:Gradings} makes sense when $(Y,\iota,\s)$ has a real $\mathrm{spin^c}$ bound. It is obvious that double branched covers over knots and links fit into this assumption. Moreover, using the contact surgery description for real contact manifolds in \cite{Cengiz_2022}, we know that such a bound exists whenever the real $\mathrm{spin^c}$ structure is torsion and supports a real contact structure. The main theorem in \cite{Cengiz_2022} states that every closed real three-manifold supports a real contact structure, which tells us whenever $H^1(Y;\Z)^{-\iota_*}=0$, absolute grading is well-defined on $\mathit{HMR}(Y,\iota,\s)$ for at least one $\s\in \mathrm{Spin^c_R}(Y,\iota)$.
    \item On the definition of Fr\o yshov-invariant: First, we mention that $h_R$ is always well-defined whenever the grading admits an absolute lift, so we have no need to restrict ourselves to branchced cover over links with non-zero determinants. We can also prove a more general version of the Fr\o yshov-type inequality using the same argument as in \cite[Proposition 3.3]{li2022triangle}.  Furthermore, as we have seen at the end of Section \ref{sub:Main thm}, we do not need to assume that $Y$ is a rational homology sphere, $d$ is well-defined whenever $H^1(Y,\Z)^{-\iota^*}$ is zero and for $h_R$ to be well-defined, we need an extra (not very restrictive) assumption on a real $\mathrm{spin^c}$ bound. We still have $d=-h_R$, when they are both well-defined and an analogue of Proposition \ref{prop:Froyshov inequality from SWFR} can be proved as well.
    \item About the absolute grading: One may have noted that the absolute grading on $\widetilde{H}^{\Z_2}_*(\mathit{SWF}_{\Z_2}(Y,\iota,\s))$ is always well-defined, regardless of whether there is a real $\mathrm{spin^c}$ bound or not, so we can actually use the isomorphism as relative graded module to induce an absolute grading on $\widecheck{\mathit{HMR}}_*(Y,\iota,\s)$.
\end{enumerate}
    
\end{remark}

\subsection{Smith-type inequalities}\label{sub: Smith type inequalities}
In this subsection, we introduce some new concepts and prove Theorem \ref{thm:Smith inequality} and \ref{thm:inequality in reduced version}. We will also consider some direct application of it.

Recall that we have defined a $v$-action on $\widecheck{\mathit{HMR}}_*(Y,\iota,\s)$ in Subsection \ref{subsub:Gradings} for $v$, the generator of $\widetilde{H}^{*}_{\Z_2}(S^0;\F)\cong\F[v]$. This comes from a chain map $v:\widecheck{CMR}_*(Y,\iota,\s)\to \widecheck{CMR}_{*-1}(Y,\iota,\s)$. The idea of this module structure originates from the $\widetilde{H}^*_{S^1}(S^0;\Z)\cong \Z[U]$-module structure on $\widecheck{\mathit{HM}}(Y,\underline{\s})$ which comes from a chain map $U:\widecheck{CM}_{*}(Y,\underline{\s})\to \widecheck{CM}_{*-2}(Y,\underline{\s})$. In Heegaard Floer theory, $\widehat{\mathit{HF}}(Y,\underline{\s})$ is the simplest version that counts pseudo-holomorphic disks that do not cross the basepoint. Alternatively, $\widehat{CF}(Y,\underline{\s})$ can be characterized as the mapping cone of $U:CF^+_{*}(Y,\underline{\s})\to CF^+_{*-2}(Y,\underline{\s})$. Motivated by this, Bloom defined $\widetilde{CM}_*(Y,\underline{\s})$ as the mapping cone of this $U$ action on $\widecheck{CM}_*(Y,\underline{\s})$ in \cite{BLOOM20113216} and let $\widetilde{\mathit{HM}}_*(Y,\underline{\s})$ be its homology. Li introduced the real analogue by defining $\widetilde{CMR}(Y,\iota,\s)=\mathrm{Cone}(v:\widecheck{CMR}(Y,\iota,\s)\to \widecheck{CMR}(Y,\iota,\s))$ and taking its homology $\widetilde{\mathit{HMR}}_*(Y,\iota,\s)=H_*(\widetilde{CMR}(Y,\iota,\s))$. 

\begin{lem}\label{lem:equivariant homology and usual homology}
Let $X$ be a based finite $\Z_2$-CW complex. Then we have the following graded isomorphism \[\widetilde{H}^{n}(X;\F)\cong H^n([ H_{\Z_2}^*(X;\F)]_1 \xrightarrow{v} H_{\Z_2}^{*+1}(X;\F)).\]
Here $[C]_1$ means that the degree of $C$ is shifted up by one.
\end{lem}
\begin{proof}
    This is just \cite[Lemma 14.0.1]{Lidman2016TheEO} with $S^1$ replaced by $\Z_2$ and $\Z$ coefficients replaced by $\F$ coefficients. The proof for that lemma works for us after replacing the Euler class with the first Stiefel Whitney class and changing the grading shift in the long exact sequence of degree two to degree one. 
\end{proof}

This lemma (taking dual to homology theory) and Theorem \ref{thm:main theorem} imply the following isomorphism.

\begin{cor}\label{cor:tilde version isom}
    Let $(Y,\iota,\s)$ be as in Theorem \ref{thm:main theorem}, then we have an isomorphism of graded $\F$-vector spaces \[\widetilde{\mathit{HMR}}_*(Y,\iota,\s)\cong \widetilde{H}_*(\mathit{SWF}_{\Z_2}(Y,\iota,\s)).\]
\end{cor}
This proves \cite[Conjecture 1.4]{li2024realmonopolesspectralsequence}.

Tracing the definition in Subsection \ref{sub:Real Seiberg-Witten homotopy type} (from \cite{Konno2024})and the one in \cite{Manolescu_2003}, we see that \[\mathit{SWF}_{\Z_2}(Y,\iota,\s)=\mathit{SWF}(Y,\underline{\s})^I,\] where we use $\underline{\s}$ to denote the underlying $\mathrm{spin^c}$ structure of the real $\mathrm{spin^c}$ structure $\s$ and $I$ is the involution we defined in Subsection \ref{subsub: Configuration space and Coulomb slices} which contains the information from $\iota$ and $\s$. Recall that we have the following classical Smith inequality; see \cite{Smith1938} or \cite{Floyd52}.
\begin{lem}\label{lem:Smith inequality}
    Suppose that a group $G$ of order $p^n$ ($p$ is some prime number) acts on a compact topological space $X$ with finite-dimensional $H_*(X;\F_p)$. Let $X^G$ be the fixed-point set. Then we have an inequality of Betti numbers: \[\sum_{i} \mathrm{dim}H_i(X^G;\F_p)\le \sum_{i} \mathrm{dim}H_i(X;\F_p). \]
\end{lem}

Applying this to $X=\mathit{SWF}(Y,\underline{\s})$, $G=\left \langle I \right \rangle\cong\Z_2$, we see that \[\sum_{i} \mathrm{dim}\widetilde{H}_i(\mathit{SWF}_{\Z_2}(Y,\iota,\s);\F)\le \sum_{i} \mathrm{dim}\widetilde{H}_i(\mathit{SWF}(Y,\underline{\s});\F).\]

Using the results from \cite{Lidman2016TheEO} with $\F$-coefficients, we know that \[\widetilde{H}_*(\mathit{SWF}(Y,\underline{\s});\F)\cong \widetilde{\mathit{HM}}_*(Y,\underline{\s}).\]

Combining this with Corollary \ref{cor:tilde version isom}, we have proved the first part of Theorem \ref{thm:Smith inequality}.

The notion of L-space was first introduced by Ozsváth and Szabó in \cite{OZSVATHlensspacesurgery} to characterize rational homology spheres that cannot be distinguished from lens spaces by Heegaard Floer homology. More precisely, a rational homology sphere is called an \emph{$\F$-L-space} if for each $\mathrm{spin^c}$ structure $\underline{\s}$, $\mathrm{dim} \widehat{\mathit{HF}}(Y,\underline{\s})=1$ as a $\F$-vector space. This is equivalent to $\mathit{HF}^+(Y,\underline{\s})\cong \F[U]$ for each $\mathrm{spin^c}$ structure $\underline{\s}$ by an elementary algebraic argument. In a series of papers \cite{HF=HM1,HF=HM2,HF=HM3,HF=HM4,HF=HM5}, they proved an isomorphism between usual monopole Floer homology and Heegaard Floer homology. So L-spaces can be characterized alternatively by having simplest $\widetilde{\mathit{HM}}$ or $\widecheck{\mathit{HM}}$. Originally, Ozsváth and Szabó defined L-space using Heegaard Floer theory over $\Z$. That notion is strongest in the sense that an L-space would be $\mathbb{K}$-L-space for any field $\mathbb{K}$. In all know examples, no torsion appears in Heegaard Floer theory over $\Z$, so, conjecturally, all these notions of L-space are the same.

Generalizing this notion, we say a real rational homology three-sphere $(Y,\iota)$ is a \emph{real L-space} if for each real $\mathrm{spin^c}$ structure $\s$, $\mathrm{dim} \widetilde{\mathit{HMR}}(Y,\iota,\s)=1$.

Let $Y$ be an L-space. Then for any real structure $\iota$ and a compatible real $\mathrm{spin^c}$ structure $\s$, we know that $\widetilde{\mathit{HM}}(Y,\underline{\s})$ is one dimensional with $\F$ coefficient. Then, the first part of Theorem \ref{thm:Smith inequality} tells us that $\widetilde{\mathit{HMR}}(Y,\iota,\s)$ is also one dimensional. This concludes the proof of this theorem.

\begin{example} 
It was shown in \cite{OZSVATHdoublebranchedcover} that the double branched cover of any quasi-alternating link is an L-space. Quasi-alternating link is a generalization of the classical notion of an alternating link. For its precise definition, see \cite[Definition 3.1]{OZSVATHdoublebranchedcover}. The covering transformation endows $\Sigma(K)$ with a natural real structure $\iota_K$. Thus, Theorem \ref{thm:Smith inequality} allows us to conclude that $(\Sigma(K),\iota_K)$ is a real L-space whenever $K$ is quasi-alternating. Note that our result is actually real structure independent, i.e., for any other real structure $\iota'$ on $\Sigma(K)$, $(\Sigma(K),\iota')$ is also a real L-space. 

In addition, Motegi showed in \cite{MOTEGI2017172} that for links in the infinite family of non-quasi-alternating links $\{L_{m,n}\}_{n>m\ge 2}$ constructed in \cite{greene2009homologicallythinnonquasialternatinglinks} by Greene, each has $\Sigma(L_{m,n})$ an L-space. So we can also conclude for this family that $(\Sigma(L_{m,n}),\iota_{L_{m,n}})$ and more generally $(\Sigma(L_{m,n}),\iota')$ are real L-spaces.
\end{example}

\begin{example} A knot $K$ in $S^3$ is called an \emph{L-space knot} if the result of some non-trivial positive Dehn surgery on $K$ is an L-space(\cite{OZSVATHlensspacesurgery}). In \cite{Hedden2010notionsofpositivity,Ni07fibredknot,Ozsvath2011rationalsurgery,OZSVATHlensspacesurgery}, they showed in various ways that if $K$ is an L-space knot, then for any $r\in \Q_{>0}$ satisfying $r\ge 2g(K)-1$, $S^3_{r}(K)$ is an L-space. It was conjectured that all L-space knots are strongly invertible, see \cite{LidmanMoore2013} or \cite{WATSON2017915}. Although this was disproved by Baker and Luecke in \cite{Baker_2020}, most known examples of L-space knots are strongly invertible. 

On the other hand, we know from \cite[Section 5.1]{Daicorks2023} that we can perform equivariant surgery with any coefficient on an equivariant knot. In particular, this tells us that if $K$ is a strongly invertible knot in $S^3$, then any Dehn surgery $S^3_{r}(K)$ has a real structure $\iota_{K,r}$ inherited from $S^3$. 

\end{example}

\begin{remark}
In \cite{guth2025realheegaardfloerhomology}, Guth and Manolescu constructed real Heegaard Floer homology $\widehat{\mathit{HFR}}$ and its counterparts for real three-manifolds. Thus, we can also consider a Heegaard Floer version of real L-space to be those real rational homology spheres with the simplest $\widehat{\mathit{HFR}}$. In \cite[Remark 6.6]{guth2025realheegaardfloerhomology}, they conjectured that $\widehat{\mathit{HFR}}(\Sigma(K),\iota,\s)\cong \F$ when $Y$ is an L-space. Their conjecture was proved by Hendricks in \cite[Corollary 1.3]{hendricks2025noterealheegaardfloer}, and what we have proved is the monopole version of their conjecture. %It would be interesting to compare our approaches or more generally, the two theories $\mathit{HFR}^\circ$ and $\mathit{HMR}^\circ$. We conjecture that they are actually isomorphic.
\end{remark}

Next, we move to Theorem \ref{thm:inequality in reduced version}, which is analogous to \cite[Theorem 1.4]{LM2016coveringspaces}. The proof here needs a little more algebraic topology, since for $\mathit{HM}$, we consider $S^1$-equivariant theory but for $\mathit{HMR}$, we consider $\Z_2$-equivariant theory. 

Before proving the theorem, we first recall and introduce some definitions. In \cite{OS04closedthreemanifold}, they introduced the reduced version of Heegaard Floer homology, which can be described as \[\mathit{HF}_{\mathrm{red}}(Y,\underline{\s})=\mathit{HF}^+(Y,\underline{\s})/ U^N \mathit{HF}^+(Y,\underline{\s}),\] for $N\gg 0$. It was shown there that this reduced group is independent of the choice of a sufficiently large $N$.  Similarly, we can define \[\mathit{HM}_{\mathrm{red}}(Y,\underline{\s})=\widecheck{\mathit{HM}}(Y,\underline{\s})/ U^N \widecheck{\mathit{HM}}(Y,\underline{\s}),\] \[\mathit{HMR}_{\mathrm{red}}(Y,\iota,\s)=\widecheck{\mathit{HMR}}(Y,\iota,\s)/ v^N \widecheck{\mathit{HMR}}(Y,\iota,\s),\]  for sufficiently large $N$ and show that it is independent of $N$ when it is large enough. Using the isomorphism between Heegaard Floer theory and monopole Floer theory from \cite{HF=HM1} and its sequel papers, we also know that $\mathit{HM}_{\mathrm{red}}(Y,\underline{\s})\cong \mathit{HF}_{\mathrm{red}}(Y,\underline{\s})$. 

\begin{proof}(proof of Theorem \ref{thm:inequality in reduced version})

For notational convenience, let $X=\mathit{SWF}(Y,\underline{\s})$, then $X^I=\mathit{SWF}_{\Z_2}(Y,\iota,\s)$ for $I$ the involution defined in Subsection \ref{subsub: Configuration space and Coulomb slices}. Throughout the proof, we will use $\F$ as coefficient for various homology theories.

Using Theorem \ref{thm:main theorem} and the corresponding isomorphism in usual Seiberg-Witten Floer theory from \cite{Lidman2016TheEO}, we have that \[\mathit{HMR}_{\mathrm{red}}(Y,\iota,\s)=\widetilde{H}^{\Z_2}_*(X^I)/ v^N \widetilde{H}^{\Z_2}_*(X^I) \text{ and } \mathit{HM}_{\mathrm{red}}(Y,\underline{\s})=\widetilde{H}^{S^1}_*(X)/ U^N \widetilde{H}^{S^1}_*(X) ,\] for $N$ sufficiently large.
  
Now we introduce \[\mathit{HM}_{\mathrm{red}}^{\Z_2}(Y,\underline{\s})=\widetilde{H}^{\Z_2}_*(X)/ v^N \widetilde{H}^{\Z_2}_*(X),\] which will bridge the gap between $ \mathit{HM}_{\mathrm{red}}(Y,\underline{\s})$ and $\mathit{HMR}_{\mathrm{red}}(Y,\iota,\s)$.

We first show that $ \mathrm{dim} \mathit{HMR}_{\mathrm{red}}(Y,\iota,\s)\le \mathrm{dim}\mathit{HM}_{\mathrm{red}}^{\Z_2}(Y,\underline{\s})$ following the argument for \cite[Theorem 1.4]{LM2016coveringspaces}. For a representation $V$, which is also a vector space, let $S(V)_+$, $D(V)_+$ denote the unit sphere and unit disk in $V$ with an extra base point added, respectively. We will consider the long exact sequence of $\Z_2$-equivariant Borel homology associated to the pair $(X\wedge D(\R^N)_+, X\wedge S(\R^N)_+)$. Before doing so, we observe that \begin{itemize}
    \item The space $X\wedge S(\R^N)_+$ has a free $\Z_2$-action away from the base point, so its (reduced) Borel homology is isomorphic to the ordinary (reduced) homology of the quotient $\widetilde{H}_*(X\wedge_{\Z_2} S(\R^N)_+)$;
    \item The space $X\wedge D(\R^N)_+$ is $\Z_2$-equivalent to $X$;
    \item Smashing with $(D(\R^N)_+/S(\R^N)_+)\sim (\R^N)^+$ preserves Borel homology up to some degree shift by $N$.
\end{itemize}

Thus, the long exact sequence can be written as \[\ldots\widetilde{H}_*(X\wedge_{\Z_2} S(\R^N)_+)\to \widetilde{H}^{\Z_2}_*(X)\to \widetilde{H}_{*-N}^{\Z_2}(X)\to \ldots\]

The map $\widetilde{H}^{\Z_2}_*(X)\to \widetilde{H}_{*-N}^{\Z_2}(X)$ comes from the composition \[X\hookrightarrow X\wedge D(\R^N)_+\to X\wedge(D(\R^N)_+/S(\R^N)_+)\to X \wedge (\R^N)^+\cong \Sigma^{N\R}X. \] Therefore, on the homology level, the map is given by multiplication by the equivariant mod 2 Euler class of $N\R$, which is $v^N\in H^{N}_{\Z_2}(pt)$.

For $N$ large enough, multiplication by $v^N$ on $\widetilde{H}^{\Z_2}(X)$ has kernel of dimension $N+\mathrm{dim} \mathit{HM}_{\mathrm{red}}^{\Z_2}(Y,\underline{\s})$ and cokernel isomorphic to $\mathit{HM}_{\mathrm{red}}^{\Z_2}(Y,\underline{\s})$. So we have an equality \[\mathrm{dim} \widetilde{H}^{\Z_2}_*(X\wedge_{\Z_2} S(\R^N)_+)=N+ 2\mathrm{dim}\mathit{HM}^{\Z_2}_{\mathrm{red}}(Y,\underline{\s}).\]
With $X^I$ in place of $X$, we have \[\mathrm{dim} \widetilde{H}^{\Z_2}_*(X^I\wedge_{\Z_2} S(\R^N)_+)=N+ 2\mathrm{dim}\mathit{HMR}_{\mathrm{red}}(Y,\iota,\s).\]

Note that the $\Z_2$ fixed-point set of $X\wedge_{\Z_2} S(\R^N)_+$ is $X^I\wedge_{\Z_2} S(\R^N)_+$. Now applying the usual Smith inequality, Lemma \ref{lem:Smith inequality} concludes the proof. 

Next, we show that \[\mathrm{dim} \mathit{HM}_{\mathrm{red}}^{\Z_2}(Y,\underline{\s}) \le 2 \mathrm{dim} \mathit{HM}_{\mathrm{red}}(Y,\underline{\s}),\] which is equivalent to \[\mathrm {dim} \widetilde{H}^{\Z_2}_*(X)/ v^N \widetilde{H}^{\Z_2}_*(X) \le
2 \mathrm {dim} \widetilde{H}^{S^1}_*(X)/ U^N \widetilde{H}^{S^1}_*(X). \]

To see this, we recall the construction of Borel homology: For a $G$ space $X$, the $G$-equivariant Borel homology is defined as the usual homology of the homotopy quotient \[X_{G}= X\times _{G} EG=X\times EG/ (x,y)\sim (gx,g^{-1}y), \forall g\in G.\] In our case, we have $X_{\Z_2}= X\times S^\infty/ (x,y)\sim(-x,-y)$ and $X_{S^1}=X\times S^\infty /(x,y)\sim (e^{i\theta}x, e^{-i\theta}y)$, so there is a natural quotient map $q: X_{\Z_2}\to X_{S^1}$, which is a fibration with $S^1$ fiber. The following argument is suggested by an anonymous referee. 
In this way, we get two $S^1$-bundles $X_{\Z_2}$ and $X\times S^{\infty} $ over $X_{S^1}$. The Euler class of the first is twice that of the second one, so in particular, it vanishes modulo two. Thus, the Gysin sequence associated to $q$ split into short exact sequences \[0\to \widetilde{H}_{k-1}^{S^1}(X)\xrightarrow{q^*} \widetilde{H}_{k}^{\Z_2}(X) \xrightarrow{q_*}\widetilde{H}_{k}^{S^1}(X) \to 0. \] It follows that \[ \frac{\widetilde{H}_{k-1}^{S^1}(X)}{\mathrm{im}(U^N)} \xrightarrow{q^*}\frac{\widetilde{H}_{k}^{\Z_2}(X)}{\mathrm{im}(v^{2N})} \xrightarrow{q_*} \frac{\widetilde{H}_{k-1}^{S^1}(X;\F)}{\mathrm{im}(U^N)} \] is exact in the middle for $N$ sufficiently large. 
These groups are only nonzero for finitely many $k$, summing over all $k$, we get the desired inequality \[\mathrm {dim} \widetilde{H}^{\Z_2}_*(X)/ v^N \widetilde{H}^{\Z_2}_*(X) \le
2 \mathrm {dim} \widetilde{H}^{S^1}_*(X)/ U^N \widetilde{H}^{S^1}_*(X) \], which concludes the proof.

\end{proof}

Using the identification of the usual monopole and Heegaard Floer homology, we have the following direct corollary.
\begin{cor}
    Under the assumption of Theorem \ref{thm:Smith inequality}, we have \[\mathrm{dim} \mathit{HMR}_{\mathrm{red}}(Y,\iota,\s) \le 2 \mathrm{dim} \mathit{HF}_{\mathrm{red}} (Y,\underline{\s}),\] and in particular when $\mathit{HF}_{\mathrm{red}} (Y,\underline{\s})=0$, $\mathit{HMR}_{\mathrm{red}}(Y,\iota,\s)$ must also be zero.
\end{cor}

\begin{example}
   Seiberg-Witten theory is in general hard to compute, we illustrate Theorem \ref{thm:inequality in reduced version} using the only non-trivial example that the author could find. Basing on the calculation in \cite{MOY97}, Li computes in \cite[Section 14.6]{li2022monopolefloerhomologyreal} real monopole group of several families of Brieskorn spheres $\Sigma(p,q,r)$, equipped the real structure as double branched cover over Montesinos knots $k(p,q,r)$ and its unique real $\mathrm{spin}$ structure. These examples are trivial in the sense that the $I$-action on Seiberg-Witten moduli spaces is trivial, so we actually have a dimension equality \[\mathrm{dim} \mathit{HMR}_{\mathrm{red}}(Y,\iota,\s)= \mathrm{dim} \mathit{HM}_{\mathrm{red}} (Y,\underline{\s})= \mathrm{dim} \mathit{HF}_{\mathrm{red}} (Y,\underline{\s}).\] 
   
   It would be interesting to see whether we can improve the constant 2 that appears in Theorem \ref{thm:inequality in reduced version} or is there any real $\mathrm{spin^c}$ rational homology sphere $(Y,\iota,\s)$ with $\mathrm{dim} \mathit{HM}_{\mathrm{red}} (Y,\underline{\s})<\mathrm{dim} \mathit{HMR}_{\mathrm{red}}(Y,\iota,\s) \le 2 \mathrm{dim} \mathit{HM}_{\mathrm{red}}(Y,\underline{\s})$.
\end{example}

\bibliographystyle{plain}
\bibliography{bibliography}

\end{document}